\documentclass[12pt,,reqno]{amsart}
\usepackage{amsrefs}
\usepackage{mathrsfs}
\usepackage{amsfonts}
\usepackage[centertags]{amsmath}
\usepackage{amssymb}
\usepackage{amsthm}
\usepackage{graphicx}
\usepackage{caption}
\usepackage{enumerate}
\usepackage[textwidth=16cm, hmarginratio=1:1]{geometry}
\usepackage{appendix}
\usepackage[colorlinks]{hyperref}
\usepackage{color}
\usepackage[all]{xy}
\usepackage{float}
\usepackage{amsmath,amsfonts,amssymb,graphicx}
\usepackage{longtable}
\usepackage{resizegather}
\usepackage{framed}
\usepackage{tikz}
\usepackage{multirow}
\hypersetup{
bookmarksnumbered,
pdfstartview={FitH},
breaklinks=true,
linkcolor=blue,
urlcolor=blue,
citecolor=blue,
bookmarksdepth=2
}
\allowdisplaybreaks[4]

\newtheorem{theorem}{Theorem}[section]
\newtheorem{corollary}[theorem]{Corollary}
\newtheorem{lemma}[theorem]{Lemma}
\newtheorem{proposition}[theorem]{Proposition}

\newtheorem{remark}[theorem]{Remark}
\newtheorem{example}[theorem]{Example}

\numberwithin{equation}{section}

\usetikzlibrary{arrows,automata,positioning}
\newcommand{\midarrow}{\tikz \draw[thin,-angle 45] (0,0) -- +(.25,0);}
\newcommand{\scri}[1]{\text{\scriptsize{#1}}}

\allowdisplaybreaks[4]

\begin{document}

\title[A combinatorial model for $q$-characters of fundamental modules]{A combinatorial model for $q$-characters of fundamental modules of type $D_{n}$}
\author{Jun Tong}
\address{Jun Tong: School of Mathematics and Statistics, Lanzhou University, Lanzhou 730000, P. R. China.}
\email{tongj18@lzu.edu.cn}

\author{Bing Duan}
\address{Bing Duan: School of Mathematics and Statistics, Lanzhou University, Lanzhou 730000, P. R. China.}
\email{duanbing@lzu.edu.cn}

\author{Yan-Feng Luo$^\dag$}
\address{Yan-Feng Luo: School of Mathematics and Statistics, Lanzhou University, Lanzhou 730000, P. R. China.}
\email{luoyf@lzu.edu.cn}
\thanks{$\dag$ Corresponding author}

\date{}

\maketitle

\begin{abstract}

In this paper, we introduce a combinatorial path model of representation of the quantum affine algebra of type $D_n$, inspired by Mukhin and Young's combinatorial path models of representations of the quantum affine algebras of types $A_n$ and $B_n$. In particular, we give a combinatorial formula for $q$-characters of fundamental modules of type $D_{n}$ by assigning each path to a monomial or binomial. By counting our paths, a new expression on dimensions of fundamental modules of type $D_n$ is obtained.

\hspace{0.15cm}

\noindent
{\bf Keywords}: Quantum affine algebras; Fundamental modules; $q$-characters; Combinatorial path models; Screening operators

\hspace{0.15cm}

\noindent
{\bf 2020 Mathematics Subject Classification}: 17B37
\end{abstract}

\section{Introduction}\label{Introduction}

Let $\mathfrak{g}$ be a simple Lie algebra over $\mathbb{C}$, and $I$ the vertex set of the Dynkin diagram of $\mathfrak{g}$. Let $\widehat{\mathfrak{g}}$ be the corresponding untwisted affine Kac-Moody algebra, and $U_q (\widehat{\mathfrak{g}})$ its quantum affine algebra with quantum parameter $q\in \mathbb{C}^{\times}$ not a root of unity. Denote by $\mathscr{C}$ the category of finite-dimensional $U_q (\widehat{\mathfrak{g}})$-modules. Every simple module in $\mathscr{C}$ is parameterized by its highest $\ell$-weight monomial \cite{CP91,CP94,FR98}.

In recent decades, the study of the category $\mathscr{C}$ has attracted much attention of many researchers and scholars from different perspectives, for examples, analytic \cite{BR90,KOS95}, algebraic \cite{Bit21a,Bit21b,C95,CP91,CP94,CMY13,FR98,FM01,HL10,HL16,LQ17,MY12b,TDL23,ZDLL16}, combinatoric \cite{GDL22,MY12a,Nak03,NN07}, and geometric \cite{Nak01,Nak04,VV02}.

The concept of $q$-characters was introduced by Frenkel and Reshetikhin \cite{FR98}. The $q$-character map is defined as an injective ring homomorphism from the Grothendieck ring $\mathcal{K}_0(\mathscr{C})$ of $\mathscr{C}$ to the ring $\mathbb{Z}[Y^{\pm1}_{i,a}| i\in I,a\in \mathbb{C}^{\times}]$ of Laurent polynomials in the infinitely formal variables $(Y_{i,a})_{i\in I,a\in \mathbb{C}^{\times}}$. Similar to Cartan's highest weight classification of finite-dimensional representations of $\mathfrak{g}$, for a $U_q(\widehat{\mathfrak{g}})$-module $V$, $\chi_{q}([V])$ encodes the decomposition of $V$ into common generalized eigenspaces for the action of a large commutative subalgebra (called the loop-Cartan subalgebra) of $U_q (\widehat{\mathfrak{g}})$, where $[V]\in \mathcal{K}_0(\mathscr{C})$ is the equivalent class of $V$. These generalized eigenspaces are called $\ell$-weight spaces of $V$, and generalized eigenvalues are called $\ell$-weights of $V$. It has turned out that the theory of $q$-characters already plays an important role in the study of $\mathscr{C}$, for example, every simple module in $\mathscr{C}$ is determined up to isomorphism by its $q$-character.

Frenkel and Mukhin \cite{FM01} proposed an algorithm to compute $q$-characters of some simple modules, now the algorithm is called the Frenkel-Mukhin algorithm. In some cases, the Frenkel-Mukhin algorithm does not return all terms in the $q$-character of a module, some counterexamples were given in \cite{NN11}. However, Frenkel-Mukhin algorithm produces the correct $q$-characters of modules in many cases. In particular, if a module $L(m)$ is special, then the Frenkel-Mukhin algorithm applied to $m$, produces the correct $q$-character $\chi_{q}([L(m)])$, see \cite{FM01}, where $L(m)\in \mathscr{C}$ denotes the simple $U_q (\widehat{\mathfrak{g}})$-module with the highest $\ell$-weight monomial $m$.

A new and interesting connection between $q$-characters and cluster algebras was established by Hernandez and Leclerc \cite{HL10}, in particular, the notion of monoidal categorification of a cluster algebra was introduced. Many achievements in monoidal categorifications of cluster algebras have sprung up or are emerging, see \cite{Bit21a,Bit21b,BC19,DLL19,DS20,HL16,KKKO18,KKOP22,LQ17,Q17,TDL23,ZDLL16}.

The concept of screening operators was introduced by Frenkel and Reshetikhin \cite{FR98}. For $\mathfrak{g}=\mathfrak{sl_{2}}$, Frenkel and Reshetikhin proved that the image of the $q$-character homomorphism equals the intersection of the kernels of screening operators. Frenkel and Mukhin \cite{FM01} proved it for general $\mathfrak{g}$, and predicted that a purely combinatorial algorithm for the $q$-character of a simple module may exist.

In this paper, we devote ourselves to developing a combinatorial algorithm for $q$-characters of fundamental modules of type $D_n$.

Mukhin and Young \cite{MY12a} introduced the notion of snake modules of type $A_n$ and type $B_n$, and found combinatorial models to compute $q$-characters of snake modules. The Mukhin-Young algorithm is a useful tool in subsequent studies of snake modules \cite{DLL19,DS20,MY12b,BM17}. In \cite{GDL22}, the authors introduced a path description for the $q$-characters of Hernandez-Leclerc modules of type $A_n$, where overlapped paths are allowed. In \cite{J22}, the author gave a path description for $q$-characters of fundamental modules of type $C_n$.

Inspired by Mukhin-Young's combinatorial path model for snake modules of type $B_{n-1}$, we introduce a combinatorial path model of type $D_n$, see Section \ref{Paths, corners, and moves}, such that the $q$-characters of fundamental modules of type $D_n$ are computed by paths, See Theorem \ref{path formula for fundamental modules}, where each path is assigned to a monomial or binomial, see Equations (\ref{the mononial ssociated to a path}) and (\ref{monomials associated to paths1}). Moreover, our paths are different from those of \cite{NN07}, refer the reader to \cite[Remark 7.7 (i)]{MY12a} for details. 

As a consequence, a new expression on dimensions of fundamental modules of type $D_n$ is obtained by counting our paths, see Theorem \ref{the number of monomials} and Corollary \ref{a new expression on dimensions of fundamental modules}. Note that the Chari-Pressley's decomposition \cite{CP95a} of fundamental modules as $U_q (\mathfrak{g})$-modules in fact gave a formula on dimensions of a fundamental modules of type $D_n$. Our dimensional formulas are purely combinatorial methods, just by counting paths without a priori representation-theoretical information about $\mathfrak{g}$.

The paper is organized as follows. In Section \ref{Quantum affine algebras and $q$-characters}, some necessary knowledge about $q$-characters and representations of quantum affine algebras are collected. In Section \ref{Paths, corners, and moves}, we give a combinatorial model of type $D_n$ and set   the correspondence between paths and monomials in variables $(Y^{\pm1}_{i,a})_{i\in I,a\in \mathbb{C}^{\times}}$. In Section \ref{Dimensions of fundamental modules}, dimension formulas on all the fundamental modules of type $D_{n}$ are obtained by counting our paths, see Theorem \ref{the number of monomials} and Corollary \ref{a new expression on dimensions of fundamental modules}. In Section \ref{Path description for $q$-characters of fundamental modules}, we use our combinatorial model to give an algorithm of the $q$-characters of fundamental modules of type $D_n$, see Theorem \ref{path formula for fundamental modules}. Finally, we give an example of type $D_4$ to illustrate our main theorem.

\section{Quantum affine algebras and $q$-characters}\label{Quantum affine algebras and $q$-characters}

\subsection{Cartan data}
Let $\mathfrak{g}$ be a simple Lie algebra over $\mathbb{C}$, and $I=\{1, \ldots, n\}$ the vertex set of the Dynkin diagram of $\mathfrak{g}$, where we use the same labeling with the one in \cite{Bou02}. Let $\{\alpha_{i}\} _{i\in I}$, $\{\alpha^{\vee}_{i}\} _{i\in I}$, and $\{\omega_{i}\} _{i\in I}$ be the set of simple roots, simple coroots and fundamental weights, respectively. Denote by $Q$ (resp. $Q^+$) and $P$ (resp. $P^+$) the $\mathbb{Z}$-span (resp. $\mathbb{Z}_{\geq 0}$-span) of the simple roots and fundamental weights, respectively. One can define a partial order $\leq$ on $P$ by $\lambda \leq \lambda'$ if and only if $\lambda' - \lambda \in Q^+$. Let $C=(c_{ij})_{i,\,j\in I}$ be the Cartan matrix of $\mathfrak{g}$, where $c_{ij}=\frac{2(\alpha_i,\alpha_j) }{(\alpha_i, \alpha_i)}$. There exists a diagonal matrix $D=\text{diag}(d_1,\ldots,d_n)$ with positive integer entries $d_i \,(i\in I)$ such that $B=(b_{ij})_{i,\,j\in I}=DC$ is a symmetric matrix. Here we require that $\min\{d_i \mid i\in I\}=1$.

Fix a $q\in \mathbb{C}^{\times}$, not a root of unity, one defines the $q$-number, $q$-factorial and $q$-binomial as follows:
\begin{align*}
[n]_{q} := \frac{q^n-q^{-n}}{q-q^{-1}},\quad  [n]_{q} ! := [n]_{q}[n-1]_{q} \cdots [1]_{q},\quad \binom{n}{m}_q := \frac{[n]_{q} ! }{[n-m]_{q} !  [m]_{q}  !}.
\end{align*}

\subsection{Quantum affine algebras}
Let $q\in \mathbb{C}^{\times}$ be not a root of unity unless otherwise specified. The quantum affine algebra $U_q (\widehat{\mathfrak{g}})$ has a Drinfeld's new realization \cite{Dri88,Beck94}, with generators $x^{\pm}_{i,n} \, (i\in I, n\in \mathbb{Z})$, $k_i^{\pm 1}\, (i\in I)$, $h_{i,n}\,(i\in I, n\in \mathbb{Z}\setminus\{0\})$ and central elements $c^{\pm 1/2}$, subject to the following relations (here we refer to \cite{MY12a}):
\begin{align*}
k_ik_j = k_jk_i, \quad  & k_ih_{j,n} =h_{j,n}k_i,\nonumber \\
k_ix^\pm_{j,n}k_i^{-1} &= q^{\pm b_{ij}}x_{j,n}^{\pm},\nonumber \\
[h_{i,n}, x_{j,m}^{\pm}] &= \pm \frac{1}{n} [n b_{ij}]_q c^{\mp {|n|/2}}x_{j,n+m}^{\pm},\nonumber \\
x_{i,n+1}^{\pm}x_{j,m}^{\pm} -q^{\pm b_{ij}}x_{j,m}^{\pm}x_{i,n+1}^{\pm} &=q^{\pm b_{ij}}x_{i,n}^{\pm}x_{j,m+1}^{\pm}-x_{j,m+1}^{\pm}x_{i,n}^{\pm}, \\ [h_{i,n},h_{j,m}] & =\delta_{n,-m} \frac{1}{n} [n b_{ij}]_q \frac{c^n - c^{-n}}{q-q^{-1}},\nonumber \\
[x_{i,n}^+ , x_{j,m}^-]=\delta_{ij} & \frac{ c^{(n-m)/2}\phi_{i,n+m}^+ - c^{-(n-m)/2} \phi_{i,n+m}^-}{q^{d_i}-q^{-d_i}}, \\
\sum_{\pi\in\Sigma_s}\sum_{k=0}^s(-1)^k \left[\begin{array}{cc} s \\ k \end{array} \right]_{q^{d_i}} x_{i, n_{\pi(1)}}^{\pm} \ldots x_{i,n_{\pi(k)}}^{\pm} & x_{j,m}^{\pm} x_{i,n_{\pi(k+1)}}^{\pm}\ldots x_{i,n_{\pi(s)}}^{\pm} =0,\nonumber \,\,\, s=1-c_{ij},
\end{align*}
for all sequences of integers $n_1,\ldots,n_s$, and $i\ne j$, where $\Sigma_s$ is the symmetric group on $\{1,\ldots,s\}$, and $\phi_{i,n}^{\pm}$'s are defined by the formula
\begin{equation*} \label{phidef}
\phi_i^\pm(u) := \sum_{n=0}^{\infty}\phi_{i,\pm n}^{\pm}u^{\pm n} = k_i^{\pm 1} \exp\left(\pm(q-q^{-1})\sum_{m=1}^{\infty}h_{i,\pm m} u^{\pm m}\right),
\end{equation*}
where $\phi^{+}_{i, n}=0$ for $n<0$, and $\phi^{-}_{i, n}=0$ for $n>0$.

The quantum affine algebra $U_q (\widehat{\mathfrak{g}})$ is an associative and non-commutative algebra. There exist a coproduct, counit and antipode making $U_q (\widehat{\mathfrak{g}})$ into a Hopf algebra, see \cite[Proposition 1.2]{C95}. Let $U_q(\mathfrak{g})$ be the quantized universal enveloping algebra of $\mathfrak{g}$ with Chevalley generators $x^{\pm}_i$ and $k^{\pm 1}_i$, with $i\in I$, subject to Chevalley-Serre relations, see \cite[Definition 9.1.1]{CP94}. It is well-known that $U_q(\mathfrak{g})$ is a (Hopf) subalgebra of $U_q (\widehat{\mathfrak{g}})$. So, every $U_q (\widehat{\mathfrak{g}})$-module restricts to a $U_q (\mathfrak{g})$-module.

\subsection{Finite-dimensional representations of $U_q (\widehat{\mathfrak{g}})$}
In this section, we recall some  necessary background about finite-dimensional representations of $U_q (\widehat{\mathfrak{g}})$.

A representation $V$ of $U_q (\widehat{\mathfrak{g}})$ is of type $1$ if $c^{\pm 1/2}$ act as the identity on $V$ and $V$ is of type 1 as a  $U_q (\mathfrak{g})$-module, that is,
\begin{equation}\label{decomposition}
V = \bigoplus_{\lambda\in P} V_\lambda, \quad  V_\lambda = \{ v \in V \mid k_i  v = q^{(\alpha_i, \lambda)} v\}.
\end{equation}
Following \cite{CP94}, every finite-dimensional irreducible representation of $U_q (\widehat{\mathfrak{g}})$ can be obtained from a type $1$ representation by twisting with an automorphism of $U_q (\widehat{\mathfrak{g}})$. In what follows, all representations are assumed to be finite-dimensional and of type $1$.

In (\ref{decomposition}), the decomposition of a finite-dimensional representation $V$ into its $U_q (\widehat{\mathfrak{g}})$-weight spaces can be refined by decomposing it into Jordan subspaces of mutually commuting operators
\begin{equation*}
 V = \bigoplus_{\gamma} V_\gamma, \quad \gamma = (\gamma_{i,\pm r}^\pm)_{i\in I, r\in  \mathbb{Z}_{\geq 0}}, \quad \gamma_{i,\pm r}^\pm \in \mathbb{C},
\end{equation*}
where
\begin{equation*}
V_{\gamma} = \{  v \in V \mid \exists \, k \in \mathbb{N}, \, \forall \, i \in I, m\geq 0, \, \left(\phi_{i,\pm m}^\pm - \gamma_{i,\pm m}^\pm\right)^k v = 0 \}.
\end{equation*}

If $\dim(V_{\gamma})>0$, then $\gamma$ is called an \textit{$\ell$-weight} of $V$, and $V_{\gamma}$ is called an \textit{$\ell$-weight space} of $V$ with $\ell$-weight $\gamma$. Following \cite{FR98}, for every finite-dimensional representation of $U_q (\widehat{\mathfrak{g}})$, the $\ell$-weights are of the form
\begin{equation}\label{gwts}
\gamma_i^\pm(u) := \sum_{r =0}^\infty \gamma_{i,\pm r}^\pm u^{\pm r}= q^{d_i(\deg Q_i - \deg R_i)} \, \frac{Q_i(uq^{-d_i}) R_i(uq^{d_i})}{Q_i(uq^{d_i}) R_i(uq^{-d_i})} \,,\end{equation}
where the right hand side is to be treated as a formal series in positive (resp. negative) integer powers of $u$, and $Q_i$ and $R_i$ are polynomials of the form
\begin{equation}\label{gwtt}
Q_i(u) = \prod_{a\in \mathbb{C}^{\times}} \left( 1- ua\right)^{w_{i,a}}, \quad R_i(u) = \prod_{a\in \mathbb{C}^{\times}} \left( 1- ua\right)^{x_{i,a}},
\end{equation}
for some $w_{i,a}, x_{i,a}\geq 0$, $i\in I,a\in \mathbb{C}^{\times}$.

Let $\mathcal{P}$ be the free abelian multiplicative group of monomials in infinitely many formal variables $(Y^{\pm 1}_{i,a})_{i\in I, a\in \mathbb{C}^{\times}}$, and $\mathcal{P}^{+}$ (respectively, $\mathcal{P}^{-}$) the submonoid of $\mathcal{P}$ generated by $(Y_{i,a})_{i\in I, a\in \mathbb{C}^{\times}}$ (respectively, $(Y^{-1}_{i,a})_{i\in I, a\in \mathbb{C}^{\times}}$). Every monomial in $\mathcal{P}^+$ (respectively, $\mathcal{P}^-$) is called a dominant (respectively, anti-dominant) monomial. There is a bijection from $\mathcal{P}$ to the set of $\ell$-weights $\gamma$ of finite-dimensional $U_q (\widehat{\mathfrak{g}})$-modules such that for the monomial
\[
m = \prod_{i \in I, a\in \mathbb{C}^{\times}} Y_{i,a}^{w_{i,a}-x_{i,a}},
\]
the $\ell$-weights are given by (\ref{gwts}), (\ref{gwtt}). We identify $\ell$-weights of finite-dimensional representations with elements of $\mathcal{P}$ in this way.

It is well-known that every finite-dimensional $U_q (\widehat{\mathfrak{g}})$-module $V$ is a highest $\ell$-weight module, that is, there exists a non-zero vector $v\in V$ such that $\phi^\pm_{i,\pm t}$, with $i\in I, t\in \mathbb{Z}_{\geq 0}$, diagonally act on $V$ and
\begin{equation*}
x^+_{i,r}v = 0 \, \text{ for all } i\in I, r\in \mathbb{Z}.
\end{equation*}

It is known that for each $m \in\mathcal{P}^{+}$, there is a unique finite-dimensional irreducible representation, denoted by $L(m)$, of $U_q(\widehat{\mathfrak{g}})$ that is a highest $\ell$-weight module with the highest $\ell$-weight $\gamma(m)$, and moreover every finite-dimensional irreducible $U_q(\widehat{\mathfrak{g}})$-module is of this form for some $m\in \mathcal{P}^{+}$.

In the case where $m=Y_{i,a}$ for some $i\in I$, $a\in \mathbb{C}^{\times}$, $L(Y_{i,a})$ is called a \textit{fundamental module}. If $m_1,m_2 \in \mathcal{P}^{+}$ and $m_1 \neq m_2$, then $L(m_1)\neq L(m_2)$.

From now on, we fix an $a\in \mathbb{C}^{\times}$, by an abuse of notation, it is convenient to write
\[
Y_{i,k}:=Y_{i,aq^{k}}, \quad A_{i,k}:=A_{i,aq^{k}}.
\]

A finite-dimensional $U_q(\widehat{\mathfrak{g}})$-module $V$ is said to be \textit{thin} if and only if  every $\ell$-weight space of $V$ has dimension $1$. %In other words, the module is thin if and only if the $(\phi_{i,\pm r}^\pm)_{i\in I,r\in \mathbb{Z}_{\geq 0}}$ are simultaneously diagonalizable with joint simple spectrum.
A finite-dimensional $U_q(\widehat{\mathfrak{g}})$-module $V$ is said to be \textit{prime} if and only if it cannot be written as a tensor product of two non-trivial $U_q(\widehat{\mathfrak{g}})$-modules \cite{CP97}.

\subsection{$q$-characters}
The notion of $q$-characters was introduced by Frenkel and Reshetikhin \cite{FR98}.

Let $\mathcal{Y}=\mathbb{Z}[Y^{\pm1}_{i,a}]_{i\in I, a\in\mathbb{C}^{\times}}$ be the ring of Laurent polynomials in the variables $(Y_{i,a})_{i\in I, a\in \mathbb{C}^{\times}}$ with integer coefficients, and $\mathrm{Rep}(U_q (\widehat{\mathfrak{g}}))$ the Grothendieck ring of finite-dimensional representations of $U_q (\widehat{\mathfrak{g}})$ and $[V] \in \mathrm{Rep}(U_q (\widehat{\mathfrak{g}}))$ the equivalent class of a finite-dimensional $U_q (\widehat{\mathfrak{g}})$-module $V$.

The $q$-character map is defined as an injective ring homomorphism from $\mathrm{Rep}(U_q (\widehat{\mathfrak{g}}))$ to $\mathcal{Y}$ such that for any finite-dimensional $U_q (\widehat{\mathfrak{g}})$-module $V$,
\begin{align*}
\chi_q([V])=\sum_{m\in \mathcal{P}} \dim(V_{m})m,
\end{align*}
where $V_m$ is the $\ell$-weight space of $V$ with $\ell$-weight $m$. Define $A_{i,a}\in \mathcal{P}$, with $i\in I, a\in\mathbb{C}^{\times}$, by
\begin{equation*}
\label{adef} A_{i,a} = Y_{i,aq^{d_i}} Y_{i,aq^{-d_i}} \prod_{c_{ji}=-1}Y^{-1}_{j,a}  \prod_{c_{ji}=-2}Y^{-1}_{j,aq} Y^{-1}_{j,aq^{-1}} \prod_{c_{ji}=-3}Y^{-1}_{j,aq^2} Y^{-1}_{j,a} Y^{-1}_{j,aq^{-2}}.
\end{equation*}

Let $\mathcal{Q}$ be the subgroup of $\mathcal{P}$ generated by $A^{\pm 1}_{i, a}$, with $i\in I, a\in \mathbb{C}^{\times}$. Let $\mathcal{Q}^{\pm}$ be the monoids generated by $A_{i, a}^{\pm 1}$, with $i\in I, a\in \mathbb{C}^{\times}$. There is a partial order $\leq$ on $\mathcal{P}$ in which
\begin{equation*}
m\leq m' \text{ if and only if } m'm^{-1}\in \mathcal{Q}^{+}. \label{partial order of monomials}
\end{equation*}

For a simple finite-dimensional $U_q (\widehat{\mathfrak{g}})$-module $V$, Frenkel and Mukhin \cite{FM01} proved that the $q$-character of $V$ has the following form
\begin{equation}\label{mathcal{Q}}
\chi_{q}([V]) = m_{+}(1+\sum_{p}M_{p}),
\end{equation}
where $m_{+}\in \mathcal{P}^+$, and each $M_{p}$ is a monomial in $A^{-1}_{i,r}$, with $i\in I, r\in \mathbb{C}^{\times}$.

Every simple finite-dimensional $U_q (\widehat{\mathfrak{g}})$-module is determined up to isomorphism by its $q$-character. A finite-dimensional $U_q(\widehat{\mathfrak{g}})$-module $V$ is said to be \textit{special} if and only if $\chi_q([V])$ has exactly one dominant monomial. It is \textit{anti-special} if and only if $\chi_q([V])$ has exactly one anti-dominant monomial. It is well-known that if a module is special or anti-special, then it is simple.

For $m \in \mathcal{P}^{+}$, let $\mathscr{M}$ be the set of all monomials in $\chi_{q}([L(m)])$. If $m\in \mathscr{M}$, then we write $m\in \chi_{q}([L(m)])$.

\subsection{Screening operators}

For each $i\in I$, let $\widetilde{\mathcal{Y}}_i$ be the free $\mathcal Y$-module with basis $S_{i,x}$, with $x \in\mathbb{C}^{\times}$, and ${\mathcal Y}_i$ the quotient of $\widetilde{\mathcal{Y}}_i$ by the submodule generated by elements of the form $S_{i,xq_i^{2}}=A_{i,xq_i}S_{i,x}$, where $q_i=q^{d_i}$.

Define a linear operator $\widetilde{S_{i}} : \mathcal Y \rightarrow \widetilde{\mathcal{Y}}_i$ by the formula
\[
\widetilde{S_{i}}(Y_{j,a}) = \delta_{ij} Y_{i,a}S_{i,a},
\]
and the Leibniz rule: $\widetilde{S_{i}}(ab) = b\widetilde{S_{i}}(a) + a\widetilde{S_{i}}(b)$. By definition, we have
\[
\widetilde{S_{i}}(Y^{-1}_{j,a}) = -\delta_{ij} Y^{-1}_{i,a} S_{i,a}.
\]

Let $S_{i} :\mathcal Y \rightarrow {\mathcal Y}_{i}$ be the composition of $\widetilde{S_{i}}$ and the canonical projection $\pi_{i}: \widetilde{\mathcal{Y}}_i \rightarrow {\mathcal Y}_{i}$. The $S_{i}$ is called the $i$-th screening operator.

The following two propositions will be very useful in the proof of our theorem.

\begin{proposition}[{\cite[Proposition 5.2]{FM01}}]\label{Kernel}
The kernel of $S_{i} : \mathcal Y \rightarrow {\mathcal Y}_{i}$ equals
 \begin{equation*}
 \mathbb{Z}[Y^{\pm 1}_{j,a}]_{j\neq i,a\in \mathbb{C}^{\times}}\otimes \mathbb{Z}[Y_{i,b} + Y_{i,b}A^{-1}_{i,bq^{d_i}}]_{b\in \mathbb{C}^{\times}}.
\end{equation*}
\end{proposition}

\begin{proposition}[{\cite[Corollary 5.7]{FM01}}] \label{isomorphism}
The image of the $q$-character homomorphism $\chi_q$ equals the intersection of the kernels of the screening operators $S_i$, with $i\in I$, equivalently,
\[
\chi_{q}: Rep(U_q(\widehat{\mathfrak{g}}))\rightarrow \bigcap_{i\in I} ker S_{i}
\]
is a ring isomorphism.
\end{proposition}

Proposition \ref{isomorphism} was conjectured by Frenkel and Reshetikhin \cite{FR98}, and they proved it for $\mathfrak{g}=\mathfrak{sl}_2$. Subsequently, Frenkel and Mukhin \cite{FM01} proved it for general $\mathfrak{g}$.

\section{A combinatorial model of type $D_n$}\label{Paths, corners, and moves}

From now on, let $\mathfrak{g}$ be the simple Lie algebra of type $D_n$, with $n\geq 4$. Inspired by Mukhin-Young's combinatorial model of type $B_{n-1}$ \cite{MY12a}, we introduce a combinatorial model of type $D_n$.

\subsection{Paths}\label{defintion of paths}

Let $N=2n-2$ and $\mathscr{S}=\{i \in \mathbb{Z} \mid 0\leq i\leq N\}$. Define a map $\overline{\cdot}$ from $\mathscr{S}$ to the power set of $\{0,1,\ldots,n\}$ such that
\begin{align*}
\overline{i} =
\begin{cases}
\{ i \} & \text{if $0 \leq i \leq n-2$}, \\
\{ n-1, n\} & \text{if $i=n-1$}, \\
\{ N-i \} & \text{if $n \leq i \leq N$}.
\end{cases}
\end{align*}
Subsequently, for a single element set $\overline{i}$, we denote by its element the set $\overline{i}$.

We define a subset $\mathcal{X}$ of $\mathscr{S} \times \mathbb{Z}$ as follows:
\begin{align*}
\mathcal{X} :=  \{ (i,k) \in \mathscr{S} \times \mathbb{Z} \mid  i-k\equiv 1  \hspace{-0.3cm} \pmod  2 \}.
\end{align*}

%%and
%%\[
%%\mathcal{W}=\{(\overline{i},k)\mid (i,k-1)\in \mathcal{X}\} \cup \{(n,k) \mid  k \in \mathbb{Z}, n-k\equiv 1 \hspace{-0.3cm} \pmod 2\}.
%%\]
%%Then by (\ref{mathcal{Q}}), for all $m \in \mathbb{Z}[Y_{i,aq^{k}}]_{(i,k)\in \mathcal{X}}$, we have
%%\[
%%\mathscr{M}(L(m)) \subset m \mathbb{Z}[A_{i,aq^{k}}^{-1}]_{(i,k)\in \mathcal{W}}.
%%\]

A path is a finite sequence of points in the plane $\mathbb{R}^{2}$. For each $(n-1, k)\in \mathcal{X}$, we define a set $\mathscr{P}_{n-1,k} \subset \mathcal{X}$ of paths in the following way:
\begin{align*}
\mathscr{P}_{n-1,k}=\{ & ((0,y_{0}),(1,y_{1}),\ldots,(n-2,y_{n-2}),(n-1,y_{n-1})) \mid  y_{0}=n-1+k, \\
&\text{ and } y_{i+1}-y_{i}\in \{1,-1\}, \  0\leq i\leq n-2\}.
\end{align*}

To our needs, let
\begin{align*}
\widehat{\mathscr{P}}_{n-1,k}=\{ & ((N,y_{0}),(N-1,y_{1}),\ldots,(n,y_{n-2}),(n-1,y_{n-1}))\mid y_{0}=n-1+k, \\
&\text{ and } y_{i+1}-y_{i}\in \{1,-1\}, \  0\leq i\leq n-2\}.
\end{align*}
So points in $\widehat{\mathscr{P}}_{n-1,k}$ are symmetric with points in $\mathscr{P}_{n-1,k}$ with respect to $x=n-1$ axis.

For $(i,k)\in \mathcal{X}$, where $i\in\{1,2,\ldots,n-2\}$, we define a set $\mathscr{P}_{i,k} \subset \mathcal{X}$ of paths in the following way:
\begin{align*}
\mathscr{P}_{i,k}=\{ & (a_0,a_1,\ldots, a_{n-1}, \overline{a}_{n-1}, \overline{a}_{n-2}, \ldots, \overline{a}_0)\mid (a_0,a_1,\ldots, a_{n-1})\in \mathscr{P}_{n-1,k-(n-i-1)}, \\
& (\overline{a}_{n-1}, \overline{a}_{n-2}, \ldots, \overline{a}_0) \in  \widehat{\mathscr{P}}_{n-1,k+(n-i-1)},   \\
& \, a_{n-1}-\overline{a}_{n-1}=(0,y),  \text{ where } y\geq 0 \}.
\end{align*}
For simplicity, we assume without loss of generality that each path in $\mathscr{P}_{i,k}$ has the following form:
\begin{align}\label{path form in pik}
p=((0,y_{0}),(1,y_{1}),\ldots,(j,y_{j}),\ldots,(n-1,y_{n-1}), (n-1,y'_{n-1}), (n,y_n),\ldots, (N,y_N)).
\end{align}

To illustrate our paths, we give examples of $\mathscr{P}_{6,1}$ and $\mathscr{P}_{9,0}$ for $n=10$, see Figure \ref{path in grid1} and Figure \ref{path in grid2}. In our figures, we connect consecutive points of a path by line segments, for illustrative purposes only. We write $(j, \ell) \in p$ if $(j, \ell)$ is a point in a path $p$.

\begin{figure}
\resizebox{1.0\width}{1.0\height}{
\begin{minipage}[b]{0.4\linewidth}
\centerline{
\begin{tikzpicture}[scale=.5]
\draw[help lines, color=gray!60, thin] (0,0) grid (9,18);
\begin{scope}[thick, every node/.style={sloped,allow upside down}]
\draw (0,9)--node {\midarrow}(1,10)--node {\midarrow}(2,11)--node {\midarrow}(3,12)--node {\midarrow}(4,13)--node {\midarrow}(5,14)--node {\midarrow}(6,15)--node {\midarrow}(7,16)--node {\midarrow}(8,17)--node {\midarrow}(9,18);
\draw (1,8)--node {\midarrow}(2,9)--node {\midarrow}(3,10)--node {\midarrow}(4,11)--node {\midarrow}(5,12)--node {\midarrow}(6,13)--node {\midarrow}(7,14)--node {\midarrow}(8,15)--node {\midarrow}(9,16);
\draw (2,7)--node {\midarrow}(3,8)--node {\midarrow}(4,9)--node {\midarrow}(5,10)--node {\midarrow}(6,11)--node {\midarrow}(7,12)--node {\midarrow}(8,13)--node {\midarrow}(9,14);
\draw (3,6)--node {\midarrow}(4,7)--node {\midarrow}(5,8)--node {\midarrow}(6,9)--node {\midarrow}(7,10)--node {\midarrow}(8,11)--node {\midarrow}(9,12);
\draw (4,5)--node {\midarrow}(5,6)--node {\midarrow}(6,7)--node {\midarrow}(7,8)--node {\midarrow}(8,9)--node {\midarrow}(9,10);
\draw (5,4)--node {\midarrow}(6,5)--node {\midarrow}(7,6)--node {\midarrow}(8,7)--node {\midarrow}(9,8);
\draw (6,3)--node {\midarrow}(7,4)--node {\midarrow}(8,5)--node {\midarrow}(9,6);
\draw (7,2)--node {\midarrow}(8,3)--node {\midarrow}(9,4);
\draw (8,1)--node {\midarrow}(9,2);
\draw (0,9)--node {\midarrow}(1,8)--node {\midarrow}(2,7)--node {\midarrow}(3,6)--node {\midarrow}(4,5)--node {\midarrow}(5,4)--node {\midarrow}(6,3)--node {\midarrow}(7,2)--node {\midarrow}(8,1)--node {\midarrow}(9,0);
\draw (1,10)--node {\midarrow}(2,9)--node {\midarrow}(3,8)--node {\midarrow}(4,7)--node {\midarrow}(5,6)--node {\midarrow}(6,5)--node {\midarrow}(7,4)--node {\midarrow}(8,3)--node {\midarrow}(9,2);
\draw (2,11)--node {\midarrow}(3,10)--node {\midarrow}(4,9)--node {\midarrow}(5,8)--node {\midarrow}(6,7)--node {\midarrow}(7,6)--node {\midarrow}(8,5)--node {\midarrow}(9,4);
\draw (3,12)--node {\midarrow}(4,11)--node {\midarrow}(5,10)--node {\midarrow}(6,9)--node {\midarrow}(7,8)--node {\midarrow}(8,7)--node {\midarrow}(9,6);
\draw (4,13)--node {\midarrow}(5,12)--node {\midarrow}(6,11)--node {\midarrow}(7,10)--node {\midarrow}(8,9)--node {\midarrow}(9,8);
\draw (5,14)--node {\midarrow}(6,13)--node {\midarrow}(7,12)--node {\midarrow}(8,11)--node {\midarrow}(9,10);
\draw (6,15)--node {\midarrow}(7,14)--node {\midarrow}(8,13)--node {\midarrow}(9,12);
\draw (7,16)--node {\midarrow}(8,15)--node {\midarrow}(9,14);
\draw (8,17)--node {\midarrow}(9,16);
\end{scope}
\node at (0, 18.3) {\scalebox{0.55}{\scri{$0$}}};
\node at (1, 18.3) {\scalebox{0.55}{\scri{$1$}}};
\node at (2, 18.3) {\scalebox{0.55}{\scri{$2$}}};
\node at (3, 18.3) {\scalebox{0.55}{\scri{$3$}}};
\node at (4, 18.3) {\scalebox{0.55}{\scri{$4$}}};
\node at (5, 18.3) {\scalebox{0.55}{\scri{$5$}}};
\node at (6, 18.3) {\scalebox{0.55}{\scri{$6$}}};
\node at (7, 18.3) {\scalebox{0.55}{\scri{$7$}}};
\node at (8, 18.3) {\scalebox{0.55}{\scri{$8$}}};
\node at (9, 18.3) {\scalebox{0.55}{\scri{$9$}}};
%\node at (1, 18.3) {\scalebox{0.55}{\scri{$\bullet$}}};
%\node at (2, 18.3) {\scalebox{0.55}{\scri{$\bullet$}}};
%\node at (3, 18.3) {\scalebox{0.55}{\scri{$\bullet$}}};
%\node at (4, 18.3) {\scalebox{0.55}{\scri{$\bullet$}}};
%\node at (5, 18.3) {\scalebox{0.55}{\scri{$\bullet$}}};
%\node at (6, 18.3) {\scalebox{0.55}{\scri{$\bullet$}}};
%\node at (7, 18.3) {\scalebox{0.55}{\scri{$\bullet$}}};
%\node at (8, 18.3) {\scalebox{0.55}{\scri{$\bullet$}}};
%\node at (9, 18.3) {\scalebox{0.55}{\scri{$\bullet$}}};
\node at (-0.5, 18) {\scalebox{0.55}{\scri{$0$}}};
\node at (-0.5, 17) {\scalebox{0.55}{\scri{$1$}}};
\node at (-0.5, 16) {\scalebox{0.55}{\scri{$2$}}};
\node at (-0.5, 15) {\scalebox{0.55}{\scri{$3$}}};
\node at (-0.5, 14) {\scalebox{0.55}{\scri{$4$}}};
\node at (-0.5, 13) {\scalebox{0.55}{\scri{$5$}}};
\node at (-0.5, 12) {\scalebox{0.55}{\scri{$6$}}};
\node at (-0.5, 11) {\scalebox{0.75}{\scri{$7$}}};
\node at (-0.5, 10) {\scalebox{0.75}{\scri{$8$}}};
\node at (-0.5, 9) {\scalebox{0.55}{\scri{$9$}}};
\node at (-0.5, 8) {\scalebox{0.55}{\scri{$10$}}};
\node at (-0.5, 7) {\scalebox{0.55}{\scri{$11$}}};
\node at (-0.5, 6) {\scalebox{0.55}{\scri{$12$}}};
\node at (-0.5, 5) {\scalebox{0.55}{\scri{$13$}}};
\node at (-0.5, 4) {\scalebox{0.55}{\scri{$14$}}};
\node at (-0.5, 3) {\scalebox{0.55}{\scri{$15$}}};
\node at (-0.5, 2) {\scalebox{0.75}{\scri{$16$}}};
\node at (-0.5, 1) {\scalebox{0.75}{\scri{$17$}}};
\node at (-0.5, 0) {\scalebox{0.75}{\scri{$18$}}};
%\draw (1, 18.3) -- (2, 18.3) -- (3, 18.3) -- (4, 18.3) -- (5, 18.3) -- (6, 18.3) -- (7, 18.3) -- (8, 18.3) -- (9, 18.3);
\end{tikzpicture}}
\end{minipage}}
\caption{All paths in $\mathscr{P}_{9,0}$ for $n=10$.}\label{path in grid1}
\end{figure}
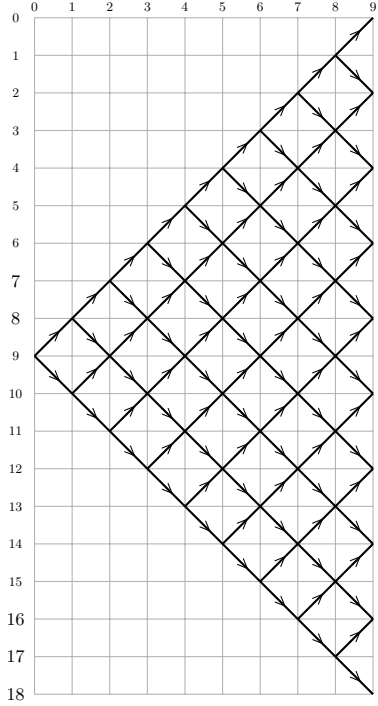

\begin{figure}
\resizebox{1.0\width}{1.0\height}{
\begin{minipage}[b]{0.4\linewidth}
\centerline{
\begin{tikzpicture}[scale=.5]
\draw[help lines, color=gray!60, thin] (0,1) grid (18,19);
\begin{scope}[thick, every node/.style={sloped,allow upside down}]
\draw (0,13)--node {\midarrow}(1,14)--node {\midarrow}(2,15)--node {\midarrow}(3,16)--node {\midarrow}(4,17)--node {\midarrow}(5,18)--node {\midarrow}(6,19);
\draw (1,12)--node {\midarrow}(2,13)--node {\midarrow}(3,14)--node {\midarrow}(4,15)--node {\midarrow}(5,16)--node {\midarrow}(6,17)--node {\midarrow}(7,18);
\draw (2,11)--node {\midarrow}(3,12)--node {\midarrow}(4,13)--node {\midarrow}(5,14)--node {\midarrow}(6,15)--node {\midarrow}(7,16)--node {\midarrow}(8,17);
\draw (3,10)--node {\midarrow}(4,11)--node {\midarrow}(5,12)--node {\midarrow}(6,13)--node {\midarrow}(7,14)--node {\midarrow}(8,15)--node {\midarrow}(9,16);
\draw (4,9)--node {\midarrow}(5,10)--node {\midarrow}(6,11)--node {\midarrow}(7,12)--node {\midarrow}(8,13)--node {\midarrow}(9,14)--node {\midarrow}(10,15);
\draw(5,8)--node {\midarrow}(6,9)--node {\midarrow}(7,10)--node {\midarrow}(8,11)--node {\midarrow}(9,12)--node {\midarrow}(10,13)--node {\midarrow}(11,14);
\draw(6,7)--node {\midarrow}(7,8)--node {\midarrow}(8,9)--node {\midarrow}(9,10)--node {\midarrow}(10,11)--node {\midarrow}(11,12)--node {\midarrow}(12,13);
\draw(7,6)--node {\midarrow}(8,7)--node {\midarrow}(9,8)--node {\midarrow}(10,9)--node {\midarrow}(11,10)--node {\midarrow}(12,11)--node {\midarrow}(13,12);
\draw(8,5)--node {\midarrow}(9,6)--node {\midarrow}(10,7)--node {\midarrow}(11,8)--node {\midarrow}(12,9)--node {\midarrow}(13,10)--node {\midarrow}(14,11);
\draw(9,4)--node {\midarrow}(10,5)--node {\midarrow}(11,6)--node {\midarrow}(12,7)--node {\midarrow}(13,8)--node {\midarrow}(14,9)--node {\midarrow}(15,10);
\draw(10,3)--node {\midarrow}(11,4)--node {\midarrow}(12,5)--node {\midarrow}(13,6)--node {\midarrow}(14,7)--node {\midarrow}(15,8)--node {\midarrow}(16,9);
\draw(11,2)--node {\midarrow}(12,3)--node {\midarrow}(13,4)--node {\midarrow}(14,5)--node {\midarrow}(15,6)--node {\midarrow}(16,7)--node {\midarrow}(17,8);
\draw(12,1)--node {\midarrow}(13,2)--node {\midarrow}(14,3)--node {\midarrow}(15,4)--node {\midarrow}(16,5)--node {\midarrow}(17,6)--node {\midarrow}(18,7);
\draw(6,19)--node {\midarrow}(7,18)--node {\midarrow}(8,17)--node {\midarrow}(9,16)--node {\midarrow}(10,15)--node {\midarrow}(11,14)--node {\midarrow}(12,13)--node {\midarrow}(13,12)--node {\midarrow}(14,11)--node {\midarrow}(15,10)--node {\midarrow}(16,9)--node {\midarrow}(17,8)--node {\midarrow}(18,7);
\draw(5,18)--node {\midarrow}(6,17)--node {\midarrow}(7,16)--node {\midarrow}(8,15)--node {\midarrow}(9,14)--node {\midarrow}(10,13)--node {\midarrow}(11,12)--node {\midarrow}(12,11)--node {\midarrow}(13,10)--node {\midarrow}(14,9)--node {\midarrow}(15,8)--node {\midarrow}(16,7)--node {\midarrow}(17,6);
\draw(4,17)--node {\midarrow}(5,16)--node {\midarrow}(6,15)--node {\midarrow}(7,14)--node {\midarrow}(8,13)--node {\midarrow}(9,12)--node {\midarrow}(10,11)--node {\midarrow}(11,10)--node {\midarrow}(12,9)--node {\midarrow}(13,8)--node {\midarrow}(14,7)--node {\midarrow}(15,6)--node {\midarrow}(16,5);
\draw(3,16)--node {\midarrow}(4,15)--node {\midarrow}(5,14)--node {\midarrow}(6,13)--node {\midarrow}(7,12)--node {\midarrow}(8,11)--node {\midarrow}(9,10)--node {\midarrow}(10,9)--node {\midarrow}(11,8)--node {\midarrow}(12,7)--node {\midarrow}(13,6)--node {\midarrow}(14,5)--node {\midarrow}(15,4);
\draw(2,15)--node {\midarrow}(3,14)--node {\midarrow}(4,13)--node {\midarrow}(5,12)--node {\midarrow}(6,11)--node {\midarrow}(7,10)--node {\midarrow}(8,9)--node {\midarrow}(9,8)--node {\midarrow}(10,7)--node {\midarrow}(11,6)--node {\midarrow}(12,5)--node {\midarrow}(13,4)--node {\midarrow}(14,3);
\draw(1,14)--node {\midarrow}(2,13)--node {\midarrow}(3,12)--node {\midarrow}(4,11)--node {\midarrow}(5,10)--node {\midarrow}(6,9)--node {\midarrow}(7,8)--node {\midarrow}(8,7)--node {\midarrow}(9,6)--node {\midarrow}(10,5)--node {\midarrow}(11,4)--node {\midarrow}(12,3)--node {\midarrow}(13,2);
\draw(0,13)--node {\midarrow}(1,12)--node {\midarrow}(2,11)--node {\midarrow}(3,10)--node {\midarrow}(4,9)--node {\midarrow}(5,8)--node {\midarrow}(6,7)--node {\midarrow}(7,6)--node {\midarrow}(8,5)--node {\midarrow}(9,4)--node {\midarrow}(10,3)--node {\midarrow}(11,2)--node {\midarrow}(12,1);
\draw(9,4)--node {\midarrow}(9,6)--node {\midarrow}(9,8)--node {\midarrow}(9,10)--node {\midarrow}(9,12)--node {\midarrow}(9,14)--node {\midarrow}(9,16);
\end{scope}
\node at (0, 19.4) {\scalebox{0.55}{\scri{$0$}}};
\node at (1, 19.4) {\scalebox{0.55}{\scri{$1$}}};
\node at (2, 19.4) {\scalebox{0.55}{\scri{$2$}}};
\node at (3, 19.4) {\scalebox{0.55}{\scri{$3$}}};
\node at (4, 19.4) {\scalebox{0.55}{\scri{$4$}}};
\node at (5, 19.4) {\scalebox{0.55}{\scri{$5$}}};
\node at (6, 19.4) {\scalebox{0.55}{\scri{$6$}}};
\node at (7, 19.4) {\scalebox{0.55}{\scri{$7$}}};
\node at (8, 19.4) {\scalebox{0.55}{\scri{$8$}}};
\node at (9, 19.4) {\scalebox{0.55}{\scri{$9$}}};
\node at (10, 19.4) {\scalebox{0.55}{\scri{$10$}}};
\node at (11, 19.4) {\scalebox{0.55}{\scri{$11$}}};
\node at (12, 19.4) {\scalebox{0.55}{\scri{$12$}}};
\node at (13, 19.4) {\scalebox{0.55}{\scri{$13$}}};
\node at (14, 19.4) {\scalebox{0.55}{\scri{$14$}}};
\node at (15, 19.4) {\scalebox{0.55}{\scri{$15$}}};
\node at (16, 19.4) {\scalebox{0.55}{\scri{$16$}}};
\node at (17, 19.4) {\scalebox{0.55}{\scri{$17$}}};
\node at (18, 19.4) {\scalebox{0.55}{\scri{$18$}}};
%\node at (1, 19.4) {\scalebox{0.55}{\scri{$\bullet$}}};
%\node at (2, 19.4) {\scalebox{0.55}{\scri{$\bullet$}}};
%\node at (3, 19.4) {\scalebox{0.55}{\scri{$\bullet$}}};
%\node at (4, 19.4) {\scalebox{0.55}{\scri{$\bullet$}}};
%\node at (5, 19.4) {\scalebox{0.55}{\scri{$\bullet$}}};
%\node at (6, 19.4) {\scalebox{0.55}{\scri{$\bullet$}}};
%\node at (7, 19.4) {\scalebox{0.55}{\scri{$\bullet$}}};
%\node at (8, 19.4) {\scalebox{0.55}{\scri{$\bullet$}}};
%\node at (9, 19.4) {\scalebox{0.55}{\scri{$\bullet$}}};
%\node at (10, 19.4) {\scalebox{0.55}{\scri{$\bullet$}}};
%\node at (11, 19.4) {\scalebox{0.55}{\scri{$\bullet$}}};
%\node at (12, 19.4) {\scalebox{0.55}{\scri{$\bullet$}}};
%\node at (13, 19.4) {\scalebox{0.55}{\scri{$\bullet$}}};
%\node at (14, 19.4) {\scalebox{0.55}{\scri{$\bullet$}}};
%\node at (15, 19.4) {\scalebox{0.55}{\scri{$\bullet$}}};
%\node at (16, 19.4) {\scalebox{0.55}{\scri{$\bullet$}}};
%\node at (17, 19.4) {\scalebox{0.55}{\scri{$\bullet$}}};
\node at (-0.5, 19) {\scalebox{0.55}{\scri{$1$}}};
\node at (-0.5, 18) {\scalebox{0.55}{\scri{$2$}}};
\node at (-0.5, 17) {\scalebox{0.55}{\scri{$3$}}};
\node at (-0.5, 16) {\scalebox{0.55}{\scri{$4$}}};
\node at (-0.5, 15) {\scalebox{0.55}{\scri{$5$}}};
\node at (-0.5, 14) {\scalebox{0.55}{\scri{$6$}}};
\node at (-0.5, 13) {\scalebox{0.55}{\scri{$7$}}};
\node at (-0.5, 12) {\scalebox{0.55}{\scri{$8$}}};
\node at (-0.5, 11) {\scalebox{0.55}{\scri{$9$}}};
\node at (-0.5, 10) {\scalebox{0.55}{\scri{$10$}}};
\node at (-0.5, 9) {\scalebox{0.55}{\scri{$11$}}};
\node at (-0.5, 8) {\scalebox{0.55}{\scri{$12$}}};
\node at (-0.5, 7) {\scalebox{0.55}{\scri{$13$}}};
\node at (-0.5, 6) {\scalebox{0.55}{\scri{$14$}}};
\node at (-0.5, 5) {\scalebox{0.55}{\scri{$15$}}};
\node at (-0.5, 4) {\scalebox{0.55}{\scri{$16$}}};
\node at (-0.5, 3) {\scalebox{0.55}{\scri{$17$}}};
\node at (-0.5, 2) {\scalebox{0.55}{\scri{$18$}}};
\node at (-0.5,1) {\scalebox{0.55}{\scri{$19$}}};
%\draw (1, 19.4) -- (2, 19.4) -- (3, 19.4) -- (4, 19.4) -- (5, 19.4) -- (6, 19.4) -- (7, 19.4) -- (8, 19.4) -- (9, 19.4)  -- (10, 19.4)  -- (11, 19.4)  -- (12, 19.4)  -- (13, 19.4)  -- (14, 19.4)  -- (15, 19.4)  -- (16, 19.4)  -- (17, 19.4);
\end{tikzpicture}}
\end{minipage}}
\caption{All paths in $\mathscr{P}_{6,1}$ for $n=10$.}\label{path in grid2}
\end{figure}
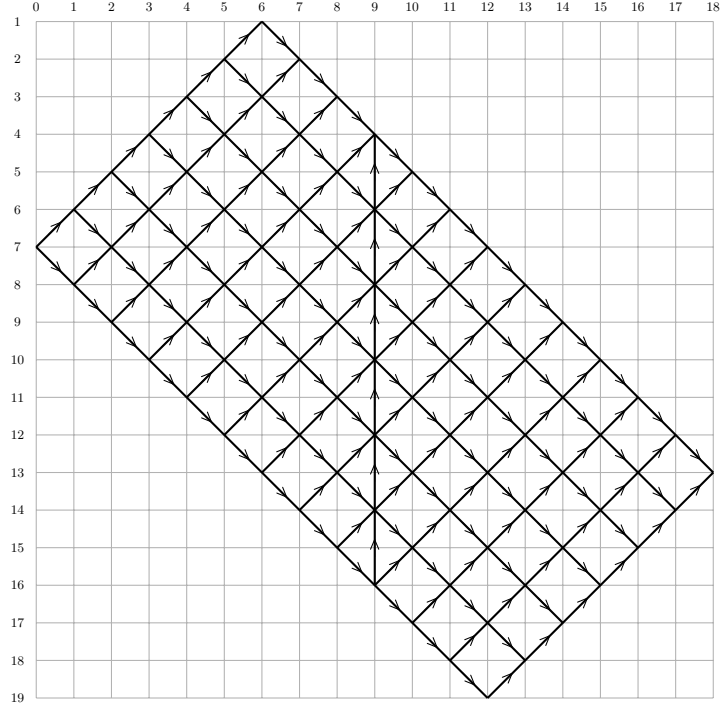

\subsection{Corners}

The sets $C_{p}^{\pm}$ of upper and lower corners of a path $p \in \mathscr{P}_{n-1,k}$ are defined as follows:
\begin{align*}
C^{+}_{p}=\{(r,y_{r})\in p \mid  r \in \{1,\ldots,n-2\}, \ y_{r-1}=y_{r}+1=y_{r+1} \}
\\\sqcup\{(n-1,y_{n-1})\in p \mid  y_{n-1}+1=y_{n-2}\},\\
C^{-}_{p}=\{(r,y_{r})\in p \mid  r\in \{1,\ldots,n-2\}, \ y_{r-1}=y_{r}-1=y_{r+1} \}\\
\sqcup\{(n-1,y_{n-1})\in p \mid  y_{n-1}-1=y_{n-2}\}.
\end{align*}

The sets $C_{p}^{\pm}$ of upper and lower corners of a path $p \in \mathscr{P}_{i,k}$, with $1\leq i \leq n-2$, are defined as follows:
\begin{align*}
C^{+}_{p} & = \{(n-1,y_{n-1})\in p \mid  y_{n-1}+1=y_{n-2} \,\,\text{or}\,\, y_{n-1}+1=y_{n}\}\\
& \sqcup \{(r,y_{r})\in p \mid  r\in \mathscr{S}\backslash \{0,n-1,N\}, \ y_{r-1}=y_{r}+1=y_{r+1}\},\\
C^{-}_{p} & = \{(n-1,y_{n-1})\in p\mid y_{n-1}-1=y_{n-2} \,\,\text{or}\,\,  y_{n-1}-1=y_{n}\}\\
& \sqcup\{(r,y_{r})\in p\mid  r\in \mathscr{S}\backslash \{0,n-1,N\}, \ y_{r-1}=y_{r}-1=y_{r+1}\}.
\end{align*}

Let $(i,k)\in \mathcal{X}$. By the definition of paths, any path $p \in \mathscr{P}_{i,k}$ has at least one lower corner or upper corner, which is unique defined by its set of lower or upper corners. There exists a unique path without lower (respectively, upper) corners in $\mathscr{P}_{i,k}$. The path without lower (respectively, upper) corners is called the highest (resp. lowest) path, denoted by $p^{+}_{i,k}$ (respectively, $p^{-}_{i,k}$).

\subsection{Moves of paths}

\subsubsection{Lowering moves}
Let $(i,k)\in\mathcal{X}$. A path $p\in \mathscr{P}_{i,k}$ is said to be lowered at $(j,\ell)$ if and only if $(j,\ell-1) \in C^+_p$ and $(j,\ell+1) \not\in C^+_p$, see \cite[section 5.2]{MY12a}. We denote by $p\mathscr{A}^{-1}_{j,\ell}$ the new path obtained by lowering move on $p$ at $(j,\ell)$.

We define lowering moves on paths by case-by-case. Firstly, we define lowering moves on paths in $\mathscr{P}_{n-1,k}$. For any $p\in \mathscr{P}_{n-1,k}$, we assume without loss of generality that
\[
p=((0,y_{0}),(1,y_{1}),\ldots,(j,y_{j}),\ldots,(n-1,y_{n-1})).
\]

\begin{enumerate}[(i)]
\item If $(j,y_j)\in C^{+}_{p}$ for some $j<n-1$, then $(j,y_j+2) \not\in C^+_p$ follows automatically, and $y_{j-1}=y_{j}+1=y_{j+1}=\ell$. We define
\[
p\mathscr{A}^{-1}_{j,y_j+1}:=((0,y_{0}),\ldots,(j-1,y_{j-1}),(j,y_{j}+2),(j+1,y_{j+1}),\ldots,(n-1,y_{n-1})).
\]
Obviously, $p\mathscr{A}^{-1}_{j,y_j+1} \in \mathscr{P}_{n-1,k}$.

\item If $(n-1,y_{n-1})\in C^+_{p}$, then $y_{n-2}=y_{n-1}+1=\ell$. We define
\[
p\mathscr{A}^{-1}_{n-1,y_{n-1}+1}:=((0,y_{0}),(1,y_{1}),\ldots,(n-2,y_{n-2}),(n-1,y_{n-1}+2)) \in \mathscr{P}_{n-1,k}.
\]
\end{enumerate}
Pictorially, the lowering moves of a path are depicted in Figure \ref{lowering moves of a path}.
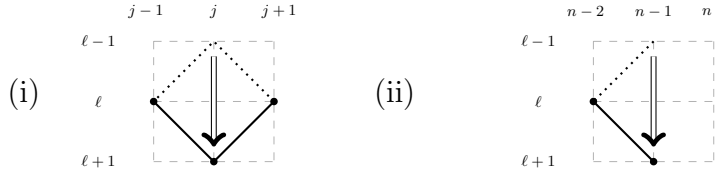
\begin{figure}[H]
\resizebox{1.0\width}{1.0\height}{
\begin{minipage}[b]{0.3\linewidth}
\centerline{
\text{(i) } \begin{tikzpicture}[baseline=0cm,scale=.8,yscale=-1]
\draw[help lines, color=gray!60, dashed] (-1,-1) grid (1,1);
\node at (-1, -1.5) {\scalebox{0.65}{\scri{$j-1$}}\,\,\,};
 \node at (0, -1.5) {\scalebox{0.65}{\scri{$j$}}};
 \node at (1, -1.5) {\scalebox{0.65}{\,\,\,\scri{$j+1$}}};
\node at (-2, -1) {\scalebox{0.65}{\,\,\,\scri{$\ell-1$}}};
\node at (-2, 0) {\scalebox{0.65}{\,\,\,\scri{$\ell$}}};
\node at (-2, 1) {\scalebox{0.65}{\,\,\,\scri{$\ell+1$}}};
\begin{scope}[every node/.style={minimum size=.1cm,inner sep=0mm,fill,circle}]
\draw[thick] (-1,0) node {} -- (0,1) node {} -- (1,0)node {} ;
\draw [thick,dotted] (1,0) -- (0,-1) -- (-1,0);
\end{scope}
\draw[double distance=.5mm,->,shorten <= 2mm,shorten >= 2mm] (0,-1) -- (0,1);
\end{tikzpicture}}
\end{minipage}
\quad\text{(ii)} \begin{minipage}[b]{0.3\linewidth}
\centerline{
\begin{tikzpicture}[baseline=0cm,scale=.8,yscale=-1]
\draw[help lines, color=gray!60, dashed] (-1,-1) grid (1,1);
\node at (-1, -1.5) {\scalebox{0.65}{\scri{$n-2$}}\,\,\,};
 \node at (0, -1.5) {\scalebox{0.65}{\scri{$n-1$}}};
 \node at (0.8, -1.5) {\scalebox{0.65}{\,\,\,\scri{$n$}}};
\node at (-2, -1) {\scalebox{0.65}{\,\,\,\scri{$\ell-1$}}};
\node at (-2, 0) {\scalebox{0.65}{\,\,\,\scri{$\ell$}}};
\node at (-2, 1) {\scalebox{0.65}{\,\,\,\scri{$\ell+1$}}};
\begin{scope}[every node/.style={minimum size=.1cm,inner sep=0mm,fill,circle}]
\draw[thick] (-1,0) node {} -- (0,1) node {} ;
\draw [thick,dotted](0,-1) -- (-1,0);
\end{scope}
\draw[double distance=.5mm,->,shorten <= 2mm,shorten >= 2mm] (0,-1) -- (0,1);
\end{tikzpicture}}
\end{minipage}}
\caption{Lowering moves of a path $p\in \mathscr{P}_{n-1,k}$.}\label{lowering moves of a path}
\end{figure}

We nextly define the lowering moves on paths in $\mathscr{P}_{i,k}$, with $i<n-1$. Assume without loss of generality that each path in $\mathscr{P}_{i,k}$ has the form (\ref{path form in pik}).

\begin{enumerate}[(i)]
\item If $(j,y_j)\in C^{+}_{p}$ for some $j\neq n-1$, then $y_{j-1}=y_{j}+1=y_{j+1}=\ell$, and we define
\[
p\mathscr{A}^{-1}_{j,y_j+1}:=((0,y_{0}),\ldots,(j-1,y_{j-1}),(j,y_{j}+2),(j+1,y_{j+1}),\ldots,(N,y_{N})).
\]
\item If $(n-1,y_{n-1})\in C^{+}_{p}$ and $\ell=y_{n-2}=y_{n-1}+1 \neq y_n$, we define
\[
p\mathscr{A}^{-1}_{n-1,y_{n-1}+1}:=((0,y_{0}),\ldots,(n-2,y_{n-2}),(n-1,y_{n-1}+2), (n-1,y'_{n-1}),\ldots, (N,y_N)).
\]
\item If $(n-1,y'_{n-1})\in C^{+}_{p}$, and $\ell=y_n=y'_{n-1}+1 \neq y_{n-2}$, we define
\[
p\mathscr{A}^{-1}_{n-1,y'_{n-1}+1}:=((0,y_{0}),\ldots,(n-1,y_{n-1}), (n-1,y'_{n-1}+2),(n,y_n),\ldots, (N,y_N)).
\]
\item If $(n-1,y_{n-1})\in C^{+}_{p}$, and $\ell=y_{n-2}= y_{n-1}+1 = y'_{n-1}+1 = y_{n}$, we define
\begin{align*}
p\mathscr{A}^{-1}_{n-1,y_{n-1}+1}:= & ((0,y_{0}),\ldots,(n-2,y_{n-2}),(n-1,y_{n-1}+2), (n-1,y'_{n-1}+2), \\
&  (n,y_n),\ldots, (N,y_N)).
\end{align*}
\end{enumerate}

Pictorially, the lowering moves are depicted in Figure \ref{lowering moves of a path 2}.
\begin{figure}[H]
\resizebox{1.0\width}{1.0\height}{
\begin{minipage}[b]{0.23\linewidth}
\centerline{
\begin{tikzpicture}[baseline=0cm,scale=.8,yscale=-1]
\draw[help lines, color=gray!60, dashed] (-1,-1) grid (1,1);
\node at (-1, -1.5) {\scalebox{0.65}{\scri{$j-1$}}\,\,\,};
 \node at (0, -1.5) {\scalebox{0.65}{\scri{$j$}}};
 \node at (1, -1.5) {\scalebox{0.65}{\,\,\,\scri{$j+1$}}};
\node at (-2, -1) {\scalebox{0.65}{\,\,\,\scri{$\ell-1$}}};
\node at (-2, 0) {\scalebox{0.65}{\,\,\,\scri{$\ell$}}};
\node at (-2, 1) {\scalebox{0.65}{\,\,\,\scri{$\ell+1$}}};
\node at (0, 2) {(i)};
\begin{scope}[every node/.style={minimum size=.1cm,inner sep=0mm,fill,circle}]
\draw[thick] (-1,0) node {} -- (0,1) node {} -- (1,0)node {} ;
\draw [thick,dotted] (1,0) -- (0,-1) -- (-1,0);
\end{scope}
\draw[double distance=.5mm,->,shorten <= 2mm,shorten >= 2mm] (0,-1) -- (0,1);
\end{tikzpicture}}
\end{minipage}
\begin{minipage}[b]{0.23\linewidth}
\centerline{
\begin{tikzpicture}[baseline=0cm,scale=.8,yscale=-1]
\draw[help lines, color=gray!60, dashed] (-1,-1) grid (1,1);
\node at (-1, -1.5) {\scalebox{0.65}{\scri{$n-2$}}\,\,\,};
 \node at (0, -1.5) {\scalebox{0.65}{\scri{$n-1$}}};
 \node at (1, -1.5) {\scalebox{0.65}{\,\,\,\scri{$n$}}};
\node at (-2, -1) {\scalebox{0.65}{\,\,\,\scri{$\ell-1$}}};
\node at (-2, 0) {\scalebox{0.65}{\,\,\,\scri{$\ell$}}};
\node at (-2, 1) {\scalebox{0.65}{\,\,\,\scri{$\ell+1$}}};
\node at (0, 2) {(ii)};
\begin{scope}[every node/.style={minimum size=.1cm,inner sep=0mm,fill,circle}]
\draw[thick] (-1,0) node {} -- (0,1) node {} ;
\draw [thick,dotted](0,-1) -- (-1,0);
\end{scope}
\draw[double distance=.5mm,->,shorten <= 2mm,shorten >= 2mm] (0,-1) -- (0,1);
\end{tikzpicture}}
\end{minipage}
\begin{minipage}[b]{0.23\linewidth}
\centerline{
\begin{tikzpicture}[baseline=0cm,scale=.8,yscale=-1]
\draw[help lines, color=gray!60, dashed] (-1,-1) grid (1,1);
\node at (-1, -1.5) {\scalebox{0.65}{\scri{$n-2$}}\,\,\,};
 \node at (0, -1.5) {\scalebox{0.65}{\scri{$n-1$}}};
 \node at (1, -1.5) {\scalebox{0.65}{\,\,\,\scri{$n$}}};
\node at (-2, -1) {\scalebox{0.65}{\,\,\,\scri{$\ell-1$}}};
\node at (-2, 0) {\scalebox{0.65}{\,\,\,\scri{$\ell$}}};
\node at (-2, 1) {\scalebox{0.65}{\,\,\,\scri{$\ell+1$}}};
\node at (0, 2) {(iii)};
\begin{scope}[every node/.style={minimum size=.1cm,inner sep=0mm,fill,circle}]
\draw[thick] (0,1) node {} -- (1,0)node {} ;
\draw [thick,dotted] (1,0) -- (0,-1);
\end{scope}
\draw[double distance=.5mm,->,shorten <= 2mm,shorten >= 2mm] (0,-1) -- (0,1);
\end{tikzpicture}}
\end{minipage}
\begin{minipage}[b]{0.23\linewidth}
\centerline{
\begin{tikzpicture}[baseline=0cm,scale=.8,yscale=-1]
\draw[help lines, color=gray!60, dashed] (-1,-1) grid (1,1);
\node at (-1, -1.5) {\scalebox{0.65}{\scri{$n-2$}}\,\,\,};
 \node at (0, -1.5) {\scalebox{0.65}{\scri{$n-1$}}};
 \node at (1, -1.5) {\scalebox{0.65}{\,\,\,\scri{$n$}}};
\node at (-2, -1) {\scalebox{0.65}{\,\,\,\scri{$\ell-1$}}};
\node at (-2, 0) {\scalebox{0.65}{\,\,\,\scri{$\ell$}}};
\node at (-2, 1) {\scalebox{0.65}{\,\,\,\scri{$\ell+1$}}};
\node at (0, 2) {(iv)};
\begin{scope}[every node/.style={minimum size=.1cm,inner sep=0mm,fill,circle}]
\draw[thick] (-1,0) node {} -- (0,1) node {} -- (1,0)node {} ;
\draw [thick,dotted] (1,0) -- (0,-1) -- (-1,0);
\end{scope}
\draw[double distance=.5mm,->,shorten <= 2mm,shorten >= 2mm] (0,-1) -- (0,1);
\end{tikzpicture}}
\end{minipage}}
\caption{Lowering moves of a path $p\in \mathscr{P}_{i,k}$ for $i<n-1$.}\label{lowering moves of a path 2}
\end{figure}
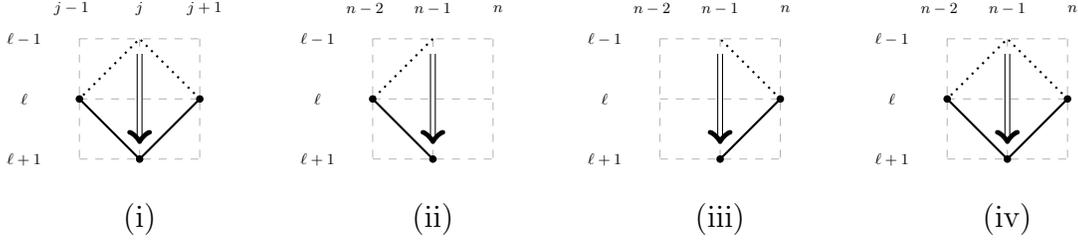

\subsubsection{Raising moves}

Following \cite[Section 5.3]{MY12a}, let $(i,k)\in\mathcal{X}$. A path $p\in \mathscr{P}_{i,k}$ is said to be raised at $(j,\ell)$ if and only if $p=p'\mathscr{A}^{-1}_{j,\ell}$ for some $p'\in \mathscr{P}_{i,k}$. It is unique if $p'$ exists, and we define $p':=p\mathscr{A}_{j,\ell}$. We can verify that $p$ can be raised at $(j,\ell)$ if and only if $(j,\ell+1)\in C^{-}_{p}$ and $(j,\ell-1)\notin C^{-}_{p}$.

\subsection{A lattice structure from paths} \label{a lattice structure from paths}
Let $(i,k)\in \mathcal{X}$. We assume that
\begin{align*}
& p=((0,y_0),(1,y_1),\ldots,(j,y_j),\ldots)\in \mathscr{P}_{i,k}, \\
& q=((0,z_0),(1,z_1),\ldots,(j,z_j),\ldots)\in \mathscr{P}_{i,k}.
\end{align*}
Define
\begin{gather}
\begin{align*}
& p \vee q = ((0,\min\{y_0,z_0\}),(1,\min\{y_1,z_1\}),\ldots,(j,\min\{y_j,z_j\}),\ldots), \\
& p \wedge q = ((0,\max\{y_0,z_0\}),(1,\max\{y_1,z_1\}),\ldots,(j,\max\{y_j,z_j\}),\ldots).
\end{align*}
\end{gather}
Obviously, both $p \vee q$ and $p \wedge q$ are paths in $\mathscr{P}_{i,k}$.

The set of all paths in $\mathscr{P}_{i,k}$ forms a lattice under the operators $\vee$ and $\wedge$. Let $p$ and $q$ be two paths in $\mathscr{P}_{i,k}$. We say that $p \prec q$ if and only if there exists a unique $(j,\ell)$ such that $p=q \mathscr{A}^{-1}_{j,\ell}$. The highest (respectively, lowest) path $p^{+}_{i,k}$ (respectively, $p^{-}_{i,k}$) is the maximum (respectively, minimum) element in $\mathscr{P}_{i,k}$ with respect to $\prec$.

\subsection{From paths to monomials or binomials}\label{from paths to monomials}

In the section, we assign a monomial or binomial to a path $p\in \mathscr{P}_{i,k}$, with $(i,k)\in \mathcal{X}$.

For $p\in \mathscr{P}_{n-1,k}$, we assume without loss of generality that
\[
p=((0,y_{0}),(1,y_{1}),\ldots,(j,y_{j}),\ldots,(n-1,y_{n-1})).
\]
Define a monomial $m(p)$ associated to $p$ as follows:
\begin{align}\label{the mononial ssociated to a path}
m(p)=\begin{cases}
Y_{f(y_{n-1}),y_{n-1}} \prod\limits_{\substack{(j,\ell)\in C^{+}_{p} \\ j \neq n-1}}Y_{j,\ell} \prod\limits_{\substack{(j,\ell)\in C^{-}_{p} \\ j\neq n-1}}Y^{-1}_{j,\ell}  &  \text{if $(n-1,y_{n-1})\in C^+_p$},  \\
Y^{-1}_{g(y_{n-1}),y_{n-1}} \prod\limits_{ \substack{(j,\ell)\in C^{+}_{p} \\ j \neq n-1}}Y_{j,\ell} \prod\limits_{\substack{(j,\ell)\in C^{-}_{p} \\ j\neq n-1}}Y^{-1}_{j,\ell}  &  \text{if $(n-1,y_{n-1})\in C^-_p$},
\end{cases}
\end{align}
where $f(y_{n-1})= \begin{cases} n-1 & \text{if $y_{n-1}-k \equiv 0 \pmod 4$}, \\  n & \text{if $y_{n-1}-k \equiv 2 \pmod 4$},  \end{cases}$ and
\[
g(y_{n-1})= \begin{cases} n & \text{if $y_{n-1}-k \equiv 0 \pmod 4$}, \\  n-1 & \text{if $y_{n-1}-k \equiv 2 \pmod 4$}. \end{cases}
\]

For $p\in \mathscr{P}_{i,k}$ and $i< n-1$, define  
\begin{align*}
m(p)=Z \prod_{\substack{(j,\ell)\in C^{+}_{p} \\ j\neq n-1} }Y_{\overline{j},\ell}\prod_{\substack{(j,\ell)\in C^{-}_{p}, \\ j\neq n-1}}Y^{-1}_{\overline{j},\ell},
\end{align*}
where $Z$ is defined as follows:
\begin{align} \label{monomials associated to paths1}
Z=\begin{cases} 
Y_{n-1,\ell_{1}}Y^{-1}_{n,\ell_{2}}+Y_{n,\ell_{1}}Y^{-1}_{n-1,\ell_{2}} &  \text{if $p$ travels (1) in Figure \ref{paths and monomials}, $\ell_2-\ell_1 \equiv 2 \hspace{-0.3cm} \pmod 4$}, \\
Y_{n-1,\ell_{1}}Y^{-1}_{n-1,\ell_{2}}+Y_{n,\ell_{1}}Y^{-1}_{n,\ell_{2}} &  \text{if $p$ travels (1) in Figure \ref{paths and monomials}, $\ell_2-\ell_1 \equiv 0 \hspace{-0.3cm} \pmod 4$}, \\
Y^{-1}_{n-1,\ell_{1}}Y^{-1}_{n,\ell_{2}}+Y^{-1}_{n,\ell_{1}}Y^{-1}_{n-1,\ell_{2}} &  \text{if $p$ travels (2) in Figure \ref{paths and monomials}, $\ell_2-\ell_1 \equiv 0 \hspace{-0.3cm} \pmod 4$}, \\
Y^{-1}_{n-1,\ell_{1}}Y^{-1}_{n-1,\ell_{2}}+Y^{-1}_{n,\ell_{1}}Y^{-1}_{n,\ell_{2}} &  \text{if $p$ travels (2) in Figure \ref{paths and monomials}, $\ell_2-\ell_1 \equiv 2 \hspace{-0.3cm} \pmod 4$}, \\
Y_{n-1,\ell_{1}}Y_{n,\ell_{2}}+Y_{n,\ell_{1}}Y_{n-1,\ell_{2}} & \text{if $p$ travels (3) in Figure \ref{paths and monomials}, $\ell_2-\ell_1 \equiv 0 \hspace{-0.3cm} \pmod 4$}, \\
Y_{n-1,\ell_{1}}Y_{n-1,\ell_{2}}+Y_{n,\ell_{1}}Y_{n,\ell_{2}} & \text{if $p$ travels (3) in Figure \ref{paths and monomials}, $\ell_2-\ell_1 \equiv 2 \hspace{-0.3cm} \pmod 4$}, \\
Y^{-1}_{n-1,\ell_{1}}Y_{n,\ell_{2}}+Y^{-1}_{n,\ell_{1}}Y_{n-1,\ell_{2}}  & \text{if $p$ travels (4) in Figure \ref{paths and monomials}, $\ell_2-\ell_1 \equiv 2 \hspace{-0.3cm} \pmod 4$}, \\
Y^{-1}_{n-1,\ell_{1}}Y_{n-1,\ell_{2}}+Y^{-1}_{n,\ell_{1}}Y_{n,\ell_{2}}  & \text{if $p$ travels (4) in Figure \ref{paths and monomials}, $\ell_2-\ell_1 \equiv 0 \hspace{-0.3cm} \pmod 4$}, \\
Y_{n-1,\ell-1}Y_{n,\ell-1} & \text{if $p$ travels (5) in Figure \ref{paths and monomials}}, \\
Y^{-1}_{n-1,\ell}Y^{-1}_{n,\ell} & \text{if $p$ travels (6) in Figure \ref{paths and monomials}}.
\end{cases}
\end{align}

\begin{figure}[H]
\resizebox{1.0\width}{1.0\height}{
\begin{minipage}[b]{0.15\linewidth}
\centerline{
\begin{tikzpicture}[scale=.5]
\draw[help lines, color=gray!60, dashed] (-1,0) grid (1,4);
\begin{scope}[thick, every node/.style={sloped,allow upside down}]
\draw (-1,1)-- node {\midarrow}(0,0);
\draw[dashed] (0,0)--node {\midarrow}(0,4);
\draw (0,4)--node {\midarrow}(1,3);
\end{scope}
\node at (0, 4.5) {\scalebox{0.55}{\scri{$n-1$}}};
\node at (-1.5, 0) {\scalebox{0.55}{\,\,\,\scri{$\ell_{2}$}}};
\node at (-1.5, 2) {\scalebox{0.55}{\,\,\,\scri{$\vdots$}}};
\node at (-1.5, 4) {\scalebox{0.65}{\,\,\,\scri{$\ell_{1}$}}};
\node at (0, -1) {(1)};
\end{tikzpicture}}
\end{minipage}
\begin{minipage}[b]{0.15\linewidth}
\centerline{
\begin{tikzpicture}[scale=.5]
\draw[help lines, color=gray!60, dashed] (-1,0) grid (1,4);
\begin{scope}[thick, every node/.style={sloped,allow upside down}]
\draw (-1,1)--node {\midarrow}(0,0);
\draw[dashed] (0,0)--node {\midarrow}(0,3);
\draw (0,3)--node {\midarrow}(1,4);
\end{scope}
\node at (0, 4.5) {\scalebox{0.55}{\scri{$n-1$}}};
 \node at (-1.5, 0) {\scalebox{0.55}{\,\,\,\scri{$\ell_{2}$}}};
\node at (-1.5, 1.5) {\scalebox{0.55}{\,\,\,\scri{$\vdots$}}};
\node at (-1.5, 3) {\scalebox{0.65}{\,\,\,\scri{$\ell_{1}$}}};
\node at (0, -1) {(2)};
\end{tikzpicture}}
\end{minipage}
\begin{minipage}[b]{0.15\linewidth}
\centerline{
\begin{tikzpicture}[scale=.5]
\draw[help lines, color=gray!60, dashed] (-1,0) grid (1,4);
\begin{scope}[thick, every node/.style={sloped,allow upside down}]
\draw (-1,0)--node {\midarrow}(0,1);
\draw[dashed] (0,1)--node {\midarrow}(0,4);
\draw (0,4)--node {\midarrow}(1,3);
\end{scope}
\node at (0, 4.5) {\scalebox{0.55}{\scri{$n-1$}}};
 \node at (-1.5, 1) {\scalebox{0.55}{\,\,\,\scri{$\ell_{2}$}}};
\node at (-1.5, 2.5) {\scalebox{0.55}{\,\,\,\scri{$\vdots$}}};
\node at (-1.5, 4) {\scalebox{0.65}{\,\,\,\scri{$\ell_{1}$}}};
\node at (0, -1) {(3)};
\end{tikzpicture}}
\end{minipage}
\begin{minipage}[b]{0.15\linewidth}
\centerline{
\begin{tikzpicture}[scale=.5]
\draw[help lines, color=gray!60, dashed] (-1,0) grid (1,4);
\begin{scope}[thick, every node/.style={sloped,allow upside down}]
\draw (-1,0)--node {\midarrow}(0,1);
\draw[dashed] (0,1)--node {\midarrow}(0,3);
\draw (0,3)--node {\midarrow}(1,4);
\end{scope}
\node at (0, 4.5) {\scalebox{0.55}{\scri{$n-1$}}};
 \node at (-1.5, 1) {\scalebox{0.55}{\,\,\,\scri{$\ell_{2}$}}};
\node at (-1.5, 2) {\scalebox{0.55}{\,\,\,\scri{$\vdots$}}};
\node at (-1.5, 3) {\scalebox{0.65}{\,\,\,\scri{$\ell_{1}$}}};
\node at (0, -1) {(4)};
\end{tikzpicture}}
\end{minipage}
\begin{minipage}[b]{0.15\linewidth}
\centerline{
\begin{tikzpicture}[scale=.5]
\draw[help lines, color=gray!90, dashed] (0,0) grid (2,2);
\begin{scope}[thick, every node/.style={sloped,allow upside down}]
\draw (0,1) -- node {\midarrow}(1,2);
\draw (1,2) -- node {\midarrow}(2,1);;
\end{scope}
\node at (1, 2.35) {\scalebox{0.75}{\scri{$n-1$}}};
\node at (-0.4, 1) {\scalebox{0.75}{\,\,\,\scri{$\ell$}}};
\node at (-0.8, 2) {\scalebox{0.75}{\,\,\,\scri{$\ell-1$}}};
\node at (1,-1) {(5)};
\end{tikzpicture}}
\end{minipage}
\begin{minipage}[b]{0.15\linewidth}
\centerline{
\begin{tikzpicture}[scale=.5]
\draw[help lines, color=gray!90, dashed] (0,0) grid (2,2);
\begin{scope}[thick, every node/.style={sloped,allow upside down}]
\draw (0,2) -- node {\midarrow}(1,1);
\draw (1,1) -- node {\midarrow}(2,2);
\end{scope}
\node at (1, 2.35) {\scalebox{0.75}{\scri{$n-1$}}};
\node at (-0.4, 1) {\scalebox{0.75}{\,\,\,\scri{$\ell$}}};
\node at (-0.8, 2) {\scalebox{0.75}{\,\,\,\scri{$\ell-1$}}};
\node at (1,-1) {(6)};
\end{tikzpicture}}
\end{minipage}}
\caption{All possible ways of a path $p$ travelling $x=n-1$ axis.}\label{paths and monomials}
\end{figure}
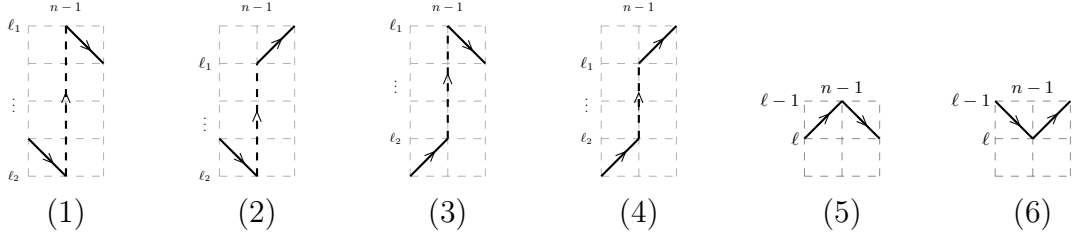

A map $m$ sending paths in $\mathscr{P}_{i,k}$ to Laurent polynomials is defined by
\begin{align*}
m: \mathscr{P}_{i,k} & \longrightarrow \mathbb{Z}[Y^{\pm1}_{j,\ell} | (j,\ell) \in I \times \mathbb{Z}]   \\
p & \longmapsto  m(p).
\end{align*}

We always identify a path $p$ with  $m(p)$.

\section{A combinatorial approach to dimensions of fundamental modules of type $D_{n}$} \label{Dimensions of fundamental modules}
In this section, we compute the number of monomials (including multiplicities) associated to paths in $\mathscr{P}_{i,k}$, which is proved to be the same with the dimension of the fundamental module $L(Y_{i,k})$ of type $D_n$. The number does not depend on the choice of the parameter $k$, so we assume without loss of generality that $k=0$.

We agree that $\sum_{j=0}^{k} \binom{n}{j}=0$ if $k<0$ and $n\geq 0$, and extend the definition of $\mathscr{P}_{i,k}$ to the domain $\{(i,k) \in \mathscr{S} \times \mathbb{Z} \mid  i-k\equiv 0 \hspace{-0.1cm} \pmod  2 \}$. The following lemma counts the number of paths in $\mathscr{P}_{i,0}$, with $i\in I$.

\begin{lemma}\label{the numbers of paths}
Suppose that $\mathscr{P}_{i,0}$, with $i\in I$, is the set of paths defined in Section \ref{defintion of paths}.
\begin{itemize}
\item[(1)] For $i<n-1$, we have
\begin{equation*}\label{dimension1}
|\mathscr{P}_{i,0}|=\sum^{i}_{j=0} \sum^{i-j}_{l=0} \binom{n-1}{j} \binom{n-1}{l}.
\end{equation*}

\item[(2)] For $i=n-1$, we have
\begin{equation*}\label{dimension2}
|\mathscr{P}_{i,0}|=2^{n-1}.
\end{equation*}
\end{itemize}

\begin{proof}

(1) Assume without loss of generality that each path in $\mathscr{P}_{i,0}$ has the form (\ref{path form in pik}). Let $a$ denote the point $(0,i)$ and $b$ denote the point $(N,N-i)$. Obviously, $a$ (respectively, $b$) is the leftmost (respectively, rightmost) point of any path in $\mathscr{P}_{i,0}$.

The number of points in $\mathscr{P}_{i,0} \cap (x=n-1)$ is $(i+1)$, and we assume without loss of generality that these points are $p_0,p_1, \ldots,p_i$ in the descending order of vertical coordinates. The number of paths from $a$ to $p_j$ ($0\leq j \leq i$) in $\mathscr{P}_{i,0} \cap (x\leq n-1)$ is $\binom{n-1}{j}$, and the number of paths from $p_l$ ($0\leq l \leq i$) to $b$ in $\mathscr{P}_{i,0} \cap (x\geq n-1)$ is $\binom{n-1}{i-l}$. Every path in $\mathscr{P}_{i,0}$ starts with $a$, goes through two points $p_j$ and $p_l$ ($j\leq l$) in order, and ends with $b$. When we fix $j$, the $l$ can take any value in $\{j,j+1,\ldots,i\}$. So
\begin{equation*}
|\mathscr{P}_{i,0}|=\sum^{i}_{j=0} \sum^{i-j}_{l=0} \binom{n-1}{j} \binom{n-1}{l}.
\end{equation*}

(2) Let $a$ denote the point $(0,n-1)$. Obviously, $a$ is the leftmost point of any path in $\mathscr{P}_{n-1,0}$. The number of points in $\mathscr{P}_{n-1,0} \cap (x=n-1)$ is $n$, and we assume without loss of generality that these points are $p_0,p_1, \ldots, p_{n-1}$ in the descending order of vertical coordinates. The number of paths from $a$ to $p_j$ ($0\leq j \leq n-1$) in $\mathscr{P}_{i,0}$ is $\binom{n-1}{j}$. Let $j$ take all values in $\{0,1,\ldots,n-1\}$, our result follows.

\end{proof}
\end{lemma}

Recall that we assign monomials to each path in $\mathscr{P}_{i,0}$, with $i\in I$, see Section \ref{from paths to monomials}. Let $\mathcal{M}(\mathscr{P}_{i,0})$ be the set of monomials associated to paths in $\mathscr{P}_{i,0}$.

The following Lemma records the cardinality of the set $\mathcal{M}(\mathscr{P}_{i,0})$.
\begin{theorem}\label{the number of monomials}
Under the assumption of Lemma \ref{the numbers of paths},
\begin{itemize}
\item[(1)] For $i<n-1$, we have
\begin{equation} \label{dimension3}
|\mathcal{M}(\mathscr{P}_{i,0})|=\binom{2n-2}{i}+2\sum^{i}_{j=0} \sum^{i-j-1}_{l=0} \binom{n-1}{j} \binom{n-1}{l}.
\end{equation}
\item[(2)] For $i=n-1$, we have
\begin{equation}\label{dimension4}
|\mathcal{M}(\mathscr{P}_{i,0})|=2^{n-1}.
\end{equation}
\end{itemize}
\begin{proof}

(1) Keep the assumptions and notation in the proof of Lemma \ref{the numbers of paths}~(1). Let $p$ be a path in $\mathscr{P}_{i,0}$ starting with $a$, travelling two points $p_j$ and $p_l$ ($j\leq l$) in order, and ending with $b$. We assign binomials to the $p$ if $j<l$ and assign one monomial to the $p$ if $j=l$, see Equation (\ref{monomials associated to paths1}). The number of monomials for $j=l$ is
\[
\binom{2n-2}{i}=\sum^{i}_{j=0} \binom{n-1}{j} \binom{n-1}{i-j},
\]
and the number of monomials for $j<l$ is
\[
2\sum^{i}_{j=0} \sum^{i-j-1}_{l=0} \binom{n-1}{j} \binom{n-1}{l}.
\]
The cardinality $|\mathcal{M}(\mathscr{P}_{i,0})|$ is a sum of the two cases above.

(2) Every path in $\mathscr{P}_{n-1,0}$ is assigned to a monomial, see Equation (\ref{the mononial ssociated to a path}). Our result follows directly from Lemma \ref{the numbers of paths}~(2).
\end{proof}
\end{theorem}

Following \cite[Chapter 10]{CP94}, for any $\lambda\in P$, the Verma module $M_q(\lambda)$ is defined as the quotient of $U_q (\mathfrak{g})$ by the left ideal generated by $x^+_{i}$ and $k_i-q^{(\lambda,\omega_i)}$, for $i\in I$. It is obvious that $M(\lambda)$ is a highest weight $U_q (\mathfrak{g})$-module with the highest weight $\lambda$, which has a unique simple quotient $V_q(\lambda)$. Moreover, every simple highest weight $U_q (\mathfrak{g})$-module with the highest weight $\lambda$ is isomorphic to $V_q(\lambda)$. If $\lambda=\omega_{i}$, then $V_q(\omega_{i})$ is called a fundamental module of $U_q (\mathfrak{g})$.

It is well-known that finite-dimensional $U_q (\mathfrak{g})$-modules of type 1 have one-to-one correspondence with finite-dimensional $\mathfrak{g}$-modules, see \cite{CP94,Lus93}, they have the same characters, and hence the same dimension.

Let $\text{Rep}(U_q (\mathfrak{g}))$ be the Grothendieck ring of the category of finite-dimensional $U_q(\mathfrak{g})$-modules. There is a characteristic homomorphism
\[
\chi: \text{Rep}(U_q (\mathfrak{g})) \rightarrow \mathbb{Z}[y^{\pm1}_{i}]_{i\in I},
\]
where $y_{i}$ is the function corresponding to the character of the fundamental module $V_q(\omega_{i})$. One of the properties of $q$-character $\chi_{q}$ is that if we replace each $Y^{\pm1}_{i,a}$ by $y^{\pm1}_{i}$ in $\chi_{q}([V])$, where $V$ is a $U_q(\widehat{\mathfrak{g}})$-module, then we obtain the character $\chi(V|_{U_q (\mathfrak{g})})$ of $V$ as a $U_q(\mathfrak{g})$-module.

In \cite[Theorem 6.8]{CP95a}, Chari and Pressley gave the $U_q (\mathfrak{g})$-structure of most of the fundamental $U_q(\widehat{\mathfrak{g}})$-modules. Denote by $V|_{U_q(\mathfrak{g})}$ (respectively, $V|_{U_q(\widehat{\mathfrak{g}})}$) the $U_q(\mathfrak{g})$-structure (respectively, $U_q(\widehat{\mathfrak{g}})$-structure) of $V$. For the Lie algebra $\mathfrak{g}$ of type $D_n$, the Chari-Pressley decomposition is as follows:
 \begin{itemize}
\item[$(1)$] For $i\in\{1,n-1,n\}$, $L(Y_{i,k})|_{U_q (\mathfrak{g})} \cong V_q(\omega_i)$ (independent on the choice of $k$).
\item[$(2)$] For $1<i<n-1$, $L(Y_{i,k})|_{U_q (\mathfrak{g})} \cong \bigoplus^{\lfloor\frac{i}{2}\rfloor}_{j=0}V_q(\omega_{i-2j})$ (independent on the choice of $k$).
\end{itemize}
Here for any integer $i$, $\lfloor i \rfloor$ is the greatest integer less than or equal to $i$. By the Weyl dimension formula in \cite[Chapter 6, Section 24]{Hum78}, for $i\in \{1,2\ldots,n-2\}$, the dimension of the $i$-th fundamental module $V_q(\omega_i)$ is $\binom{2n}{i}$, and  $V_q(\omega_{n-1})$ and $V_q(\omega_n)$ have the same dimension, which equals $2^{n-1}$. These results were explicitly computed in \cite[Proposition 13.10]{Car05}. Hence
\begin{itemize}
\item[(1)] for $i\in\{n-1,n\}$, $\text{dim}(L(Y_{i,0})|_{U_q (\mathfrak{g})})=2^{n-1}$,
\item[(2)] for $1\leq i<n-1$, $\text{dim}(L(Y_{i,0})|_{U_q (\mathfrak{g})})=\sum^{\lfloor\frac{i}{2}\rfloor}_{j=0}\binom{2n}{i-2j}$.
\end{itemize}

The following corollary will be particularly useful in the sequel.

\begin{corollary}\label{a new expression on dimensions of fundamental modules}
For any $i\in I$, we have
\begin{align}\label{equation1}
\text{dim}(L(Y_{i,0})|_{U_q (\mathfrak{g})})=|\mathcal{M}(\mathscr{P}_{i,0})|.
\end{align}
\end{corollary}
\begin{proof}
For $i\in\{n-1,n\}$, our result directly follows from Theorem \ref{the number of monomials}~(2). The rest of the proof is to show that our equation holds for $1\leq i<n-1$.

Suppose that $i$ is odd. We prove Equation (\ref{equation1}) by the induction on $i$. If $i=1$, then the term at the left hand side of Equation (\ref{equation1}) equals $\binom{2n}{1}=2n$, and the term at the right hand side of Equation (\ref{equation1}) equals $\binom{2n-2}{1}+2\binom{n-1}{0}\binom{n-1}{0}=2n$. So Equation (\ref{equation1}) holds. Suppose that Equation (\ref{equation1}) holds for $i$. We prove it for $i+2$.
\begin{align*}
\text{dim}(L(Y_{i+2,0})|_{U_q (\mathfrak{g})}) & = \binom{2n}{1}+\binom{2n}{3}+\cdots+\binom{2n}{i}+\binom{2n}{i+2} \\
& = \binom{2n-2}{i}+2\sum^{i}_{j=0} \sum^{i-j-1}_{l=0} \binom{n-1}{j} \binom{n-1}{l} + \binom{2n}{i+2},
\end{align*}
where the last equation follows from our induction.

By the combination formula $\binom{n+1}{m}=\binom{n}{m}+\binom{n}{m-1}$, we have
\[
\binom{2n}{i+2} = \binom{2n-2}{i+2} + 2 \binom{2n-2}{i+1} + \binom{2n-2}{i}.
\]
Hence
\begin{gather}
\begin{align*}
\text{dim}(L(Y_{i+2,0})|_{U_q (\mathfrak{g})}) = \binom{2n-2}{i+2} +2\sum^{i}_{j=0} \sum^{i-j-1}_{l=0} \binom{n-1}{j} \binom{n-1}{l}+ 2 \binom{2n-2}{i+1}+2\binom{2n-2}{i}.
\end{align*}
\end{gather}

On the other hand, by Theorem \ref{the number of monomials}~(1),

\begin{align*}
|\mathcal{M} & (\mathscr{P}_{i+2,0})| = \binom{2n-2}{i+2} + 2\sum^{i+2}_{j=0} \sum^{i-j+1}_{l=0} \binom{n-1}{j} \binom{n-1}{l} \\
& =\binom{2n-2}{i+2} +2\sum^{i}_{j=0} \sum^{i-j-1}_{l=0} \binom{n-1}{j} \binom{n-1}{l} \\
&+2\left( \binom{n-1}{i+1}\binom{n-1}{0}+\binom{n-1}{i}\binom{n-1}{1}+\cdots+\binom{n-1}{0}\binom{n-1}{i+1} \right)  \\
&+2\left( \binom{n-1}{i}\binom{n-1}{0}+\binom{n-1}{i-1}\binom{n-1}{1}+\cdots+\binom{n-1}{0}\binom{n-1}{i} \right) \\
& = \binom{2n-2}{i+2} +2\sum^{i}_{j=0} \sum^{i-j-1}_{l=0} \binom{n-1}{j} \binom{n-1}{l}+ 2 \binom{2n-2}{i+1}+2\binom{2n-2}{i}.
\end{align*}
This completes the induction step.

By the same argument with the case where $i$ is odd, we can prove Equation (\ref{equation1}) for an even number $i$. The proof is complete.

%%%When $i$ is even number, we prove the equation by induction on $i$. If $i=2$, then the left side of the equation (\ref{equation1}) is equals to $\binom{2n}{0}+\binom{2n}{2}=2n^2-n+1$, and the right side of the equation (\ref{equation1})
%%%is equals to $\binom{2n-2}{2}+2(\binom{n-1}{0}\binom{n-1}{0}+\binom{n-1}{0}\binom{n-1}{1}+\binom{n-1}{1}\binom{n-1}{0})=2n^2-n+1$. The statement is true.
%%%
%%%Assume that the statement is true for $i$. We prove it for $i+2$.
%%%\begin{align*}
%%%&\binom{2n}{0}+\binom{2n}{2}+\cdots+\binom{2n}{i}+\binom{2n}{i+2} \\
%%%&=\binom{2n-2}{i}+2\sum^{i}_{j=0} \sum^{i-j-1}_{l=0} \binom{n-1}{j} \binom{n-1}{l}+\binom{2n}{i+2} \quad \text{(by the induction)}\\
%%%&=\binom{2n-2}{i+2} +2\sum^{i}_{j=0} \sum^{i-j-1}_{l=0} \binom{n-1}{j} \binom{n-1}{l}\\
%%%&+2(\binom{2n-2}{i+1}+\binom{2n-2}{i})
%%%\quad \text{(by $\binom{n+1}{m}=\binom{n}{m}+\binom{n}{m-1}$)}   \\
%%%& =\binom{2n-2}{i+2} +2\sum^{i}_{j=0} \sum^{i-j-1}_{l=0} \binom{n-1}{j} \binom{n-1}{l}\\
%%%&+2(\binom{n-1}{i+1}\binom{n-1}{0}+\binom{n-1}{i}\binom{n-1}{1}+\cdots+\binom{n-1}{0}\binom{n-1}{i+1})  \\
%%%&+2(\binom{n-1}{i}\binom{n-1}{0}+\binom{n-1}{i-1}\binom{n-1}{1}+\cdots+\binom{n-1}{0}\binom{n-1}{i}) \\
%%%&\text{(by $\sum^{k}_{j=0}\binom{n-m}{k-j}\binom{m}{j}=\binom{n}{k}$)}\\
%%%& =\binom{2n-2}{i+2} +2\sum^{i+2}_{j=0} \sum^{i-j+1}_{l=0} \binom{n-1}{j} \binom{n-1}{l}.
%%%\end{align*}
%%%This completes the induction step.
\end{proof}

\section{A combinatorial formula for $q$-characters of fundamental modules}\label{Path description for $q$-characters of fundamental modules}

In this section, we give a combinatorial algorithm for the $q$-characters of fundamental modules of type $D_n$. The $q$-character of the fundamental module $\chi_{q}([L(Y_{n,k})])$ can be obtained from $\chi_{q}([L(Y_{n-1,k})])$ by switching $Y_{n-1,\ell}$ with $Y_{n,\ell}$. So it is enough to investigate the behavior of the monomials in $\chi_{q}([L(Y_{i,k})])$, with $i\leq n-1$.

Let $(i,k)\in \mathcal{X}$. There exists a unique dominant (respectively, anti-dominant) monomial $Y_{i,k}$ (respectively, $Y^{-1}_{i^*,2n-2+k}$) in $\{ m(p) \mid p \in \mathscr{P}_{i,k}\}$, where $i^*$ is defined by $w_0(\alpha_i)=-\alpha_{i^*}$ for the longest element $w_0$ in the Weyl group of type $D_n$.

Now everything is in the place for our main theorem.

\begin{theorem}\label{path formula for fundamental modules}
For $(i,k) \in \mathcal{X}$, we have
\begin{align*}
\chi_{q}([L(Y_{i,k})])=\sum_{p \in \mathscr{P}_{i,k}} m(p).
\end{align*}
\end{theorem}

\begin{proof}
We fix $j\in I$. Following the proof of Theorem 4.3 in \cite{J22}, the set $\mathscr{P}_{i,k}$ can be refined as a disjoint union of the connected components with respect to lowering moves or raising moves at $(u,\ell)$ with $j \in \overline{u}$ for any $\ell\in \mathbb{Z}$. Let $C$ be a $j$-connected component of $\mathscr{P}_{i,k}$, and denote by $|C|$ the number of paths in $C$.

{\bf Case 1.} Assume that $|C|=1$. The path $p$ in $C$ has no upper or lower corner at $(u,\ell)$ with $j\in \overline{u}$ for any $\ell\in \mathbb{Z}$, which implies that $m(p)$ has no any factor $Y^{\pm 1}_{j,\ell}$. By the Leibniz rule of the $j$-th screening operator $S_{j}$, we have $m(p)\in \text{ker}(S_j)$ .

{\bf Case 2.} Assume that $|C|=2$. Let $p_{1}$ and $p_{2}$ be the two paths in $C$. Since $C$ is a $j$-connected component of $\mathscr{P}_{i,k}$, we have either $p_{2}=p_{1}\mathscr{A}^{-1}_{u,\ell}$ or  $p_{1}=p_{2}\mathscr{A}^{-1}_{u,\ell}$  with $j\in \overline{u}$ for some $\ell\in \mathbb{Z}$. We assume without loss of generality that $p_{2}=p_{1}\mathscr{A}^{-1}_{u,\ell}$. The local configurations of $p_1$ and $p_2$ near by $(u,\ell)$ are depicted in Figures \ref{local configuration of q (left) and p 1}--\ref{local configuration of q (left) and p 3}, and the other parts of $p_1$ and $p_2$ are the same.
\begin{figure}[H]
\resizebox{1.0\width}{1.0\height}{
\begin{minipage}[b]{0.5\linewidth}
\centerline{
\begin{tikzpicture}[baseline=0cm,scale=.8,yscale=-1]
\draw[help lines, color=gray!60, dashed] (-1,-1) grid (1,1);
\node at (-1, -1.5) {\scalebox{0.55}{\scri{$u-1$}}\,\,\,};
 \node at (0, -1.5) {\scalebox{0.55}{\scri{$u$}}};
 \node at (1, -1.5) {\scalebox{0.55}{\,\,\,\scri{$u+1$}}};
\node at (-2, -1) {\scalebox{0.65}{\,\,\,\scri{$\ell-1$}}};
\node at (-2, 0) {\scalebox{0.65}{\,\,\,\scri{$\ell$}}};
\node at (-2, 1) {\scalebox{0.65}{\,\,\,\scri{$\ell+1$}}};
\begin{scope}[every node/.style={minimum size=.1cm,inner sep=0mm,fill,circle}]
\draw[thick] (1,0) node {} -- (0,-1) node {} -- (-1,0) node {};
\end{scope}
\end{tikzpicture}}
\end{minipage}
\hspace{-20mm}
\begin{minipage}[b]{0.5\linewidth}
\centerline{
\begin{tikzpicture}[baseline=0cm,scale=.8,yscale=-1]
\draw[help lines, color=gray!60, dashed] (-1,-1) grid (1,1);
\node at (-1, -1.5) {\scalebox{0.65}{\scri{$u-1$}}\,\,\,};
 \node at (0, -1.5) {\scalebox{0.65}{\scri{$u$}}};
 \node at (1, -1.5) {\scalebox{0.65}{\,\,\,\scri{$u+1$}}};
\node at (-2, -1) {\scalebox{0.65}{\,\,\,\scri{$\ell-1$}}};
\node at (-2, 0) {\scalebox{0.65}{\,\,\,\scri{$\ell$}}};
\node at (-2, 1) {\scalebox{0.65}{\,\,\,\scri{$\ell+1$}}};
\begin{scope}[every node/.style={minimum size=.1cm,inner sep=0mm,fill,circle}]
\draw[thick] (-1,0) node {} -- (0,1) node {} -- (1,0)node {};
\end{scope}
\end{tikzpicture}}
\end{minipage}}
\caption{The local configuration of $p_1$ (left) and $p_2$ (right) near by $(u,\ell)$ for $u<n-1$.}\label{local configuration of q (left) and p 1}
\end{figure}
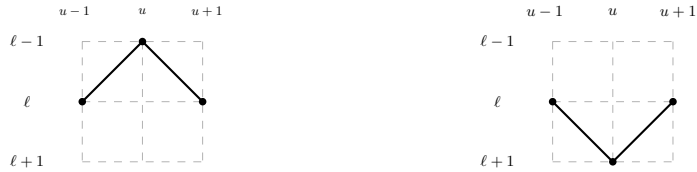

\begin{figure}[H]
\resizebox{1.0\width}{1.0\height}{
\begin{minipage}[b]{0.5\linewidth}
\centerline{
\begin{tikzpicture}[baseline=0cm,scale=.8,yscale=-1]
\draw[help lines, color=gray!60, dashed] (-1,-1) grid (1,1);
\node at (-1, -1.5) {\scalebox{0.55}{\scri{$n-2$}}\,\,\,};
 \node at (0, -1.5) {\scalebox{0.55}{\scri{$n-1$}}};
 \node at (1, -1.5) {\scalebox{0.55}{\,\,\,\scri{$n$}}};
\node at (-2, -1) {\scalebox{0.65}{\,\,\,\scri{$\ell-1$}}};
\node at (-2, 0) {\scalebox{0.65}{\,\,\,\scri{$\ell$}}};
\node at (-2, 1) {\scalebox{0.65}{\,\,\,\scri{$\ell+1$}}};
\begin{scope}[every node/.style={minimum size=.1cm,inner sep=0mm,fill,circle}]
\draw[thick]  (0,-1) node {} -- (-1,0) node {};
\end{scope}
\end{tikzpicture}}
\end{minipage}
\hspace{-20mm}
\begin{minipage}[b]{0.5\linewidth}
\centerline{
\begin{tikzpicture}[baseline=0cm,scale=.8,yscale=-1]
\draw[help lines, color=gray!60, dashed] (-1,-1) grid (1,1);
\node at (-1, -1.5) {\scalebox{0.65}{\scri{$n-2$}}\,\,\,};
 \node at (0, -1.5) {\scalebox{0.65}{\scri{$n-1$}}};
 \node at (1, -1.5) {\scalebox{0.65}{\,\,\,\scri{$n$}}};
\node at (-2, -1) {\scalebox{0.65}{\,\,\,\scri{$\ell-1$}}};
\node at (-2, 0) {\scalebox{0.65}{\,\,\,\scri{$\ell$}}};
\node at (-2, 1) {\scalebox{0.65}{\,\,\,\scri{$\ell+1$}}};
\begin{scope}[every node/.style={minimum size=.1cm,inner sep=0mm,fill,circle}]
\draw[thick] (-1,0) node {} -- (0,1) node {};
\end{scope}
\end{tikzpicture}}
\end{minipage}}
\caption{The local configuration of $p_1$ (left) and $p_2$ (right) near by $(u,\ell)$ for $u=n-1$.}\label{local configuration of q (left) and p 2}
\end{figure}

\begin{figure}[H]
\resizebox{1.0\width}{1.0\height}{
\begin{minipage}[b]{0.5\linewidth}
\centerline{
\begin{tikzpicture}[baseline=0cm,scale=.8,yscale=-1]
\draw[help lines, color=gray!60, dashed] (-1,-1) grid (1,1);
\node at (-1, -1.5) {\scalebox{0.55}{\scri{$n-2$}}\,\,\,};
 \node at (0, -1.5) {\scalebox{0.55}{\scri{$n-1$}}};
 \node at (1, -1.5) {\scalebox{0.55}{\,\,\,\scri{$n$}}};
\node at (-2, -1) {\scalebox{0.65}{\,\,\,\scri{$\ell-1$}}};
\node at (-2, 0) {\scalebox{0.65}{\,\,\,\scri{$\ell$}}};
\node at (-2, 1) {\scalebox{0.65}{\,\,\,\scri{$\ell+1$}}};
\begin{scope}[every node/.style={minimum size=.1cm,inner sep=0mm,fill,circle}]
\draw[thick] (1,0) node {} -- (0,-1);
\end{scope}
\end{tikzpicture}}
\end{minipage}
\hspace{-20mm}
\begin{minipage}[b]{0.5\linewidth}
\centerline{
\begin{tikzpicture}[baseline=0cm,scale=.8,yscale=-1]
\draw[help lines, color=gray!60, dashed] (-1,-1) grid (1,1);
\node at (-1, -1.5) {\scalebox{0.65}{\scri{$n-2$}}\,\,\,};
 \node at (0, -1.5) {\scalebox{0.65}{\scri{$n-1$}}};
 \node at (1, -1.5) {\scalebox{0.65}{\,\,\,\scri{$n$}}};
\node at (-2, -1) {\scalebox{0.65}{\,\,\,\scri{$\ell-1$}}};
\node at (-2, 0) {\scalebox{0.65}{\,\,\,\scri{$\ell$}}};
\node at (-2, 1) {\scalebox{0.65}{\,\,\,\scri{$\ell+1$}}};
\begin{scope}[every node/.style={minimum size=.1cm,inner sep=0mm,fill,circle}]
\draw[thick] (0,1) node {} -- (1,0)node {};
\end{scope}
\end{tikzpicture}}
\end{minipage}}
\caption{The local configuration of $p_1$ (left) and $p_2$ (right) near by $(u,\ell)$ for $u=n-1$.}\label{local configuration of q (left) and p 3}
\end{figure}

%In this case, $m(p_1)+m(p_2)=(Y_{j, \ell-1}+Y_{j, \ell-1}A^{-1}_{j,\ell})M$ or $m(p_1)+m(p_2)=(Y_{n, \ell-1}+Y_{n,\ell-1}A^{-1}_{n,\ell})M$, where $M$ is a monomial in $\mathbb{Z}[Y^{\pm1}_{\overline{i},\ell}, Y^{\pm 1}_{n,k}]_{ \substack{ \hspace{-1cm} (i,\ell)\in \mathcal{X},\\ n-k\equiv 0 \hspace{-0.1cm}\pmod 2}}$ and has no the factors $Y^{\pm1}_{j,p}$ for some $p\in\mathbb{Z}$. Hence

In this case, $m(p_1)+m(p_2)=(Y_{j, \ell-1}+Y_{j, \ell-1}A^{-1}_{j,\ell})M$, where $M$ is a monomial in $\{Y^{\pm1}_{i,\ell} \mid i\in I, \ell \in \mathbb{Z}\}$ without the factors $Y^{\pm1}_{j,\ell}$ for $\ell\in\mathbb{Z}$. Hence
\begin{align*}
S_j(m(p_1)+m(p_2)) & = S_j((Y_{j, \ell-1}+Y_{j, \ell-1}A^{-1}_{j,\ell})M)\\
& = S_j(Y_{j, \ell-1}+Y_{j, \ell-1}A^{-1}_{j,\ell}) M + (Y_{j, \ell-1}+Y_{j, \ell-1}A^{-1}_{j,\ell}) S_j(M) \\
& = 0,
\end{align*}
where the last equation follows from Proposition \ref{Kernel} and $S_j(M)=0$.

{\bf Case 3.} Assume that $|C|=4$. Let $p_{1}$, $p_{2}$, $p_{3}$, and $p_{4}$ be the four paths in $C$. Since $C$ is a $j$-connected component of $\mathscr{P}_{i,k}$, we assume without loss of generality that
\[
p_{2}=p_{1}\mathscr{A}^{-1}_{N-u,\ell'},\quad p_{3}=p_{1}\mathscr{A}^{-1}_{u,\ell},\quad p_{4}=p_{2}\mathscr{A}^{-1}_{u,\ell}=p_{3}\mathscr{A}^{-1}_{N-u,\ell'},
\]
where $u\neq n-1$ and $\ell, \ell' \in \mathbb{Z}$. The local configurations of $p_{1}$, $p_{2}$, $p_{3}$, and $p_{4}$ near by $(u,\ell)$ and $(N-u,\ell')$ are depicted in  Figure \ref{local configuration of p1p2p3p4}, and the other parts of $p_1$, $p_2$, $p_3$ and $p_4$ are the same.

\begin{figure}
\begin{align*}
\raisebox{0.5cm}{
\begin{tikzpicture}[scale=.5]
\draw[help lines, color=gray!60, dashed] (-2.1,-5.1) grid (8.1, 1.5);
\node at (-1, 1.5) {\scalebox{0.45}{\scri{$u-1$}}};
\node at (0, 1.5) {\scalebox{0.45}{\scri{$u$}}};
\node at (1, 1.5) {\scalebox{0.45}{\scri{$u+1$}}};
\node at (-2.5, 1) {\scalebox{0.45}{\scri{$\ell-1$}}};
\node at (-2.5, 0) {\scalebox{0.45}{\scri{$\ell$}}};
\node at (-2.5, -1) {\scalebox{0.45}{\scri{$\ell+1$}}};
\node at (5, 1.5) {\scalebox{0.45}{\scri{$N-u-1$}}\,\,\,};
\node at (6, 1.5) {\scalebox{0.45}{\scri{$N-u$}}};
\node at (7, 1.5) {\scalebox{0.45}{\,\,\,\,\,\scri{$N-u+1$}}};
\node at (-2.5, -3) {\scalebox{0.45}{\scri{$\ell'-1$}}};
\node at (-2.5, -4) {\scalebox{0.45}{\scri{$\ell'$}}};
\node at (-2.5, -5) {\scalebox{0.45}{\scri{$\ell'+1$}}};
\node at (-2.5, -2) {\scalebox{0.45}{\scri{$\vdots$}}};
\node at (3, 1.5) {\scalebox{0.45}{\scri{$\cdots$}}};
\begin{scope}[every node/.style={minimum size=.1cm,inner sep=0mm,fill,circle}]
\draw[thick] (-1,0) node {} -- (0,1) node {} -- (1,0)node {};
\draw[thick] (5,-4) node {} -- (6,-3) node {} -- (7,-4)node {};
\draw [dotted] (1,0) -- (2,-1)-- (3,-2)-- (4,-3)-- (5,-4);
\end{scope}
\node at (3, -6) {\scri{\text{The path $p_1$}}};
\end{tikzpicture}
}
\raisebox{0.60cm}{
\begin{tikzpicture}[scale=0.5]
\draw[help lines, color=gray!60, dashed] (-2.1,-5.1) grid (8.1, 1.5);
\node at (-1, 1.5) {\scalebox{0.45}{\scri{$u-1$}}};
\node at (0, 1.5) {\scalebox{0.45}{\scri{$u$}}};
\node at (1, 1.5) {\scalebox{0.45}{\scri{$u+1$}}};
\node at (-2.5, 1) {\scalebox{0.45}{\scri{$\ell-1$}}};
\node at (-2.5, 0) {\scalebox{0.45}{\scri{$\ell$}}};
\node at (-2.5, -1) {\scalebox{0.45}{\scri{$\ell+1$}}};
\node at (5, 1.5) {\scalebox{0.45}{\scri{$N-u-1$}}\,\,\,};
\node at (6, 1.5) {\scalebox{0.45}{\scri{$N-u$}}};
\node at (7, 1.5) {\scalebox{0.45}{\,\,\,\,\,\scri{$N-u+1$}}};
\node at (-2.5, -3) {\scalebox{0.45}{\scri{$\ell'-1$}}};
\node at (-2.5, -4) {\scalebox{0.45}{\scri{$\ell'$}}};
\node at (-2.5, -5) {\scalebox{0.45}{\scri{$\ell'+1$}}};
\node at (-2.5, -2) {\scalebox{0.45}{\scri{$\vdots$}}};
\node at (3, 1.5) {\scalebox{0.45}{\scri{$\cdots$}}};
\begin{scope}[every node/.style={minimum size=.1cm,inner sep=0mm,fill,circle}]
\draw[thick] (-1,0) node {} -- (0,1) node {} -- (1,0)node {};
\draw[thick] (5,-4) node {} -- (6,-5) node {} -- (7,-4)node {};
\draw [dotted] (1,0) -- (2,-1)-- (3,-2)-- (4,-3)-- (5,-4);
\end{scope}
\node at (3, -6) {\scri{\text{The path $p_2$}}};
\end{tikzpicture}
} \\
\raisebox{0.5cm}{
\begin{tikzpicture}[scale=0.5]
\draw[help lines, color=gray!60, dashed] (-2.1,-5.1) grid (8.1, 1.5);
\node at (-1, 1.5) {\scalebox{0.45}{\scri{$u-1$}}};
\node at (0, 1.5) {\scalebox{0.45}{\scri{$u$}}};
\node at (1, 1.5) {\scalebox{0.45}{\scri{$u+1$}}};
\node at (-2.5, 1) {\scalebox{0.45}{\scri{$\ell-1$}}};
\node at (-2.5, 0) {\scalebox{0.45}{\scri{$\ell$}}};
\node at (-2.5, -1) {\scalebox{0.45}{\scri{$\ell+1$}}};
\node at (5, 1.5) {\scalebox{0.45}{\scri{$N-u-1$}}\,\,\,};
\node at (6, 1.5) {\scalebox{0.45}{\scri{$N-u$}}};
\node at (7, 1.5) {\scalebox{0.45}{\,\,\,\,\,\scri{$N-u+1$}}};
\node at (-2.5, -3) {\scalebox{0.45}{\scri{$\ell'-1$}}};
\node at (-2.5, -4) {\scalebox{0.45}{\scri{$\ell'$}}};
\node at (-2.5, -5) {\scalebox{0.45}{\scri{$\ell'+1$}}};
\node at (-2.5, -2) {\scalebox{0.45}{\scri{$\vdots$}}};
\node at (3, 1.5) {\scalebox{0.45}{\scri{$\cdots$}}};
\begin{scope}[every node/.style={minimum size=.1cm,inner sep=0mm,fill,circle}]
\draw[thick] (-1,0) node {} -- (0,-1) node {} -- (1,0)node {};
\draw[thick] (5,-4) node {} -- (6,-3) node {} -- (7,-4)node {};
\draw [dotted] (1,0) -- (2,-1)-- (3,-2)-- (4,-3)-- (5,-4);
\end{scope}
\node at (3, -6) {\scri{\text{The path $p_3$}}};
\end{tikzpicture}
}
\raisebox{0.5cm}{
\begin{tikzpicture}[scale=0.5]
\draw[help lines, color=gray!60, dashed] (-2.1,-5.1) grid (8.1, 1.5);
\node at (-1, 1.5) {\scalebox{0.45}{\scri{$u-1$}}};
\node at (0, 1.5) {\scalebox{0.45}{\scri{$u$}}};
\node at (1, 1.5) {\scalebox{0.45}{\scri{$u+1$}}};
\node at (-2.5, 1) {\scalebox{0.45}{\scri{$\ell-1$}}};
\node at (-2.5, 0) {\scalebox{0.45}{\scri{$\ell$}}};
\node at (-2.5, -1) {\scalebox{0.45}{\scri{$\ell+1$}}};
\node at (5, 1.5) {\scalebox{0.45}{\scri{$N-u-1$}}\,\,\,};
\node at (6, 1.5) {\scalebox{0.45}{\scri{$N-u$}}};
\node at (7, 1.5) {\scalebox{0.45}{\,\,\,\,\,\scri{$N-u+1$}}};
\node at (-2.5, -3) {\scalebox{0.45}{\scri{$\ell'-1$}}};
\node at (-2.5, -4) {\scalebox{0.45}{\scri{$\ell'$}}};
\node at (-2.5, -5) {\scalebox{0.45}{\scri{$\ell'+1$}}};
\node at (-2.5, -2) {\scalebox{0.45}{\scri{$\vdots$}}};
\node at (3, 1.5) {\scalebox{0.45}{\scri{$\cdots$}}};
\begin{scope}[every node/.style={minimum size=.1cm,inner sep=0mm,fill,circle}]
\draw[thick] (-1,0) node {} -- (0,-1) node {} -- (1,0)node {};
\draw[thick] (5,-4) node {} -- (6,-5) node {} -- (7,-4)node {};
\draw [dotted] (1,0) -- (2,-1)-- (3,-2)-- (4,-3)-- (5,-4);
\end{scope}
\node at (3, -6) {\scri{\text{The path $p_4$}}};
\end{tikzpicture}
}
\end{align*}
\caption{The local configuration of $p_1$ (top left),  $p_2$ (top right),  $p_3$ (bottom left) and  $p_4$ (bottom right) near by $(u,\ell)$ and $(N-u,\ell')$.} \label{local configuration of p1p2p3p4}
\end{figure}

In this case, we have
\[
m(p_{1})+m(p_{2})+ m(p_3)+m(p_{4})=(Y_{j,\ell-1}+Y_{j,\ell-1}A^{-1}_{j,\ell})(Y_{j,\ell'-1}+Y_{j,\ell'-1}A^{-1}_{j,\ell'})M,
\]
where $M$ is a monomial in $\{Y^{\pm1}_{i,\ell} \mid i\in I, \ell \in \mathbb{Z}\}$ without the factors $Y^{\pm1}_{j,\ell}$ for $\ell\in\mathbb{Z}$. By the Leibniz rule of $S_{j}$ and the Proposition \ref{Kernel}, we conclude that
\[
S_{j}(m(p_{1})+m(p_{2})+ m(p_3)+m(p_{4}))=0,
\]
so $m(p_{1})+m(p_{2})+ m(p_3)+m(p_{4}) \in \text{ker}(S_j)$.

Since $\mathscr{P}_{i,k}$ is a disjoint union of all $j$-connected components, we have
\[
\sum_{p\in \mathscr{P}_{i,k}} m(p) \subset \text{ker}(S_j).
\]
When $j$ runs over the set $I$, we conclude that
\[
\sum_{p\in \mathscr{P}_{i,k}} m(p) \subseteq \bigcap_{j\in I} \text{ker}(S_j) = \chi_{q}([L(Y_{i,k})]).
\]

The reverse inclusion $\chi_{q}([L(Y_{i,k})])\subseteq \sum_{p\in \mathscr{P}_{i,k}} m(p)$ follows from Corollary \ref{equation1}. The proof is completed.
\end{proof}

\begin{remark}
The coefficient of each Laurent monomial in the $q$-character of a fundamental module is $1$ in types $A_n$, $B_{n}$ $C_{n}$ \cite{CM06,H05,KS95}, and type $G_{2}$ \cite[Section 8.4]{Her04}. It is not true for type $D_n$ \cite{CM06,H05,KS95}, types $E_{6}$, $E_{7}$, $E_{8}$ \cite{H07, Nak01, Nak10}, and type $F_{4}$ \cite[Appendix 8]{H05}. 
\end{remark}

In practice, the horizontal coordinates in our figures are labeled by $\mathscr{S}$ when we draw paths. The horizontal coordinates in our figures are labeled by the images of $\mathscr{S}\backslash \{0,N\}$ under $\overline{\cdot}$ when we assign monomials or binomials to paths.

We give an example of type $D_4$ to illustrate our Theorem \ref{path formula for fundamental modules}.

\begin{example}\label{example of D4}
Let $\mathfrak{g}=\mathfrak{so}_8(\mathbb{C})$. All paths in $\mathscr{P}_{1,0}$, $\mathscr{P}_{2,1}$, and $\mathscr{P}_{3,0}$ are shown in Figure \ref{all paths in D4}, and monomials or binomials associated to paths in $\mathscr{P}_{1,0}$, $\mathscr{P}_{2,1}$, and $\mathscr{P}_{3,0}$ are shown in Figure \ref{q-character of 1_0},  Figures \ref{q-character of 2_1 1} and \ref{q-character of 2_1 2}, and Figure \ref{q-character of 3_0} respectively. By Theorem \ref{path formula for fundamental modules},
\begin{align*}
\chi_{q}([L(Y_{1,0})])  & = Y_{1,0}+Y^{-1}_{1,2}Y_{2,1}+Y^{-1}_{2,3}Y_{3,2}Y_{4,2}+Y^{-1}_{3,4}Y_{4,2}+Y_{3,2}Y^{-1}_{4,4}+Y_{2,3}Y^{-1}_{3,4}Y^{-1}_{4,4}\\
& + Y_{1,4}Y^{-1}_{2,5}+Y^{-1}_{1,6}, \\
\chi_{q}([L(Y_{2,1})]) & = Y_{2,1}+Y_{1,2}Y^{-1}_{2,3}Y_{3,2}Y_{4,2}+Y^{-1}_{1,4}Y_{3,2}Y_{4,2}+Y_{1,2}Y^{-1}_{3,4}Y_{4,2}+Y_{1,2}Y_{3,2}Y^{-1}_{4,4}\\
& +Y^{-1}_{1,4}Y_{2,3}Y^{-1}_{3,4}Y_{4,2}+Y^{-1}_{1,4}Y_{2,3}Y_{3,2}Y^{-1}_{4,4}+Y_{1,2}Y_{2,3}Y^{-1}_{3,4}Y^{-1}_{4,4} + Y^{-1}_{2,5}Y_{3,2}Y_{3,4} \\
& + Y^{-1}_{2,5}Y_{4,2}Y_{4,4} + Y^{-1}_{1,4}Y^{2}_{2,3}Y^{-1}_{3,4}Y^{-1}_{4,4} + Y_{1,2}Y_{1,4}Y^{-1}_{2,5} + Y_{3,2}Y^{-1}_{3,6} + Y_{4,2}Y^{-1}_{4,6}\\
& + 2 Y_{2,3}Y^{-1}_{2,5} + Y_{1,2}Y^{-1}_{1,6} + Y_{2,3}Y^{-1}_{3,4}Y^{-1}_{3,6} + Y_{2,3}Y^{-1}_{4,4}Y^{-1}_{4,6} + Y_{1,4}Y^{-2}_{2,5}Y_{3,4}Y_{4,4}\\ 
& + Y^{-1}_{1,4}Y^{-1}_{1,6}Y_{2,3} + Y_{1,4}Y^{-1}_{2,5}Y^{-1}_{3,6}Y_{4,4}+Y_{1,4}Y^{-1}_{2,5}Y_{3,4}Y^{-1}_{4,6} + Y^{-1}_{1,6}Y^{-1}_{2,5}Y_{3,4}Y_{4,4} \\
& +Y_{1,4}Y^{-1}_{3,6}Y^{-1}_{4,6} + Y^{-1}_{1,6}Y^{-1}_{3,6}Y_{4,4} + Y^{-1}_{1,6}Y_{3,4}Y^{-1}_{4,6} + Y^{-1}_{1,6}Y_{2,5}Y^{-1}_{3,6}Y^{-1}_{4,6}+Y^{-1}_{2,7}, \\
\chi_{q}([L(Y_{3,0})]) & = Y_{3,0}+Y_{2,1}Y^{-1}_{3,2}+Y_{1,2}Y^{-1}_{2,3}Y_{4,2}+Y_{1,2}Y^{-1}_{4,4}+Y^{-1}_{1,4}Y_{4,2}+Y^{-1}_{1,4}Y_{2,3}Y^{-1}_{4,4} \\
& + Y^{-1}_{2,5}Y_{3,4}+Y^{-1}_{3,6},
\end{align*}
and after switching $Y_{3,\ell}$ with $Y_{4,\ell}$ in $\chi_{q}([L(Y_{3,0})])$, with $\ell\in \mathbb{Z}$, we have
\begin{align*}
\chi_{q}([L(Y_{4,0})]) & = Y_{4,0}+Y_{2,1}Y^{-1}_{4,2}+Y_{1,2}Y^{-1}_{2,3}Y_{3,2}+Y^{-1}_{1,4}Y_{3,2}+Y_{1,2}Y^{-1}_{3,4}+Y^{-1}_{1,4}Y_{2,3}Y^{-1}_{3,4}\\
& +Y^{-1}_{2,5}Y_{4,4}+Y^{-1}_{4,6}.
\end{align*}

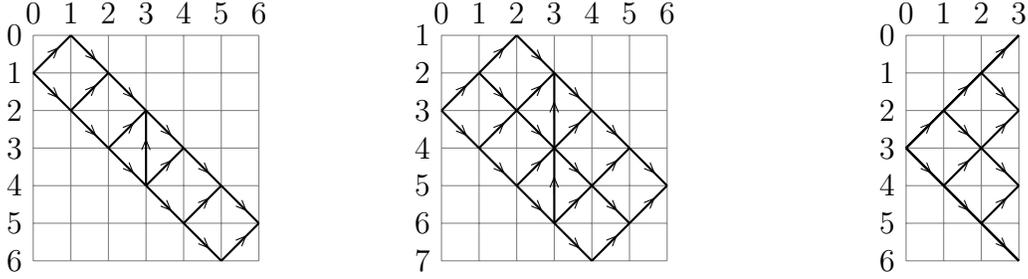
\begin{figure}
\resizebox{1.0\width}{1.0\height}{
\begin{minipage}[b]{0.33\linewidth}
\centerline{
\begin{tikzpicture}
\draw[step=.5cm,gray,thin] (-0.5,5) grid (2.5,8) (-0.5,5)--(2.5,5);
%\draw[fill] (0,8.3) circle (2pt) -- (0.5,8.3) circle (2pt) --(1,8.3) circle (2pt) --(1.5,8.3) circle (2pt)-- (2,8.3) circle (2pt);
\begin{scope}[thick, every node/.style={sloped,allow upside down}]
\draw (-0.5,7.5)--node {\midarrow}(0,7);
\draw (0,7)--node {\midarrow}(0.5,6.5);
\draw (0.5,6.5)--node {\midarrow}(1,6);
\draw (1,6)--node {\midarrow}(1.5,5.5);
\draw (1.5,5.5)--node {\midarrow}(2,5);
\draw (2,5)--node {\midarrow}(2.5,5.5);
\draw (-0.5,7.5)--node {\midarrow}(0,8);
\draw (0,8)--node {\midarrow}(0.5,7.5);
\draw (0.5,7.5)--node {\midarrow}(1,7);
\draw (1,7)--node {\midarrow}(1.5,6.5);
\draw (1.5,6.5)--node {\midarrow}(2,6);
\draw (2,6)--node {\midarrow}(2.5,5.5);
\draw (0,7)--node {\midarrow}(0.5,7.5);
\draw (0.5,6.5)--node {\midarrow}(1,7);
\draw (1,6)--node {\midarrow}(1.5,6.5);
\draw (1.5,5.5)--node {\midarrow}(2,6);
\draw (1,6)--node {\midarrow}(1,7);
\end{scope}
\node at (-0.5,8.3) {$0$};
\node at (0,8.3) {$1$};
\node at (0.5,8.3) {$2$};
\node at (1,8.3) {$3$};
\node at (1.5,8.3) {$4$};
\node at (2,8.3) {$5$};
\node at (2.5,8.3) {$6$};
\node [left] at (-0.5,8) {$0$};
\node [left] at (-0.5,7.5) {$1$};
\node [left] at (-0.5,7) {$2$};
\node [left] at (-0.5,6.5) {$3$};
\node [left] at (-0.5,6) {$4$};
\node [left] at (-0.5,5.5) {$5$};
\node [left] at (-0.5,5) {$6$};
\end{tikzpicture}}
\end{minipage}
\begin{minipage}[b]{0.33\linewidth}
\centerline{
\begin{tikzpicture}
\draw[step=.5cm,gray,thin] (-0.5,5) grid (2.5,8) (-0.5,5)--(2.5,5);
%\draw[fill] (0,8.3) circle (2pt) -- (0.5,8.3) circle (2pt) --(1,8.3) circle (2pt) --(1.5,8.3) circle (2pt)-- (2,8.3) circle (2pt);
\begin{scope}[thick, every node/.style={sloped,allow upside down}]
\draw (-0.5,7)--node {\midarrow}(0,6.5);
\draw (0,6.5)--node {\midarrow}(0.5,6);
\draw (0.5,6)--node {\midarrow}(1,5.5);
\draw (1,5.5)--node {\midarrow}(1.5,6);
\draw (1.5,6)--node {\midarrow}(2,6.5);
\draw (2,6.5)--node {\midarrow}(2.5,6);
\draw (1,5.5)--node {\midarrow}(1.5,5);
\draw (1.5,5)--node {\midarrow}(2,5.5);
\draw (2,5.5)--node {\midarrow}(2.5,6);
\draw (1.5,6)--node {\midarrow}(2,5.5);
\draw (-0.5,7)--node {\midarrow}(0,7.5);
\draw (0,7.5)--node {\midarrow}(0.5,8);
\draw (0.5,8)--node {\midarrow}(1,7.5);
\draw (1,7.5)--node {\midarrow}(1.5,7);
\draw (1.5,7)--node {\midarrow}(2,6.5);
\draw (0,7.5)--node {\midarrow}(0.5,7);
\draw (0.5,7)--node {\midarrow}(1,6.5);
\draw (1,6.5)--node {\midarrow}(1.5,6);
\draw (0,6.5)--node {\midarrow}(0.5,7);
\draw (0.5,7)--node {\midarrow}(1,7.5);
\draw (0.5,6)--node {\midarrow}(1,6.5);
\draw (1,6.5)--node {\midarrow}(1.5,7);
\draw (1,5.5)--node {\midarrow}(1,6.5);
\draw (1,6.5)--node {\midarrow}(1,7.5);
\end{scope}
\node at (-0.5,8.3) {$0$};
\node at (0,8.3) {$1$};
\node at (0.5,8.3) {$2$};
\node at (1,8.3) {$3$};
\node at (1.5,8.3) {$4$};
\node at (2,8.3) {$5$};
\node at (2.5,8.3) {$6$};
\node [left] at (-0.5,8) {$1$};
\node [left] at (-0.5,7.5) {$2$};
\node [left] at (-0.5,7) {$3$};
\node [left] at (-0.5,6.5) {$4$};
\node [left] at (-0.5,6) {$5$};
\node [left] at (-0.5,5.5) {$6$};
\node [left] at (-0.5,5) {$7$};
\end{tikzpicture}}
\end{minipage}
\begin{minipage}[b]{0.33\linewidth}
\centerline{
\begin{tikzpicture}
\draw[step=.5cm,gray,thin] (-0.5,5) grid (1,8) (-0.5,5)--(1,5);
%\draw[fill] (0,8.3) circle (2pt)--(0.5,8.3) circle (2pt);
%\draw[fill] (0.5,8.3) circle (2pt)--(1,8.3) circle (2pt);
%\draw[fill] (1,8.3) circle (2pt)--(1.5,8.3) circle (2pt);
%\draw[fill] (1.5,8.3) circle (2pt)--(2,8.3) circle (2pt);
\node at (-0.5,8.3) {$0$};
\node at (0,8.3) {$1$};
\node at (0.5,8.3) {$2$};
\node at (1,8.3) {$3$};
%\node at (1.5,8.3) {$4$};
%\node at (2,8.3) {$5$};
\node [left] at (-0.5,8) {$0$};
\node [left] at (-0.5,7.5) {$1$};
\node [left] at (-0.5,7) {$2$};
\node [left] at (-0.5,6.5) {$3$};
\node [left] at (-0.5,6) {$4$};
\node [left] at (-0.5,5.5) {$5$};
\node [left] at (-0.5,5) {$6$};
\draw[thick] (-0.5,6.5)--(0,7)--(0.5,7.5)--(1,8);
\draw[thick] (-0.5,6.5)--(0,6)--(0.5,5.5)--(1,5);
\begin{scope}[thick, every node/.style={sloped,allow upside down}]
\draw (-0.5,6.5)--node {\midarrow}(0,7);
\draw (0,7)--node {\midarrow}(0.5,7.5);
\draw (0,7)--node {\midarrow}(0.5,6.5);
\draw (0.5,6.5)--node {\midarrow}(1,7);
\draw (0.5,6.5)--node {\midarrow}(1,6);
\draw (0.5,7.5)--node {\midarrow}(1,7);
\draw (0.5,7.5)--node {\midarrow}(1,8);
\draw (-0.5,6.5)--node {\midarrow}(0,6);
\draw (0,6)--node {\midarrow}(0.5,6.5);
\draw (0,6)--node {\midarrow}(0.5,5.5);
\draw (0.5,5.5)--node {\midarrow}(1,6);
\draw (0.5,5.5)--node {\midarrow}(1,5);
\end{scope}
\end{tikzpicture}}
\end{minipage}}
\caption{All paths in $\mathscr{P}_{1,0}$ (left), $\mathscr{P}_{2,1}$ (middle), and $\mathscr{P}_{3,0}$ (right).}\label{all paths in D4}
\end{figure}

Note that the coefficient of the Laurent monomial $Y_{2,3}Y^{-1}_{2,5}$ appearing in $\chi_{q}([L(Y_{2,1})])$ is $2$. Bittmann \cite[Section 8]{Bit21b} computed explicitly the $(q,t)$-character of the fundamental module $L(Y_{2,k})$, for some $k\in \mathbb{Z}$, by quantum cluster mutations. When $t=1$, the $(q,t)$-character is the $q$-character.

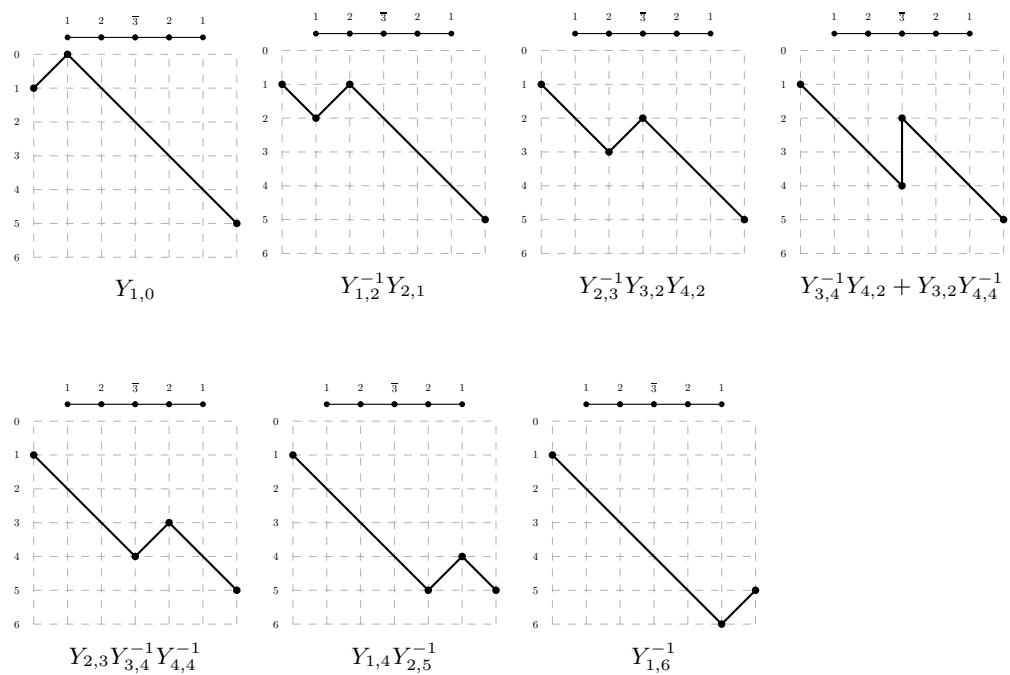
\begin{figure}
\begin{align*}
&
\raisebox{0.5cm}{
\begin{tikzpicture}[scale=.45]
\draw[help lines, color=gray!60, dashed] (-2.0,-5.0) grid (4.0, 1.0);
\draw[fill] (-1,1.5) circle (2pt)--(0,1.5) circle (2pt);
\draw[fill] (0,1.5) circle (2pt)--(1,1.5) circle (2pt);
\draw[fill] (1,1.5) circle (2pt)--(2,1.5) circle (2pt);
\draw[fill] (2,1.5) circle (2pt)--(3,1.5) circle (2pt);
\node at (-1, 2) {\scalebox{0.45}{\scri{$1$}}};
\node at (0, 2) {\scalebox{0.45}{\scri{$2$}}};
\node at (1, 2) {\scalebox{0.45}{\scri{$\overline{3}$}}};
\node at (2, 2) {\scalebox{0.45}{\scri{$2$}}};
\node at (3, 2) {\scalebox{0.45}{\scri{$1$}}};
\node at (-2.5, 1) {\scalebox{0.45}{\scri{$0$}}};
\node at (-2.5, 0) {\scalebox{0.45}{\scri{$1$}}};
\node at (-2.5, -1) {\scalebox{0.45}{\scri{$2$}}};
\node at (-2.5, -2) {\scalebox{0.45}{\scri{$3$}}};
\node at (-2.5, -3) {\scalebox{0.45}{\scri{$4$}}};
\node at (-2.5, -4) {\scalebox{0.45}{\scri{$5$}}};
\node at (-2.5, -5) {\scalebox{0.45}{\scri{$6$}}};
\begin{scope}[every node/.style={minimum size=.1cm,inner sep=0mm,fill,circle}]
\draw[thick] (-2,0) node {} -- (-1,1) node {} -- (4,-4)node {};
\end{scope}
\node at (1, -6) {\scri{\text{$Y_{1,0}$}}};
\end{tikzpicture}}
\raisebox{0.5cm}{
\begin{tikzpicture}[scale=0.45]
\draw[help lines, color=gray!60, dashed] (-2.0,-5.0) grid (4.0, 1.0);
\draw[fill] (-1,1.5) circle (2pt)--(0,1.5) circle (2pt);
\draw[fill] (0,1.5) circle (2pt)--(1,1.5) circle (2pt);
\draw[fill] (1,1.5) circle (2pt)--(2,1.5) circle (2pt);
\draw[fill] (2,1.5) circle (2pt)--(3,1.5) circle (2pt);
\node at (-1, 2) {\scalebox{0.45}{\scri{$1$}}};
\node at (0, 2) {\scalebox{0.45}{\scri{$2$}}};
\node at (1, 2) {\scalebox{0.45}{\scri{$\overline{3}$}}};
\node at (2, 2) {\scalebox{0.45}{\scri{$2$}}};
\node at (3, 2) {\scalebox{0.45}{\scri{$1$}}};
\node at (-2.5, 1) {\scalebox{0.45}{\scri{$0$}}};
\node at (-2.5, 0) {\scalebox{0.45}{\scri{$1$}}};
\node at (-2.5, -1) {\scalebox{0.45}{\scri{$2$}}};
\node at (-2.5, -2) {\scalebox{0.45}{\scri{$3$}}};
\node at (-2.5, -3) {\scalebox{0.45}{\scri{$4$}}};
\node at (-2.5, -4) {\scalebox{0.45}{\scri{$5$}}};
\node at (-2.5, -5) {\scalebox{0.45}{\scri{$6$}}};
\begin{scope}[every node/.style={minimum size=.1cm,inner sep=0mm,fill,circle}]
\draw[thick] (-2,0) node {} -- (-1,-1) node {} -- (0,0)node {} -- (4,-4)node {};
\end{scope}
\node at (1, -6) {\scri{\text{$Y^{-1}_{1,2}Y_{2,1}$}}};
\end{tikzpicture}
} 
\raisebox{0.5cm}{
\begin{tikzpicture}[scale=0.45]
\draw[help lines, color=gray!60, dashed] (-2.0,-5.0) grid (4.0, 1.0);
\draw[fill] (-1,1.5) circle (2pt)--(0,1.5) circle (2pt);
\draw[fill] (0,1.5) circle (2pt)--(1,1.5) circle (2pt);
\draw[fill] (1,1.5) circle (2pt)--(2,1.5) circle (2pt);
\draw[fill] (2,1.5) circle (2pt)--(3,1.5) circle (2pt);
\node at (-1, 2) {\scalebox{0.45}{\scri{$1$}}};
\node at (0, 2) {\scalebox{0.45}{\scri{$2$}}};
\node at (1, 2) {\scalebox{0.45}{\scri{$\overline{3}$}}};
\node at (2, 2) {\scalebox{0.45}{\scri{$2$}}};
\node at (3, 2) {\scalebox{0.45}{\scri{$1$}}};
\node at (-2.5, 1) {\scalebox{0.45}{\scri{$0$}}};
\node at (-2.5, 0) {\scalebox{0.45}{\scri{$1$}}};
\node at (-2.5, -1) {\scalebox{0.45}{\scri{$2$}}};
\node at (-2.5, -2) {\scalebox{0.45}{\scri{$3$}}};
\node at (-2.5, -3) {\scalebox{0.45}{\scri{$4$}}};
\node at (-2.5, -4) {\scalebox{0.45}{\scri{$5$}}};
\node at (-2.5, -5) {\scalebox{0.45}{\scri{$6$}}};
\begin{scope}[every node/.style={minimum size=.1cm,inner sep=0mm,fill,circle}]
\draw[thick] (-2,0) node {} -- (0,-2) node {} -- (1,-1)node {} -- (4,-4)node {};
\end{scope}
\node at (1, -6) {\scri{\text{$Y^{-1}_{2,3}Y_{3,2}Y_{4,2}$}}};
\end{tikzpicture}
}
\raisebox{0.5cm}{
\begin{tikzpicture}[scale=0.45]
\draw[help lines, color=gray!60, dashed] (-2.0,-5.0) grid (4.0, 1.0);
\draw[fill] (-1,1.5) circle (2pt)--(0,1.5) circle (2pt);
\draw[fill] (0,1.5) circle (2pt)--(1,1.5) circle (2pt);
\draw[fill] (1,1.5) circle (2pt)--(2,1.5) circle (2pt);
\draw[fill] (2,1.5) circle (2pt)--(3,1.5) circle (2pt);
\node at (-1, 2) {\scalebox{0.45}{\scri{$1$}}};
\node at (0, 2) {\scalebox{0.45}{\scri{$2$}}};
\node at (1, 2) {\scalebox{0.45}{\scri{$\overline{3}$}}};
\node at (2, 2) {\scalebox{0.45}{\scri{$2$}}};
\node at (3, 2) {\scalebox{0.45}{\scri{$1$}}};
\node at (-2.5, 1) {\scalebox{0.45}{\scri{$0$}}};
\node at (-2.5, 0) {\scalebox{0.45}{\scri{$1$}}};
\node at (-2.5, -1) {\scalebox{0.45}{\scri{$2$}}};
\node at (-2.5, -2) {\scalebox{0.45}{\scri{$3$}}};
\node at (-2.5, -3) {\scalebox{0.45}{\scri{$4$}}};
\node at (-2.5, -4) {\scalebox{0.45}{\scri{$5$}}};
\node at (-2.5, -5) {\scalebox{0.45}{\scri{$6$}}};
\begin{scope}[every node/.style={minimum size=.1cm,inner sep=0mm,fill,circle}]
\draw[thick] (-2,0) node {} -- (1,-3) node {} -- (1,-1) node {} -- (4,-4)node {};
\end{scope}
\node at (1, -6) {\scri{\text{$Y^{-1}_{3,4}Y_{4,2}+Y_{3,2}Y^{-1}_{4,4}$}}};
\end{tikzpicture}
}
\\
&
\raisebox{0.5cm}{
\begin{tikzpicture}[scale=0.45]
\draw[help lines, color=gray!60, dashed] (-2.0,-5.0) grid (4.0, 1.0);
\draw[fill] (-1,1.5) circle (2pt)--(0,1.5) circle (2pt);
\draw[fill] (0,1.5) circle (2pt)--(1,1.5) circle (2pt);
\draw[fill] (1,1.5) circle (2pt)--(2,1.5) circle (2pt);
\draw[fill] (2,1.5) circle (2pt)--(3,1.5) circle (2pt);
\node at (-1, 2) {\scalebox{0.45}{\scri{$1$}}};
\node at (0, 2) {\scalebox{0.45}{\scri{$2$}}};
\node at (1, 2) {\scalebox{0.45}{\scri{$\overline{3}$}}};
\node at (2, 2) {\scalebox{0.45}{\scri{$2$}}};
\node at (3, 2) {\scalebox{0.45}{\scri{$1$}}};
\node at (-2.5, 1) {\scalebox{0.45}{\scri{$0$}}};
\node at (-2.5, 0) {\scalebox{0.45}{\scri{$1$}}};
\node at (-2.5, -1) {\scalebox{0.45}{\scri{$2$}}};
\node at (-2.5, -2) {\scalebox{0.45}{\scri{$3$}}};
\node at (-2.5, -3) {\scalebox{0.45}{\scri{$4$}}};
\node at (-2.5, -4) {\scalebox{0.45}{\scri{$5$}}};
\node at (-2.5, -5) {\scalebox{0.45}{\scri{$6$}}};
\begin{scope}[every node/.style={minimum size=.1cm,inner sep=0mm,fill,circle}]
\draw[thick] (-2,0) node {} -- (1,-3) node {} -- (2,-2)node {}-- (4,-4)node {};
\end{scope}
\node at (1, -6) {\scri{\text{$Y_{2,3}Y^{-1}_{3,4}Y^{-1}_{4,4}$}}};
\end{tikzpicture}
}
\raisebox{0.5cm}{
\begin{tikzpicture}[scale=0.45]
\draw[help lines, color=gray!60, dashed] (-2.0,-5.0) grid (4.0, 1.0);
\draw[fill] (-1,1.5) circle (2pt)--(0,1.5) circle (2pt);
\draw[fill] (0,1.5) circle (2pt)--(1,1.5) circle (2pt);
\draw[fill] (1,1.5) circle (2pt)--(2,1.5) circle (2pt);
\draw[fill] (2,1.5) circle (2pt)--(3,1.5) circle (2pt);
\node at (-1, 2) {\scalebox{0.45}{\scri{$1$}}};
\node at (0, 2) {\scalebox{0.45}{\scri{$2$}}};
\node at (1, 2) {\scalebox{0.45}{\scri{$\overline{3}$}}};
\node at (2, 2) {\scalebox{0.45}{\scri{$2$}}};
\node at (3, 2) {\scalebox{0.45}{\scri{$1$}}};
\node at (-2.5, 1) {\scalebox{0.45}{\scri{$0$}}};
\node at (-2.5, 0) {\scalebox{0.45}{\scri{$1$}}};
\node at (-2.5, -1) {\scalebox{0.45}{\scri{$2$}}};
\node at (-2.5, -2) {\scalebox{0.45}{\scri{$3$}}};
\node at (-2.5, -3) {\scalebox{0.45}{\scri{$4$}}};
\node at (-2.5, -4) {\scalebox{0.45}{\scri{$5$}}};
\node at (-2.5, -5) {\scalebox{0.45}{\scri{$6$}}};
\begin{scope}[every node/.style={minimum size=.1cm,inner sep=0mm,fill,circle}]
\draw[thick] (-2,0) node {} -- (2,-4) node {} -- (3,-3)node {}-- (4,-4)node {};
\end{scope}
\node at (1, -6) {\scri{\text{$Y_{1,4}Y^{-1}_{2,5}$}}};
\end{tikzpicture}
}
\raisebox{0.5cm}{
\begin{tikzpicture}[scale=0.45]
\draw[help lines, color=gray!60, dashed] (-2.0,-5.0) grid (4.0, 1.0);
\draw[fill] (-1,1.5) circle (2pt)--(0,1.5) circle (2pt);
\draw[fill] (0,1.5) circle (2pt)--(1,1.5) circle (2pt);
\draw[fill] (1,1.5) circle (2pt)--(2,1.5) circle (2pt);
\draw[fill] (2,1.5) circle (2pt)--(3,1.5) circle (2pt);
\node at (-1, 2) {\scalebox{0.45}{\scri{$1$}}};
\node at (0, 2) {\scalebox{0.45}{\scri{$2$}}};
\node at (1, 2) {\scalebox{0.45}{\scri{$\overline{3}$}}};
\node at (2, 2) {\scalebox{0.45}{\scri{$2$}}};
\node at (3, 2) {\scalebox{0.45}{\scri{$1$}}};
\node at (-2.5, 1) {\scalebox{0.45}{\scri{$0$}}};
\node at (-2.5, 0) {\scalebox{0.45}{\scri{$1$}}};
\node at (-2.5, -1) {\scalebox{0.45}{\scri{$2$}}};
\node at (-2.5, -2) {\scalebox{0.45}{\scri{$3$}}};
\node at (-2.5, -3) {\scalebox{0.45}{\scri{$4$}}};
\node at (-2.5, -4) {\scalebox{0.45}{\scri{$5$}}};
\node at (-2.5, -5) {\scalebox{0.45}{\scri{$6$}}};
\begin{scope}[every node/.style={minimum size=.1cm,inner sep=0mm,fill,circle}]
\draw[thick] (-2,0) node {} -- (3,-5) node {} -- (4,-4)node {};
\end{scope}
\node at (1,-6) {\scri{\text{$Y^{-1}_{1,6}$}}};
\end{tikzpicture}
}
\end{align*}
\caption{Monomials or binomials associated to paths in $\mathscr{P}_{1,0}$.} \label{q-character of 1_0}
\end{figure}

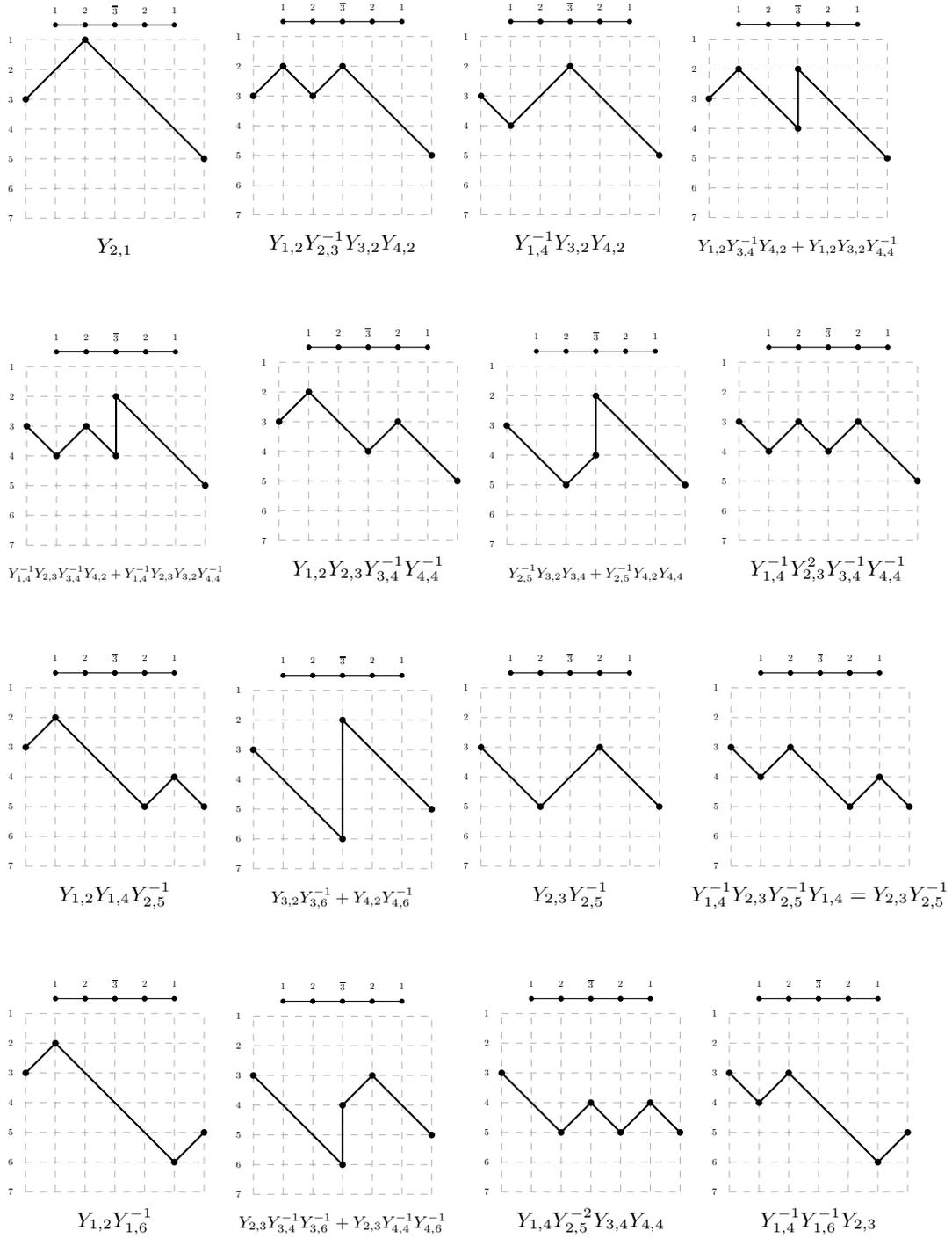
\begin{figure}
\begin{align*}
&
\raisebox{0.5cm}{
\begin{tikzpicture}[scale=.45]
\draw[help lines, color=gray!60, dashed] (-2.0,-5.0) grid (4.0, 1.0);
\draw[fill] (-1,1.5) circle (2pt)--(0,1.5) circle (2pt);
\draw[fill] (0,1.5) circle (2pt)--(1,1.5) circle (2pt);
\draw[fill] (1,1.5) circle (2pt)--(2,1.5) circle (2pt);
\draw[fill] (2,1.5) circle (2pt)--(3,1.5) circle (2pt);
\node at (-1, 2) {\scalebox{0.45}{\scri{$1$}}};
\node at (0, 2) {\scalebox{0.45}{\scri{$2$}}};
\node at (1, 2) {\scalebox{0.45}{\scri{$\overline{3}$}}};
\node at (2, 2) {\scalebox{0.45}{\scri{$2$}}};
\node at (3, 2) {\scalebox{0.45}{\scri{$1$}}};
\node at (-2.5, 1) {\scalebox{0.45}{\scri{$1$}}};
\node at (-2.5, 0) {\scalebox{0.45}{\scri{$2$}}};
\node at (-2.5, -1) {\scalebox{0.45}{\scri{$3$}}};
\node at (-2.5, -2) {\scalebox{0.45}{\scri{$4$}}};
\node at (-2.5, -3) {\scalebox{0.45}{\scri{$5$}}};
\node at (-2.5, -4) {\scalebox{0.45}{\scri{$6$}}};
\node at (-2.5, -5) {\scalebox{0.45}{\scri{$7$}}};
\begin{scope}[every node/.style={minimum size=.1cm,inner sep=0mm,fill,circle}]
\draw[thick] (-2,-1) node {} -- (0,1) node {} -- (4,-3)node {};
\end{scope}
\node at (1, -6) {\scri{\text{$Y_{2,1}$}}};
\end{tikzpicture}
}
\raisebox{0.5cm}{
\begin{tikzpicture}[scale=0.45]
\draw[help lines, color=gray!60, dashed] (-2.0,-5.0) grid (4.0, 1.0);
\draw[fill] (-1,1.5) circle (2pt)--(0,1.5) circle (2pt);
\draw[fill] (0,1.5) circle (2pt)--(1,1.5) circle (2pt);
\draw[fill] (1,1.5) circle (2pt)--(2,1.5) circle (2pt);
\draw[fill] (2,1.5) circle (2pt)--(3,1.5) circle (2pt);
\node at (-1, 2) {\scalebox{0.45}{\scri{$1$}}};
\node at (0, 2) {\scalebox{0.45}{\scri{$2$}}};
\node at (1, 2) {\scalebox{0.45}{\scri{$\overline{3}$}}};
\node at (2, 2) {\scalebox{0.45}{\scri{$2$}}};
\node at (3, 2) {\scalebox{0.45}{\scri{$1$}}};
\node at (-2.5, 1) {\scalebox{0.45}{\scri{$1$}}};
\node at (-2.5, 0) {\scalebox{0.45}{\scri{$2$}}};
\node at (-2.5, -1) {\scalebox{0.45}{\scri{$3$}}};
\node at (-2.5, -2) {\scalebox{0.45}{\scri{$4$}}};
\node at (-2.5, -3) {\scalebox{0.45}{\scri{$5$}}};
\node at (-2.5, -4) {\scalebox{0.45}{\scri{$6$}}};
\node at (-2.5, -5) {\scalebox{0.45}{\scri{$7$}}};
\begin{scope}[every node/.style={minimum size=.1cm,inner sep=0mm,fill,circle}]
\draw[thick] (-2,-1) node {} -- (-1,0) node {} -- (0,-1)node {} -- (1,0) node {} -- (4,-3)node {};
\end{scope}
\node at (1, -6) {\scri{\text{$Y_{1,2}Y^{-1}_{2,3}Y_{3,2}Y_{4,2}$}}};
\end{tikzpicture}
}
\raisebox{0.5cm}{
\begin{tikzpicture}[scale=0.45]
\draw[help lines, color=gray!60, dashed] (-2.0,-5.0) grid (4.0, 1.0);
\draw[fill] (-1,1.5) circle (2pt)--(0,1.5) circle (2pt);
\draw[fill] (0,1.5) circle (2pt)--(1,1.5) circle (2pt);
\draw[fill] (1,1.5) circle (2pt)--(2,1.5) circle (2pt);
\draw[fill] (2,1.5) circle (2pt)--(3,1.5) circle (2pt);
\node at (-1, 2) {\scalebox{0.45}{\scri{$1$}}};
\node at (0, 2) {\scalebox{0.45}{\scri{$2$}}};
\node at (1, 2) {\scalebox{0.45}{\scri{$\overline{3}$}}};
\node at (2, 2) {\scalebox{0.45}{\scri{$2$}}};
\node at (3, 2) {\scalebox{0.45}{\scri{$1$}}};
\node at (-2.5, 1) {\scalebox{0.45}{\scri{$1$}}};
\node at (-2.5, 0) {\scalebox{0.45}{\scri{$2$}}};
\node at (-2.5, -1) {\scalebox{0.45}{\scri{$3$}}};
\node at (-2.5, -2) {\scalebox{0.45}{\scri{$4$}}};
\node at (-2.5, -3) {\scalebox{0.45}{\scri{$5$}}};
\node at (-2.5, -4) {\scalebox{0.45}{\scri{$6$}}};
\node at (-2.5, -5) {\scalebox{0.45}{\scri{$7$}}};
\begin{scope}[every node/.style={minimum size=.1cm,inner sep=0mm,fill,circle}]
\draw[thick] (-2,-1) node {} -- (-1,-2) node {} -- (1,0)node {} -- (4,-3)node {};
\end{scope}
\node at (1, -6) {\scri{\text{$Y^{-1}_{1,4}Y_{3,2}Y_{4,2}$}}};
\end{tikzpicture}
}
\raisebox{0.5cm}{
\begin{tikzpicture}[scale=.45]
\draw[help lines, color=gray!60, dashed] (-2.0,-5.0) grid (4.0, 1.0);
\draw[fill] (-1,1.5) circle (2pt)--(0,1.5) circle (2pt);
\draw[fill] (0,1.5) circle (2pt)--(1,1.5) circle (2pt);
\draw[fill] (1,1.5) circle (2pt)--(2,1.5) circle (2pt);
\draw[fill] (2,1.5) circle (2pt)--(3,1.5) circle (2pt);
\node at (-1, 2) {\scalebox{0.45}{\scri{$1$}}};
\node at (0, 2) {\scalebox{0.45}{\scri{$2$}}};
\node at (1, 2) {\scalebox{0.45}{\scri{$\overline{3}$}}};
\node at (2, 2) {\scalebox{0.45}{\scri{$2$}}};
\node at (3, 2) {\scalebox{0.45}{\scri{$1$}}};
\node at (-2.5, 1) {\scalebox{0.45}{\scri{$1$}}};
\node at (-2.5, 0) {\scalebox{0.45}{\scri{$2$}}};
\node at (-2.5, -1) {\scalebox{0.45}{\scri{$3$}}};
\node at (-2.5, -2) {\scalebox{0.45}{\scri{$4$}}};
\node at (-2.5, -3) {\scalebox{0.45}{\scri{$5$}}};
\node at (-2.5, -4) {\scalebox{0.45}{\scri{$6$}}};
\node at (-2.5, -5) {\scalebox{0.45}{\scri{$7$}}};
\begin{scope}[every node/.style={minimum size=.1cm,inner sep=0mm,fill,circle}]
\draw[thick] (-2,-1) node {} -- (-1,0) node {}-- (1,-2) node {} -- (1,0) node {}-- (4,-3)node {};
\end{scope}
\node at (1, -6) {\scalebox{0.8}{\scri{\text{$Y_{1,2}Y^{-1}_{3,4}Y_{4,2}+Y_{1,2}Y_{3,2}Y^{-1}_{4,4}$}}}};
\end{tikzpicture}
} \\
& \raisebox{0.5cm}{
\begin{tikzpicture}[scale=.45]
\draw[help lines, color=gray!60, dashed] (-2.0,-5.0) grid (4.0, 1.0);
\draw[fill] (-1,1.5) circle (2pt)--(0,1.5) circle (2pt);
\draw[fill] (0,1.5) circle (2pt)--(1,1.5) circle (2pt);
\draw[fill] (1,1.5) circle (2pt)--(2,1.5) circle (2pt);
\draw[fill] (2,1.5) circle (2pt)--(3,1.5) circle (2pt);
\node at (-1, 2) {\scalebox{0.45}{\scri{$1$}}};
\node at (0, 2) {\scalebox{0.45}{\scri{$2$}}};
\node at (1, 2) {\scalebox{0.45}{\scri{$\overline{3}$}}};
\node at (2, 2) {\scalebox{0.45}{\scri{$2$}}};
\node at (3, 2) {\scalebox{0.45}{\scri{$1$}}};
\node at (-2.5, 1) {\scalebox{0.45}{\scri{$1$}}};
\node at (-2.5, 0) {\scalebox{0.45}{\scri{$2$}}};
\node at (-2.5, -1) {\scalebox{0.45}{\scri{$3$}}};
\node at (-2.5, -2) {\scalebox{0.45}{\scri{$4$}}};
\node at (-2.5, -3) {\scalebox{0.45}{\scri{$5$}}};
\node at (-2.5, -4) {\scalebox{0.45}{\scri{$6$}}};
\node at (-2.5, -5) {\scalebox{0.45}{\scri{$7$}}};
\begin{scope}[every node/.style={minimum size=.1cm,inner sep=0mm,fill,circle}]
\draw[thick] (-2,-1) node {} -- (-1,-2) node {}-- (0,-1) node {} -- (1,-2) node {}-- (1,0) node {}-- (4,-3)node {};
\end{scope}
\node at (1, -6) {\scalebox{0.65}{\scri{\text{$Y^{-1}_{1,4}Y_{2,3}Y^{-1}_{3,4}Y_{4,2}+Y^{-1}_{1,4}Y_{2,3}Y_{3,2}Y^{-1}_{4,4}$}}}};
\end{tikzpicture} 
}
\raisebox{0.5cm}{
\begin{tikzpicture}[scale=0.45]
\draw[help lines, color=gray!60, dashed] (-2.0,-5.0) grid (4.0, 1.0);
\draw[fill] (-1,1.5) circle (2pt)--(0,1.5) circle (2pt);
\draw[fill] (0,1.5) circle (2pt)--(1,1.5) circle (2pt);
\draw[fill] (1,1.5) circle (2pt)--(2,1.5) circle (2pt);
\draw[fill] (2,1.5) circle (2pt)--(3,1.5) circle (2pt);
\node at (-1, 2) {\scalebox{0.45}{\scri{$1$}}};
\node at (0, 2) {\scalebox{0.45}{\scri{$2$}}};
\node at (1, 2) {\scalebox{0.45}{\scri{$\overline{3}$}}};
\node at (2, 2) {\scalebox{0.45}{\scri{$2$}}};
\node at (3, 2) {\scalebox{0.45}{\scri{$1$}}};
\node at (-2.5, 1) {\scalebox{0.45}{\scri{$1$}}};
\node at (-2.5, 0) {\scalebox{0.45}{\scri{$2$}}};
\node at (-2.5, -1) {\scalebox{0.45}{\scri{$3$}}};
\node at (-2.5, -2) {\scalebox{0.45}{\scri{$4$}}};
\node at (-2.5, -3) {\scalebox{0.45}{\scri{$5$}}};
\node at (-2.5, -4) {\scalebox{0.45}{\scri{$6$}}};
\node at (-2.5, -5) {\scalebox{0.45}{\scri{$7$}}};
\begin{scope}[every node/.style={minimum size=.1cm,inner sep=0mm,fill,circle}]
\draw[thick] (-2,-1) node {} -- (-1,0) node {} -- (1,-2)node {} -- (2,-1)node {} -- (4,-3)node {};
\end{scope}
\node at (1, -6) {\scri{\text{$Y_{1,2}Y_{2,3}Y^{-1}_{3,4}Y^{-1}_{4,4}$}}};
\end{tikzpicture}
}
\raisebox{0.5cm}{
\begin{tikzpicture}[scale=.45]
\draw[help lines, color=gray!60, dashed] (-2.0,-5.0) grid (4.0, 1.0);
\draw[fill] (-1,1.5) circle (2pt)--(0,1.5) circle (2pt);
\draw[fill] (0,1.5) circle (2pt)--(1,1.5) circle (2pt);
\draw[fill] (1,1.5) circle (2pt)--(2,1.5) circle (2pt);
\draw[fill] (2,1.5) circle (2pt)--(3,1.5) circle (2pt);
\node at (-1, 2) {\scalebox{0.45}{\scri{$1$}}};
\node at (0, 2) {\scalebox{0.45}{\scri{$2$}}};
\node at (1, 2) {\scalebox{0.45}{\scri{$\overline{3}$}}};
\node at (2, 2) {\scalebox{0.45}{\scri{$2$}}};
\node at (3, 2) {\scalebox{0.45}{\scri{$1$}}};
\node at (-2.5, 1) {\scalebox{0.45}{\scri{$1$}}};
\node at (-2.5, 0) {\scalebox{0.45}{\scri{$2$}}};
\node at (-2.5, -1) {\scalebox{0.45}{\scri{$3$}}};
\node at (-2.5, -2) {\scalebox{0.45}{\scri{$4$}}};
\node at (-2.5, -3) {\scalebox{0.45}{\scri{$5$}}};
\node at (-2.5, -4) {\scalebox{0.45}{\scri{$6$}}};
\node at (-2.5, -5) {\scalebox{0.45}{\scri{$7$}}};
\begin{scope}[every node/.style={minimum size=.1cm,inner sep=0mm,fill,circle}]
\draw[thick] (-2,-1) node {} -- (0,-3) node {}-- (1,-2) node {}-- (1,0) node {} -- (4,-3)node {};
\end{scope}
\node at (1, -6) {\scalebox{0.7}{\scri{\text{$Y^{-1}_{2,5}Y_{3,2}Y_{3,4}+Y^{-1}_{2,5}Y_{4,2}Y_{4,4}$}}}};
\end{tikzpicture}
} 
\raisebox{0.5cm}{
\begin{tikzpicture}[scale=0.45]
\draw[help lines, color=gray!60, dashed] (-2.0,-5.0) grid (4.0, 1.0);
\draw[fill] (-1,1.5) circle (2pt)--(0,1.5) circle (2pt);
\draw[fill] (0,1.5) circle (2pt)--(1,1.5) circle (2pt);
\draw[fill] (1,1.5) circle (2pt)--(2,1.5) circle (2pt);
\draw[fill] (2,1.5) circle (2pt)--(3,1.5) circle (2pt);
\node at (-1, 2) {\scalebox{0.45}{\scri{$1$}}};
\node at (0, 2) {\scalebox{0.45}{\scri{$2$}}};
\node at (1, 2) {\scalebox{0.45}{\scri{$\overline{3}$}}};
\node at (2, 2) {\scalebox{0.45}{\scri{$2$}}};
\node at (3, 2) {\scalebox{0.45}{\scri{$1$}}};
\node at (-2.5, 1) {\scalebox{0.45}{\scri{$1$}}};
\node at (-2.5, 0) {\scalebox{0.45}{\scri{$2$}}};
\node at (-2.5, -1) {\scalebox{0.45}{\scri{$3$}}};
\node at (-2.5, -2) {\scalebox{0.45}{\scri{$4$}}};
\node at (-2.5, -3) {\scalebox{0.45}{\scri{$5$}}};
\node at (-2.5, -4) {\scalebox{0.45}{\scri{$6$}}};
\node at (-2.5, -5) {\scalebox{0.45}{\scri{$7$}}};
\begin{scope}[every node/.style={minimum size=.1cm,inner sep=0mm,fill,circle}]
\draw[thick] (-2,-1) node {} -- (-1,-2) node {} -- (0,-1)node {} -- (1,-2)node {}-- (2,-1)node {} -- (4,-3)node {};
\end{scope}
\node at (1, -6) {\scri{\text{$Y^{-1}_{1,4}Y^{2}_{2,3}Y^{-1}_{3,4}Y^{-1}_{4,4}$}}};
\end{tikzpicture}
} \\
&
\raisebox{0.5cm}{
\begin{tikzpicture}[scale=0.45]
\draw[help lines, color=gray!60, dashed] (-2.0,-5.0) grid (4.0, 1.0);
\draw[fill] (-1,1.5) circle (2pt)--(0,1.5) circle (2pt);
\draw[fill] (0,1.5) circle (2pt)--(1,1.5) circle (2pt);
\draw[fill] (1,1.5) circle (2pt)--(2,1.5) circle (2pt);
\draw[fill] (2,1.5) circle (2pt)--(3,1.5) circle (2pt);
\node at (-1, 2) {\scalebox{0.45}{\scri{$1$}}};
\node at (0, 2) {\scalebox{0.45}{\scri{$2$}}};
\node at (1, 2) {\scalebox{0.45}{\scri{$\overline{3}$}}};
\node at (2, 2) {\scalebox{0.45}{\scri{$2$}}};
\node at (3, 2) {\scalebox{0.45}{\scri{$1$}}};
\node at (-2.5, 1) {\scalebox{0.45}{\scri{$1$}}};
\node at (-2.5, 0) {\scalebox{0.45}{\scri{$2$}}};
\node at (-2.5, -1) {\scalebox{0.45}{\scri{$3$}}};
\node at (-2.5, -2) {\scalebox{0.45}{\scri{$4$}}};
\node at (-2.5, -3) {\scalebox{0.45}{\scri{$5$}}};
\node at (-2.5, -4) {\scalebox{0.45}{\scri{$6$}}};
\node at (-2.5, -5) {\scalebox{0.45}{\scri{$7$}}};
\begin{scope}[every node/.style={minimum size=.1cm,inner sep=0mm,fill,circle}]
\draw[thick] (-2,-1) node {} -- (-1,0) node {} -- (2,-3)node {} -- (3,-2)node {} -- (4,-3)node {};
\end{scope}
\node at (1, -6) {\scri{\text{$Y_{1,2}Y_{1,4}Y^{-1}_{2,5}$}}};
\end{tikzpicture}
}
\raisebox{0.5cm}{
\begin{tikzpicture}[scale=.45]
\draw[help lines, color=gray!60, dashed] (-2.0,-5.0) grid (4.0, 1.0);
\draw[fill] (-1,1.5) circle (2pt)--(0,1.5) circle (2pt);
\draw[fill] (0,1.5) circle (2pt)--(1,1.5) circle (2pt);
\draw[fill] (1,1.5) circle (2pt)--(2,1.5) circle (2pt);
\draw[fill] (2,1.5) circle (2pt)--(3,1.5) circle (2pt);
\node at (-1, 2) {\scalebox{0.45}{\scri{$1$}}};
\node at (0, 2) {\scalebox{0.45}{\scri{$2$}}};
\node at (1, 2) {\scalebox{0.45}{\scri{$\overline{3}$}}};
\node at (2, 2) {\scalebox{0.45}{\scri{$2$}}};
\node at (3, 2) {\scalebox{0.45}{\scri{$1$}}};
\node at (-2.5, 1) {\scalebox{0.45}{\scri{$1$}}};
\node at (-2.5, 0) {\scalebox{0.45}{\scri{$2$}}};
\node at (-2.5, -1) {\scalebox{0.45}{\scri{$3$}}};
\node at (-2.5, -2) {\scalebox{0.45}{\scri{$4$}}};
\node at (-2.5, -3) {\scalebox{0.45}{\scri{$5$}}};
\node at (-2.5, -4) {\scalebox{0.45}{\scri{$6$}}};
\node at (-2.5, -5) {\scalebox{0.45}{\scri{$7$}}};
\begin{scope}[every node/.style={minimum size=.1cm,inner sep=0mm,fill,circle}]
\draw[thick] (-2,-1) node {} -- (1,-4) node {}-- (1,0) node {}-- (4,-3)node {};
\end{scope}
\node at (1, -6) {\scalebox{0.8}{\scri{\text{$Y_{3,2}Y^{-1}_{3,6}+Y_{4,2}Y^{-1}_{4,6}$}}}};
\end{tikzpicture}
}
\raisebox{0.5cm}{
\begin{tikzpicture}[scale=0.45]
\draw[help lines, color=gray!60, dashed] (-2.0,-5.0) grid (4.0, 1.0);
\draw[fill] (-1,1.5) circle (2pt)--(0,1.5) circle (2pt);
\draw[fill] (0,1.5) circle (2pt)--(1,1.5) circle (2pt);
\draw[fill] (1,1.5) circle (2pt)--(2,1.5) circle (2pt);
\draw[fill] (2,1.5) circle (2pt)--(3,1.5) circle (2pt);
\node at (-1, 2) {\scalebox{0.45}{\scri{$1$}}};
\node at (0, 2) {\scalebox{0.45}{\scri{$2$}}};
\node at (1, 2) {\scalebox{0.45}{\scri{$\overline{3}$}}};
\node at (2, 2) {\scalebox{0.45}{\scri{$2$}}};
\node at (3, 2) {\scalebox{0.45}{\scri{$1$}}};
\node at (-2.5, 1) {\scalebox{0.45}{\scri{$1$}}};
\node at (-2.5, 0) {\scalebox{0.45}{\scri{$2$}}};
\node at (-2.5, -1) {\scalebox{0.45}{\scri{$3$}}};
\node at (-2.5, -2) {\scalebox{0.45}{\scri{$4$}}};
\node at (-2.5, -3) {\scalebox{0.45}{\scri{$5$}}};
\node at (-2.5, -4) {\scalebox{0.45}{\scri{$6$}}};
\node at (-2.5, -5) {\scalebox{0.45}{\scri{$7$}}};
\begin{scope}[every node/.style={minimum size=.1cm,inner sep=0mm,fill,circle}]
\draw[thick] (-2,-1) node {} -- (0,-3) node {} -- (2,-1)node {}-- (4,-3)node {};
\end{scope}
\node at (1, -6) {\scri{\text{$Y_{2,3}Y^{-1}_{2,5}$}}};
\end{tikzpicture}
}
\raisebox{0.5cm}{
\begin{tikzpicture}[scale=0.45]
\draw[help lines, color=gray!60, dashed] (-2.0,-5.0) grid (4.0, 1.0);
\draw[fill] (-1,1.5) circle (2pt)--(0,1.5) circle (2pt);
\draw[fill] (0,1.5) circle (2pt)--(1,1.5) circle (2pt);
\draw[fill] (1,1.5) circle (2pt)--(2,1.5) circle (2pt);
\draw[fill] (2,1.5) circle (2pt)--(3,1.5) circle (2pt);
\node at (-1, 2) {\scalebox{0.45}{\scri{$1$}}};
\node at (0, 2) {\scalebox{0.45}{\scri{$2$}}};
\node at (1, 2) {\scalebox{0.45}{\scri{$\overline{3}$}}};
\node at (2, 2) {\scalebox{0.45}{\scri{$2$}}};
\node at (3, 2) {\scalebox{0.45}{\scri{$1$}}};
\node at (-2.5, 1) {\scalebox{0.45}{\scri{$1$}}};
\node at (-2.5, 0) {\scalebox{0.45}{\scri{$2$}}};
\node at (-2.5, -1) {\scalebox{0.45}{\scri{$3$}}};
\node at (-2.5, -2) {\scalebox{0.45}{\scri{$4$}}};
\node at (-2.5, -3) {\scalebox{0.45}{\scri{$5$}}};
\node at (-2.5, -4) {\scalebox{0.45}{\scri{$6$}}};
\node at (-2.5, -5) {\scalebox{0.45}{\scri{$7$}}};
\begin{scope}[every node/.style={minimum size=.1cm,inner sep=0mm,fill,circle}]
\draw[thick] (-2,-1) node {} -- (-1,-2) node {} -- (0,-1) node {}-- (2,-3) node {}-- (3,-2) node {} -- (4,-3)node {};
\end{scope}
\node at (1, -6) {\scri{\text{$Y^{-1}_{1,4}Y_{2,3}Y^{-1}_{2,5}Y_{1,4}=Y_{2,3}Y^{-1}_{2,5}$}}};
\end{tikzpicture}
} \\
&
\raisebox{0.5cm}{
\begin{tikzpicture}[scale=0.45]
\draw[help lines, color=gray!60, dashed] (-2.0,-5.0) grid (4.0, 1.0);
\draw[fill] (-1,1.5) circle (2pt)--(0,1.5) circle (2pt);
\draw[fill] (0,1.5) circle (2pt)--(1,1.5) circle (2pt);
\draw[fill] (1,1.5) circle (2pt)--(2,1.5) circle (2pt);
\draw[fill] (2,1.5) circle (2pt)--(3,1.5) circle (2pt);
\node at (-1, 2) {\scalebox{0.45}{\scri{$1$}}};
\node at (0, 2) {\scalebox{0.45}{\scri{$2$}}};
\node at (1, 2) {\scalebox{0.45}{\scri{$\overline{3}$}}};
\node at (2, 2) {\scalebox{0.45}{\scri{$2$}}};
\node at (3, 2) {\scalebox{0.45}{\scri{$1$}}};
\node at (-2.5, 1) {\scalebox{0.45}{\scri{$1$}}};
\node at (-2.5, 0) {\scalebox{0.45}{\scri{$2$}}};
\node at (-2.5, -1) {\scalebox{0.45}{\scri{$3$}}};
\node at (-2.5, -2) {\scalebox{0.45}{\scri{$4$}}};
\node at (-2.5, -3) {\scalebox{0.45}{\scri{$5$}}};
\node at (-2.5, -4) {\scalebox{0.45}{\scri{$6$}}};
\node at (-2.5, -5) {\scalebox{0.45}{\scri{$7$}}};
\begin{scope}[every node/.style={minimum size=.1cm,inner sep=0mm,fill,circle}]
\draw[thick] (-2,-1) node {} -- (-1,0) node {} -- (3,-4)node {}-- (4,-3)node {};
\end{scope}
\node at (1, -6) {\scri{\text{$Y_{1,2}Y^{-1}_{1,6}$}}};
\end{tikzpicture}
}
\raisebox{0.5cm}{
\begin{tikzpicture}[scale=.45]
\draw[help lines, color=gray!60, dashed] (-2.0,-5.0) grid (4.0, 1.0);
\draw[fill] (-1,1.5) circle (2pt)--(0,1.5) circle (2pt);
\draw[fill] (0,1.5) circle (2pt)--(1,1.5) circle (2pt);
\draw[fill] (1,1.5) circle (2pt)--(2,1.5) circle (2pt);
\draw[fill] (2,1.5) circle (2pt)--(3,1.5) circle (2pt);
\node at (-1, 2) {\scalebox{0.45}{\scri{$1$}}};
\node at (0, 2) {\scalebox{0.45}{\scri{$2$}}};
\node at (1, 2) {\scalebox{0.45}{\scri{$\overline{3}$}}};
\node at (2, 2) {\scalebox{0.45}{\scri{$2$}}};
\node at (3, 2) {\scalebox{0.45}{\scri{$1$}}};
\node at (-2.5, 1) {\scalebox{0.45}{\scri{$1$}}};
\node at (-2.5, 0) {\scalebox{0.45}{\scri{$2$}}};
\node at (-2.5, -1) {\scalebox{0.45}{\scri{$3$}}};
\node at (-2.5, -2) {\scalebox{0.45}{\scri{$4$}}};
\node at (-2.5, -3) {\scalebox{0.45}{\scri{$5$}}};
\node at (-2.5, -4) {\scalebox{0.45}{\scri{$6$}}};
\node at (-2.5, -5) {\scalebox{0.45}{\scri{$7$}}};
\begin{scope}[every node/.style={minimum size=.1cm,inner sep=0mm,fill,circle}]
\draw[thick] (-2,-1) node {} -- (1,-4) node {}-- (1,-2) node {}-- (2,-1) node {}-- (4,-3)node {};
\end{scope}
\node at (1, -6) {\scalebox{0.8}{\scri{\text{$Y_{2,3}Y^{-1}_{3,4}Y^{-1}_{3,6}+Y_{2,3}Y^{-1}_{4,4}Y^{-1}_{4,6}$}}}};
\end{tikzpicture}
}
\raisebox{0.5cm}{
\begin{tikzpicture}[scale=0.45]
\draw[help lines, color=gray!60, dashed] (-2.0,-5.0) grid (4.0, 1.0);
\draw[fill] (-1,1.5) circle (2pt)--(0,1.5) circle (2pt);
\draw[fill] (0,1.5) circle (2pt)--(1,1.5) circle (2pt);
\draw[fill] (1,1.5) circle (2pt)--(2,1.5) circle (2pt);
\draw[fill] (2,1.5) circle (2pt)--(3,1.5) circle (2pt);
\node at (-1, 2) {\scalebox{0.45}{\scri{$1$}}};
\node at (0, 2) {\scalebox{0.45}{\scri{$2$}}};
\node at (1, 2) {\scalebox{0.45}{\scri{$\overline{3}$}}};
\node at (2, 2) {\scalebox{0.45}{\scri{$2$}}};
\node at (3, 2) {\scalebox{0.45}{\scri{$1$}}};
\node at (-2.5, 1) {\scalebox{0.45}{\scri{$1$}}};
\node at (-2.5, 0) {\scalebox{0.45}{\scri{$2$}}};
\node at (-2.5, -1) {\scalebox{0.45}{\scri{$3$}}};
\node at (-2.5, -2) {\scalebox{0.45}{\scri{$4$}}};
\node at (-2.5, -3) {\scalebox{0.45}{\scri{$5$}}};
\node at (-2.5, -4) {\scalebox{0.45}{\scri{$6$}}};
\node at (-2.5, -5) {\scalebox{0.45}{\scri{$7$}}};
\begin{scope}[every node/.style={minimum size=.1cm,inner sep=0mm,fill,circle}]
\draw[thick] (-2,-1) node {} -- (0,-3) node {} -- (1,-2) node {}-- (2,-3) node {}-- (3,-2) node {} -- (4,-3)node {};
\end{scope}
\node at (1, -6) {\scri{\text{$Y_{1,4}Y^{-2}_{2,5}Y_{3,4}Y_{4,4}$}}};
\end{tikzpicture}
}
\raisebox{0.5cm}{
\begin{tikzpicture}[scale=0.45]
\draw[help lines, color=gray!60, dashed] (-2.0,-5.0) grid (4.0, 1.0);
\draw[fill] (-1,1.5) circle (2pt)--(0,1.5) circle (2pt);
\draw[fill] (0,1.5) circle (2pt)--(1,1.5) circle (2pt);
\draw[fill] (1,1.5) circle (2pt)--(2,1.5) circle (2pt);
\draw[fill] (2,1.5) circle (2pt)--(3,1.5) circle (2pt);
\node at (-1, 2) {\scalebox{0.45}{\scri{$1$}}};
\node at (0, 2) {\scalebox{0.45}{\scri{$2$}}};
\node at (1, 2) {\scalebox{0.45}{\scri{$\overline{3}$}}};
\node at (2, 2) {\scalebox{0.45}{\scri{$2$}}};
\node at (3, 2) {\scalebox{0.45}{\scri{$1$}}};
\node at (-2.5, 1) {\scalebox{0.45}{\scri{$1$}}};
\node at (-2.5, 0) {\scalebox{0.45}{\scri{$2$}}};
\node at (-2.5, -1) {\scalebox{0.45}{\scri{$3$}}};
\node at (-2.5, -2) {\scalebox{0.45}{\scri{$4$}}};
\node at (-2.5, -3) {\scalebox{0.45}{\scri{$5$}}};
\node at (-2.5, -4) {\scalebox{0.45}{\scri{$6$}}};
\node at (-2.5, -5) {\scalebox{0.45}{\scri{$7$}}};
\begin{scope}[every node/.style={minimum size=.1cm,inner sep=0mm,fill,circle}]
\draw[thick] (-2,-1) node {} -- (-1,-2) node {} -- (0,-1)node {}-- (3,-4) node {}-- (4,-3)node {};
\end{scope}
\node at (1, -6) {\scri{\text{$Y^{-1}_{1,4}Y^{-1}_{1,6}Y_{2,3}$}}};
\end{tikzpicture}
} 
\end{align*}
\caption{Monomials or binomials associated to paths in $\mathscr{P}_{2,1}$, part I.} \label{q-character of 2_1 1}
\end{figure}

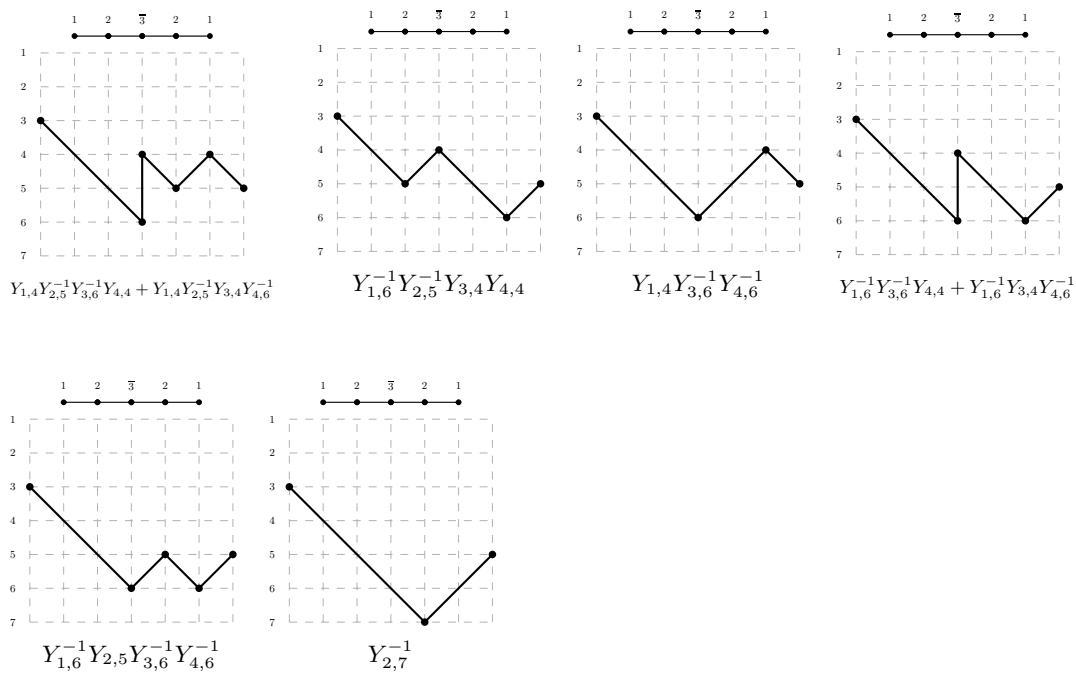
\begin{figure}
\begin{align*}
&
\raisebox{0.5cm}{
\begin{tikzpicture}[scale=.45]
\draw[help lines, color=gray!60, dashed] (-2.0,-5.0) grid (4.0, 1.0);
\draw[fill] (-1,1.5) circle (2pt)--(0,1.5) circle (2pt);
\draw[fill] (0,1.5) circle (2pt)--(1,1.5) circle (2pt);
\draw[fill] (1,1.5) circle (2pt)--(2,1.5) circle (2pt);
\draw[fill] (2,1.5) circle (2pt)--(3,1.5) circle (2pt);
\node at (-1, 2) {\scalebox{0.45}{\scri{$1$}}};
\node at (0, 2) {\scalebox{0.45}{\scri{$2$}}};
\node at (1, 2) {\scalebox{0.45}{\scri{$\overline{3}$}}};
\node at (2, 2) {\scalebox{0.45}{\scri{$2$}}};
\node at (3, 2) {\scalebox{0.45}{\scri{$1$}}};
\node at (-2.5, 1) {\scalebox{0.45}{\scri{$1$}}};
\node at (-2.5, 0) {\scalebox{0.45}{\scri{$2$}}};
\node at (-2.5, -1) {\scalebox{0.45}{\scri{$3$}}};
\node at (-2.5, -2) {\scalebox{0.45}{\scri{$4$}}};
\node at (-2.5, -3) {\scalebox{0.45}{\scri{$5$}}};
\node at (-2.5, -4) {\scalebox{0.45}{\scri{$6$}}};
\node at (-2.5, -5) {\scalebox{0.45}{\scri{$7$}}};
\begin{scope}[every node/.style={minimum size=.1cm,inner sep=0mm,fill,circle}]
\draw[thick] (-2,-1) node {} -- (1,-4) node {}-- (1,-2) node {}-- (2,-3) node {}-- (3,-2) node {} -- (4,-3)node {};
\end{scope}
\node at (1, -6) {\scalebox{0.7}{\scri{\text{$Y_{1,4}Y^{-1}_{2,5}Y^{-1}_{3,6}Y_{4,4}+Y_{1,4}Y^{-1}_{2,5}Y_{3,4}Y^{-1}_{4,6}$}}}};
\end{tikzpicture}
}
\raisebox{0.5cm}{
\begin{tikzpicture}[scale=0.45]
\draw[help lines, color=gray!60, dashed] (-2.0,-5.0) grid (4.0, 1.0);
\draw[fill] (-1,1.5) circle (2pt)--(0,1.5) circle (2pt);
\draw[fill] (0,1.5) circle (2pt)--(1,1.5) circle (2pt);
\draw[fill] (1,1.5) circle (2pt)--(2,1.5) circle (2pt);
\draw[fill] (2,1.5) circle (2pt)--(3,1.5) circle (2pt);
\node at (-1, 2) {\scalebox{0.45}{\scri{$1$}}};
\node at (0, 2) {\scalebox{0.45}{\scri{$2$}}};
\node at (1, 2) {\scalebox{0.45}{\scri{$\overline{3}$}}};
\node at (2, 2) {\scalebox{0.45}{\scri{$2$}}};
\node at (3, 2) {\scalebox{0.45}{\scri{$1$}}};
\node at (-2.5, 1) {\scalebox{0.45}{\scri{$1$}}};
\node at (-2.5, 0) {\scalebox{0.45}{\scri{$2$}}};
\node at (-2.5, -1) {\scalebox{0.45}{\scri{$3$}}};
\node at (-2.5, -2) {\scalebox{0.45}{\scri{$4$}}};
\node at (-2.5, -3) {\scalebox{0.45}{\scri{$5$}}};
\node at (-2.5, -4) {\scalebox{0.45}{\scri{$6$}}};
\node at (-2.5, -5) {\scalebox{0.45}{\scri{$7$}}};
\begin{scope}[every node/.style={minimum size=.1cm,inner sep=0mm,fill,circle}]
\draw[thick] (-2,-1) node {} -- (0,-3) node {} -- (1,-2) node {}-- (3,-4) node {} -- (4,-3)node {};
\end{scope}
\node at (1, -6) {\scri{\text{$Y^{-1}_{1,6}Y^{-1}_{2,5}Y_{3,4}Y_{4,4}$}}};
\end{tikzpicture}
}
\raisebox{0.5cm}{
\begin{tikzpicture}[scale=0.45]
\draw[help lines, color=gray!60, dashed] (-2.0,-5.0) grid (4.0, 1.0);
\draw[fill] (-1,1.5) circle (2pt)--(0,1.5) circle (2pt);
\draw[fill] (0,1.5) circle (2pt)--(1,1.5) circle (2pt);
\draw[fill] (1,1.5) circle (2pt)--(2,1.5) circle (2pt);
\draw[fill] (2,1.5) circle (2pt)--(3,1.5) circle (2pt);
\node at (-1, 2) {\scalebox{0.45}{\scri{$1$}}};
\node at (0, 2) {\scalebox{0.45}{\scri{$2$}}};
\node at (1, 2) {\scalebox{0.45}{\scri{$\overline{3}$}}};
\node at (2, 2) {\scalebox{0.45}{\scri{$2$}}};
\node at (3, 2) {\scalebox{0.45}{\scri{$1$}}};
\node at (-2.5, 1) {\scalebox{0.45}{\scri{$1$}}};
\node at (-2.5, 0) {\scalebox{0.45}{\scri{$2$}}};
\node at (-2.5, -1) {\scalebox{0.45}{\scri{$3$}}};
\node at (-2.5, -2) {\scalebox{0.45}{\scri{$4$}}};
\node at (-2.5, -3) {\scalebox{0.45}{\scri{$5$}}};
\node at (-2.5, -4) {\scalebox{0.45}{\scri{$6$}}};
\node at (-2.5, -5) {\scalebox{0.45}{\scri{$7$}}};
\begin{scope}[every node/.style={minimum size=.1cm,inner sep=0mm,fill,circle}]
\draw[thick] (-2,-1) node {} -- (1,-4) node {} -- (3,-2) node {} -- (4,-3)node {};
\end{scope}
\node at (1, -6) {\scri{\text{$Y_{1,4}Y^{-1}_{3,6}Y^{-1}_{4,6}$}}};
\end{tikzpicture}
}
\raisebox{0.5cm}{
\begin{tikzpicture}[scale=.45]
\draw[help lines, color=gray!60, dashed] (-2.0,-5.0) grid (4.0, 1.0);
\draw[fill] (-1,1.5) circle (2pt)--(0,1.5) circle (2pt);
\draw[fill] (0,1.5) circle (2pt)--(1,1.5) circle (2pt);
\draw[fill] (1,1.5) circle (2pt)--(2,1.5) circle (2pt);
\draw[fill] (2,1.5) circle (2pt)--(3,1.5) circle (2pt);
\node at (-1, 2) {\scalebox{0.45}{\scri{$1$}}};
\node at (0, 2) {\scalebox{0.45}{\scri{$2$}}};
\node at (1, 2) {\scalebox{0.45}{\scri{$\overline{3}$}}};
\node at (2, 2) {\scalebox{0.45}{\scri{$2$}}};
\node at (3, 2) {\scalebox{0.45}{\scri{$1$}}};
\node at (-2.5, 1) {\scalebox{0.45}{\scri{$1$}}};
\node at (-2.5, 0) {\scalebox{0.45}{\scri{$2$}}};
\node at (-2.5, -1) {\scalebox{0.45}{\scri{$3$}}};
\node at (-2.5, -2) {\scalebox{0.45}{\scri{$4$}}};
\node at (-2.5, -3) {\scalebox{0.45}{\scri{$5$}}};
\node at (-2.5, -4) {\scalebox{0.45}{\scri{$6$}}};
\node at (-2.5, -5) {\scalebox{0.45}{\scri{$7$}}};
\begin{scope}[every node/.style={minimum size=.1cm,inner sep=0mm,fill,circle}]
\draw[thick] (-2,-1) node {} -- (1,-4) node {}-- (1,-2) node {}-- (3,-4) node {} -- (4,-3)node {};
\end{scope}
\node at (1, -6) {\scalebox{0.8}{\scri{\text{$Y^{-1}_{1,6}Y^{-1}_{3,6}Y_{4,4}+Y^{-1}_{1,6}Y_{3,4}Y^{-1}_{4,6}$}}}};
\end{tikzpicture}
} \\
& 
\raisebox{0.5cm}{
\begin{tikzpicture}[scale=0.45]
\draw[help lines, color=gray!60, dashed] (-2.0,-5.0) grid (4.0, 1.0);
\draw[fill] (-1,1.5) circle (2pt)--(0,1.5) circle (2pt);
\draw[fill] (0,1.5) circle (2pt)--(1,1.5) circle (2pt);
\draw[fill] (1,1.5) circle (2pt)--(2,1.5) circle (2pt);
\draw[fill] (2,1.5) circle (2pt)--(3,1.5) circle (2pt);
\node at (-1, 2) {\scalebox{0.45}{\scri{$1$}}};
\node at (0, 2) {\scalebox{0.45}{\scri{$2$}}};
\node at (1, 2) {\scalebox{0.45}{\scri{$\overline{3}$}}};
\node at (2, 2) {\scalebox{0.45}{\scri{$2$}}};
\node at (3, 2) {\scalebox{0.45}{\scri{$1$}}};
\node at (-2.5, 1) {\scalebox{0.45}{\scri{$1$}}};
\node at (-2.5, 0) {\scalebox{0.45}{\scri{$2$}}};
\node at (-2.5, -1) {\scalebox{0.45}{\scri{$3$}}};
\node at (-2.5, -2) {\scalebox{0.45}{\scri{$4$}}};
\node at (-2.5, -3) {\scalebox{0.45}{\scri{$5$}}};
\node at (-2.5, -4) {\scalebox{0.45}{\scri{$6$}}};
\node at (-2.5, -5) {\scalebox{0.45}{\scri{$7$}}};
\begin{scope}[every node/.style={minimum size=.1cm,inner sep=0mm,fill,circle}]
\draw[thick] (-2,-1) node {} -- (1,-4) node {} -- (2,-3) node {}-- (3,-4) node {} -- (4,-3)node {};
\end{scope}
\node at (1, -6) {\scri{\text{$Y^{-1}_{1,6}Y_{2,5}Y^{-1}_{3,6}Y^{-1}_{4,6}$}}};
\end{tikzpicture}
}
\raisebox{0.5cm}{
\begin{tikzpicture}[scale=0.45]
\draw[help lines, color=gray!60, dashed] (-2.0,-5.0) grid (4.0, 1.0);
\draw[fill] (-1,1.5) circle (2pt)--(0,1.5) circle (2pt);
\draw[fill] (0,1.5) circle (2pt)--(1,1.5) circle (2pt);
\draw[fill] (1,1.5) circle (2pt)--(2,1.5) circle (2pt);
\draw[fill] (2,1.5) circle (2pt)--(3,1.5) circle (2pt);
\node at (-1, 2) {\scalebox{0.45}{\scri{$1$}}};
\node at (0, 2) {\scalebox{0.45}{\scri{$2$}}};
\node at (1, 2) {\scalebox{0.45}{\scri{$\overline{3}$}}};
\node at (2, 2) {\scalebox{0.45}{\scri{$2$}}};
\node at (3, 2) {\scalebox{0.45}{\scri{$1$}}};
\node at (-2.5, 1) {\scalebox{0.45}{\scri{$1$}}};
\node at (-2.5, 0) {\scalebox{0.45}{\scri{$2$}}};
\node at (-2.5, -1) {\scalebox{0.45}{\scri{$3$}}};
\node at (-2.5, -2) {\scalebox{0.45}{\scri{$4$}}};
\node at (-2.5, -3) {\scalebox{0.45}{\scri{$5$}}};
\node at (-2.5, -4) {\scalebox{0.45}{\scri{$6$}}};
\node at (-2.5, -5) {\scalebox{0.45}{\scri{$7$}}};
\begin{scope}[every node/.style={minimum size=.1cm,inner sep=0mm,fill,circle}]
\draw[thick] (-2,-1) node {} -- (2,-5) node {} -- (4,-3)node {};
\end{scope}
\node at (1, -6) {\scri{\text{$Y^{-1}_{2,7}$}}};
\end{tikzpicture}
}
\\
\end{align*}
\caption{Monomials or binomials associated to paths in $\mathscr{P}_{2,1}$, part II.} \label{q-character of 2_1 2}
\end{figure}

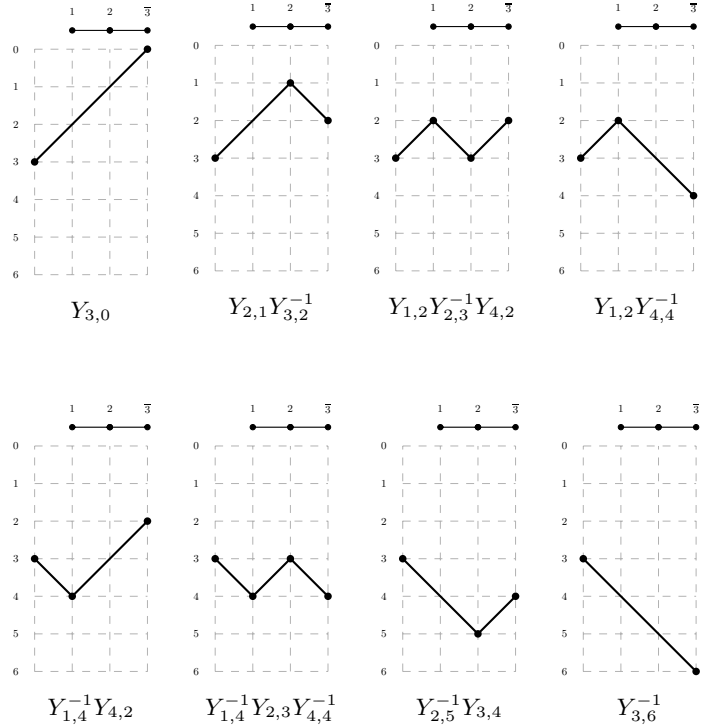
\begin{figure}
\begin{align*}
&
\raisebox{0.5cm}{
\begin{tikzpicture}[scale=.5]
\draw[help lines, color=gray!60, dashed] (-2.0,-5.0) grid (1.0, 1.0);
\draw[fill] (-1,1.5) circle (2pt)--(0,1.5) circle (2pt);
\draw[fill] (0,1.5) circle (2pt)--(1,1.5) circle (2pt);
\node at (-1, 2) {\scalebox{0.45}{\scri{$1$}}};
\node at (0, 2) {\scalebox{0.45}{\scri{$2$}}};
\node at (1, 2) {\scalebox{0.45}{\scri{$\overline{3}$}}};
\node at (-2.5, 1) {\scalebox{0.45}{\scri{$0$}}};
\node at (-2.5, 0) {\scalebox{0.45}{\scri{$1$}}};
\node at (-2.5, -1) {\scalebox{0.45}{\scri{$2$}}};
\node at (-2.5, -2) {\scalebox{0.45}{\scri{$3$}}};
\node at (-2.5, -3) {\scalebox{0.45}{\scri{$4$}}};
\node at (-2.5, -4) {\scalebox{0.45}{\scri{$5$}}};
\node at (-2.5, -5) {\scalebox{0.45}{\scri{$6$}}};
\begin{scope}[every node/.style={minimum size=.1cm,inner sep=0mm,fill,circle}]
\draw[thick] (-2,-2) node {} -- (1,1) node {};
\end{scope}
\node at (-0.5, -6) {\scri{\text{$Y_{3,0}$}}};
\end{tikzpicture}
}
\raisebox{0.5cm}{
\begin{tikzpicture}[scale=.5]
\draw[help lines, color=gray!60, dashed] (-2.0,-5.0) grid (1.0, 1.0);
\draw[fill] (-1,1.5) circle (2pt)--(0,1.5) circle (2pt);
\draw[fill] (0,1.5) circle (2pt)--(1,1.5) circle (2pt);
\node at (-1, 2) {\scalebox{0.45}{\scri{$1$}}};
\node at (0, 2) {\scalebox{0.45}{\scri{$2$}}};
\node at (1, 2) {\scalebox{0.45}{\scri{$\overline{3}$}}};
\node at (-2.5, 1) {\scalebox{0.45}{\scri{$0$}}};
\node at (-2.5, 0) {\scalebox{0.45}{\scri{$1$}}};
\node at (-2.5, -1) {\scalebox{0.45}{\scri{$2$}}};
\node at (-2.5, -2) {\scalebox{0.45}{\scri{$3$}}};
\node at (-2.5, -3) {\scalebox{0.45}{\scri{$4$}}};
\node at (-2.5, -4) {\scalebox{0.45}{\scri{$5$}}};
\node at (-2.5, -5) {\scalebox{0.45}{\scri{$6$}}};
\begin{scope}[every node/.style={minimum size=.1cm,inner sep=0mm,fill,circle}]
\draw[thick] (-2,-2) node {} -- (0,0) node {}-- (1,-1) node {};
\end{scope}
\node at (-0.5, -6) {\scri{\text{$Y_{2,1}Y^{-1}_{3,2}$}}};
\end{tikzpicture}
}
\raisebox{0.5cm}{
\begin{tikzpicture}[scale=.5]
\draw[help lines, color=gray!60, dashed] (-2.0,-5.0) grid (1.0, 1.0);
\draw[fill] (-1,1.5) circle (2pt)--(0,1.5) circle (2pt);
\draw[fill] (0,1.5) circle (2pt)--(1,1.5) circle (2pt);
\node at (-1, 2) {\scalebox{0.45}{\scri{$1$}}};
\node at (0, 2) {\scalebox{0.45}{\scri{$2$}}};
\node at (1, 2) {\scalebox{0.45}{\scri{$\overline{3}$}}};
\node at (-2.5, 1) {\scalebox{0.45}{\scri{$0$}}};
\node at (-2.5, 0) {\scalebox{0.45}{\scri{$1$}}};
\node at (-2.5, -1) {\scalebox{0.45}{\scri{$2$}}};
\node at (-2.5, -2) {\scalebox{0.45}{\scri{$3$}}};
\node at (-2.5, -3) {\scalebox{0.45}{\scri{$4$}}};
\node at (-2.5, -4) {\scalebox{0.45}{\scri{$5$}}};
\node at (-2.5, -5) {\scalebox{0.45}{\scri{$6$}}};
\begin{scope}[every node/.style={minimum size=.1cm,inner sep=0mm,fill,circle}]
\draw[thick] (-2,-2) node {} -- (-1,-1) node {}-- (0,-2) node {}-- (1,-1) node {};
\end{scope}
\node at (-0.5, -6) {\scri{\text{$Y_{1,2}Y^{-1}_{2,3}Y_{4,2}$}}};
\end{tikzpicture}
}
\raisebox{0.5cm}{
\begin{tikzpicture}[scale=.5]
\draw[help lines, color=gray!60, dashed] (-2.0,-5.0) grid (1.0, 1.0);
\draw[fill] (-1,1.5) circle (2pt)--(0,1.5) circle (2pt);
\draw[fill] (0,1.5) circle (2pt)--(1,1.5) circle (2pt);
\node at (-1, 2) {\scalebox{0.45}{\scri{$1$}}};
\node at (0, 2) {\scalebox{0.45}{\scri{$2$}}};
\node at (1, 2) {\scalebox{0.45}{\scri{$\overline{3}$}}};
\node at (-2.5, 1) {\scalebox{0.45}{\scri{$0$}}};
\node at (-2.5, 0) {\scalebox{0.45}{\scri{$1$}}};
\node at (-2.5, -1) {\scalebox{0.45}{\scri{$2$}}};
\node at (-2.5, -2) {\scalebox{0.45}{\scri{$3$}}};
\node at (-2.5, -3) {\scalebox{0.45}{\scri{$4$}}};
\node at (-2.5, -4) {\scalebox{0.45}{\scri{$5$}}};
\node at (-2.5, -5) {\scalebox{0.45}{\scri{$6$}}};
\begin{scope}[every node/.style={minimum size=.1cm,inner sep=0mm,fill,circle}]
\draw[thick] (-2,-2) node {} -- (-1,-1) node {}-- (1,-3) node {};
\end{scope}
\node at (-0.5, -6) {\scri{\text{$Y_{1,2}Y^{-1}_{4,4}$}}};
\end{tikzpicture}
}
\\
&
\raisebox{0.5cm}{
\begin{tikzpicture}[scale=.5]
\draw[help lines, color=gray!60, dashed] (-2.0,-5.0) grid (1.0, 1.0);
\draw[fill] (-1,1.5) circle (2pt)--(0,1.5) circle (2pt);
\draw[fill] (0,1.5) circle (2pt)--(1,1.5) circle (2pt);
\node at (-1, 2) {\scalebox{0.45}{\scri{$1$}}};
\node at (0, 2) {\scalebox{0.45}{\scri{$2$}}};
\node at (1, 2) {\scalebox{0.45}{\scri{$\overline{3}$}}};
\node at (-2.5, 1) {\scalebox{0.45}{\scri{$0$}}};
\node at (-2.5, 0) {\scalebox{0.45}{\scri{$1$}}};
\node at (-2.5, -1) {\scalebox{0.45}{\scri{$2$}}};
\node at (-2.5, -2) {\scalebox{0.45}{\scri{$3$}}};
\node at (-2.5, -3) {\scalebox{0.45}{\scri{$4$}}};
\node at (-2.5, -4) {\scalebox{0.45}{\scri{$5$}}};
\node at (-2.5, -5) {\scalebox{0.45}{\scri{$6$}}};
\begin{scope}[every node/.style={minimum size=.1cm,inner sep=0mm,fill,circle}]
\draw[thick] (-2,-2) node {} -- (-1,-3) node {}-- (1,-1) node {};
\end{scope}
\node at (-0.5, -6) {\scri{\text{$Y^{-1}_{1,4}Y_{4,2}$}}};
\end{tikzpicture}
}
\raisebox{0.5cm}{
\begin{tikzpicture}[scale=.5]
\draw[help lines, color=gray!60, dashed] (-2.0,-5.0) grid (1.0, 1.0);
\draw[fill] (-1,1.5) circle (2pt)--(0,1.5) circle (2pt);
\draw[fill] (0,1.5) circle (2pt)--(1,1.5) circle (2pt);
\node at (-1, 2) {\scalebox{0.45}{\scri{$1$}}};
\node at (0, 2) {\scalebox{0.45}{\scri{$2$}}};
\node at (1, 2) {\scalebox{0.45}{\scri{$\overline{3}$}}};
\node at (-2.5, 1) {\scalebox{0.45}{\scri{$0$}}};
\node at (-2.5, 0) {\scalebox{0.45}{\scri{$1$}}};
\node at (-2.5, -1) {\scalebox{0.45}{\scri{$2$}}};
\node at (-2.5, -2) {\scalebox{0.45}{\scri{$3$}}};
\node at (-2.5, -3) {\scalebox{0.45}{\scri{$4$}}};
\node at (-2.5, -4) {\scalebox{0.45}{\scri{$5$}}};
\node at (-2.5, -5) {\scalebox{0.45}{\scri{$6$}}};
\begin{scope}[every node/.style={minimum size=.1cm,inner sep=0mm,fill,circle}]
\draw[thick] (-2,-2) node {} -- (-1,-3) node {}-- (0,-2) node {}-- (1,-3) node {};
\end{scope}
\node at (-0.5, -6) {\scri{\text{$Y^{-1}_{1,4}Y_{2,3}Y^{-1}_{4,4}$}}};
\end{tikzpicture}
}
\raisebox{0.5cm}{
\begin{tikzpicture}[scale=.5]
\draw[help lines, color=gray!60, dashed] (-2.0,-5.0) grid (1.0, 1.0);
\draw[fill] (-1,1.5) circle (2pt)--(0,1.5) circle (2pt);
\draw[fill] (0,1.5) circle (2pt)--(1,1.5) circle (2pt);
\node at (-1, 2) {\scalebox{0.45}{\scri{$1$}}};
\node at (0, 2) {\scalebox{0.45}{\scri{$2$}}};
\node at (1, 2) {\scalebox{0.45}{\scri{$\overline{3}$}}};
\node at (-2.5, 1) {\scalebox{0.45}{\scri{$0$}}};
\node at (-2.5, 0) {\scalebox{0.45}{\scri{$1$}}};
\node at (-2.5, -1) {\scalebox{0.45}{\scri{$2$}}};
\node at (-2.5, -2) {\scalebox{0.45}{\scri{$3$}}};
\node at (-2.5, -3) {\scalebox{0.45}{\scri{$4$}}};
\node at (-2.5, -4) {\scalebox{0.45}{\scri{$5$}}};
\node at (-2.5, -5) {\scalebox{0.45}{\scri{$6$}}};
\begin{scope}[every node/.style={minimum size=.1cm,inner sep=0mm,fill,circle}]
\draw[thick] (-2,-2) node {} -- (0,-4) node {}-- (1,-3)node {};
\end{scope}
\node at (-0.5, -6) {\scri{\text{$Y^{-1}_{2,5}Y_{3,4}$}}};
\end{tikzpicture}
}
\raisebox{0.5cm}{
\begin{tikzpicture}[scale=.5]
\draw[help lines, color=gray!60, dashed] (-2.0,-5.0) grid (1.0, 1.0);
\draw[fill] (-1,1.5) circle (2pt)--(0,1.5) circle (2pt);
\draw[fill] (0,1.5) circle (2pt)--(1,1.5) circle (2pt);
\node at (-1, 2) {\scalebox{0.45}{\scri{$1$}}};
\node at (0, 2) {\scalebox{0.45}{\scri{$2$}}};
\node at (1, 2) {\scalebox{0.45}{\scri{$\overline{3}$}}};
\node at (-2.5, 1) {\scalebox{0.45}{\scri{$0$}}};
\node at (-2.5, 0) {\scalebox{0.45}{\scri{$1$}}};
\node at (-2.5, -1) {\scalebox{0.45}{\scri{$2$}}};
\node at (-2.5, -2) {\scalebox{0.45}{\scri{$3$}}};
\node at (-2.5, -3) {\scalebox{0.45}{\scri{$4$}}};
\node at (-2.5, -4) {\scalebox{0.45}{\scri{$5$}}};
\node at (-2.5, -5) {\scalebox{0.45}{\scri{$6$}}};
\begin{scope}[every node/.style={minimum size=.1cm,inner sep=0mm,fill,circle}]
\draw[thick] (-2,-2) node {} -- (1,-5) node {};
\end{scope}
\node at (-0.5, -6) {\scri{\text{$Y^{-1}_{3,6}$}}};
\end{tikzpicture}
}
\end{align*}
\caption{Monomials or binomials associated to paths in $\mathscr{P}_{3,0}$.} \label{q-character of 3_0}
\end{figure}
\end{example}

\section*{Acknowledgement}

The work was supported by the National Natural Science Foundation of China (No. 12171213, 12001254) and Gansu Province Science Foundation for Youths (No. 22JR5RA534).

\end{document}